\documentclass[reqno,12pt]{amsart}
\usepackage{amssymb, amsmath,latexsym,amsfonts,amsbsy, amsthm}
\usepackage{color}
\usepackage{txfonts}
\usepackage{hyperref}

\usepackage{tikz}
\usetikzlibrary{arrows,chains,matrix,positioning,scopes}
\makeatletter
\tikzset{join/.code=\tikzset{after node path={%
\ifx\tikzchainprevious\pgfutil@empty\else(\tikzchainprevious)%
edge[every join]#1(\tikzchaincurrent)\fi}}}
\makeatother
\tikzset{>=stealth',every on chain/.append style={join},
         every join/.style={->}}
\tikzstyle{labeled}=[execute at begin node=$\scriptstyle, execute at end node=$]
\newtheorem{theorem}{Theorem}[section]

\newtheorem{lemma}[theorem]{Lemma}
\newtheorem{proposition}[theorem]{Proposition}

\newtheorem{remark}[theorem]{Remark}

\def\bR{\mathbb R}
\def\bN{\mathbb N}

\def\bT{\mathbb T}

\def\t{\tilde}
\def\q{\quad}

\def\th{\theta}

\def\dl{\delta}
\def\Dl{\Delta}
\def\lt{\left}

\def\rt{\right}

\def\i{\infty}
\def\e{\epsilon}

\def \ls{\lesssim}
\def\p{\partial}
\def\f{\frac}
\def\na{\nabla}
\def\al{\alpha}

\def\O{\Omega}
\def\o{\omega}

\def\s{\sqrt}
\def\bl{\boldsymbol}

\def\nn{\nonumber}
\def\be{\begin{equation}}
\def\ee{\end{equation}}
\def\bes{\begin{equation*}}
\def\ees{\end{equation*}}
\def\bali{\begin{aligned}}
\def\eali{\end{aligned}}
\def\beas{\begin{eqnarray*}}
\def\eeas{\end{eqnarray*}}

\def\pf{\noindent {\bf Proof. \hspace{2mm}}}
\def\ef{ \hfill $ \Box $ \vskip 3mm}

\allowdisplaybreaks[4]

\begin{document}
\title[Vanishing viscosity limit in an annulus]
{Vanishing viscosity limit of the stationary Navier-Stokes equations in an expanding annulus}

\author{Xinghong Pan}
\address[X. Pan]{School of Mathematics and Key Laboratory of MIIT, Nanjing University of Aeronautics and Astronautics, Nanjing 211106, China}
\email{xinghong\_87@nuaa.edu.cn}

\author{Chuankai Zhao}
\address[C. Zhao]{School of Mathematical Sciences, Capital Normal University, Beijing 100048, China}
\email{chuankaizhao@amss.ac.cn}

\author{Jianfeng Zhao}
\address[J. Zhao]{School of Mathematics and Information Science, Center for Applied Mathematics of Guangxi, Post-doctoral Research Station of the First-level Discipline in Mathematics, Guangxi University, Nanning, 530004, China}
\email{zhaojianfeng@amss.ac.cn}

%\date{\today}



%\thanks{X. Pan is supported by National Natural Science Foundation of China (No. 12471222, 12031006.)}

\maketitle

\begin{abstract}
  In this paper, we establish vanishing viscosity limit of the 2D stationary Navier-Stokes equations in an expanding annulus with a fixed width. The flow is subjected to a large perturbation of a rotation of rigid body at one boundary and stays at rest at another boundary. Due to this different boundary condition setting, the boundary layer effects are different near the interior and exterior boundaries, which result in different thickness of boundary layer and different leading order boundary layer equations. We will construct an approximate solution to this 2D Navier-Stokes equations by using higher order asymptotic approximation and show the validity of the boundary layer expansion. The leading order of the Euler solution is the Taylor-Couette flow $\left(Ar+B\f{1}{r}\right)e_\th$ for some suitable constant $A$ and $B$, which is determined by using the principle of the Prandtl-Batchelor theory. The perturbation of the tangential velocity becomes large as the annulus expands, which causes heavy difficulties when we show the convergence of the approximate solution. Variously weighted energy estimates up to third-order derivatives of the error solution are constructed to close the energy of the nonlinear system.

\end{abstract}

\numberwithin{equation}{section}

%\tableofcontents

\indent

\section{Introduction}

\indent

We consider the following stationary Navier-Stokes equations in an annulus  $B_{r_0+1}/B_{r_0}$,
\begin{equation}\label{ns}
\left \{
\begin {array}{ll}
\bl{u}^\e\cdot\na \bl{u}^\e+\na p^\e-\e^2\Dl \bl{u}^\e=0,\\ [5pt]
\na\cdot\bl{u}^\e=0,\\
\bl{u}^\e\cdot \boldsymbol{\tau}\big|_{\p B_{r_0+1}}=\lt(\o +\delta f(\th)\rt)(r_0+1),\q\ \bl{u}^\e\cdot \boldsymbol{n}\big|_{\p B_{r_0+1}}=0,\\
\bl{u}^\e\big|_{\p B_{r_0}}=0,
\end{array}
\right.
\end{equation}
where $(x,y)\in \bR^2$ and $B_r$ denotes the disc centered at the origin in $\bR^2$ with radius $r$. $\bl{u}^\e$ is the velocity and $p^\epsilon$ is the pressure.  $\epsilon^2>0$ is the viscosity coefficient, which is reciprocal to Reynolds number. $\o>0$ is a constant, $\dl$ is a  small number, and $f(\th)$ is a smooth function, defined on the $\p B_{r_0+1}$, where $\th\in [0,2\pi)$ is the polar angle.  $\boldsymbol{n}$ and $\boldsymbol{\tau}$ are the unite outer normal vector and the tangential vector on the outer boundary $\p B_{r_0+1}$. ($\boldsymbol{n}$, $\boldsymbol{\tau}$) have the same orientation with the Euclidean orthogonal basis $(\bl{e}_1,\bl{e}_2)$. Since we consider that the annulus is expanding, the radius $r_0$ is set to be large, {\it i.e.}, $r_0\geq 100$.
\begin{remark}
The boundary condition for the tangential component of the velocity at the exterior boundary shows that its angular velocity is a small perturbation of some non-zero constant. However, the perturbation of the linear velocity at the exterior boundary is large as the radius $r_0$ expands.
\end{remark}

 Formally, as $\epsilon\rightarrow 0$, we obtain the following 2D steady Euler equations for $\bl{u}^e$

\begin{equation}\label{equeuler}
\left \{
\begin {array}{ll}
\bl{u}^e\cdot\na \bl{u}^e+\na p^e=0,\\ [5pt]
\na\cdot\bl{u}^e=0,\\
\bl{u}^e\cdot\bl{n}\big|_{\p B_{r_0}}=\bl{u}^e\cdot\bl{n}\big|_{\p B_{r_0+1}}=0.
\end{array}
\right.
\end{equation}

We will show the existence of solution $\bl{u}^\e$ to (\ref{ns}) which converges to a solution of the steady Euler equations (\ref{equeuler}).

Since our referenced domain is an annulus, it is more convenient to reformulate system \eqref{ns} and system \eqref{equeuler} in polar coordinates. In polar coordinates $(r,\th)$, $x=r \cos \th,\ y=r \sin \th$, the polar orthogonal basis
\bes
\bl{e}_r=(\cos\th,\sin\th),\q \bl{e}_\th=(-\sin\th,\cos\th)
\ees
have the same orientation with $(\bl{e}_1,\bl{e}_2)$. We denote the vector function $\bl{u}$ in the polar coordinate by
\bes
\bl{u}=(v,u)=v\bl{e}_r+u\bl{e}_\th.
\ees
Then system \eqref{ns} is reformulated in polar coordinate into the following
\be\label{nspolar}
\left\{
\begin{array}{l}
u^{\epsilon} u_\theta^{\epsilon}+r v^{\epsilon} u_r^{\epsilon}+u^{\epsilon} v^{\epsilon}+p_\theta^{\epsilon}-\epsilon^2\left(\frac{u_{\theta \theta}^{\epsilon}}{r}+r u_{r r}^{\epsilon}+u_r^{\epsilon}+\frac{2}{r} v_\theta^{\epsilon}-\frac{u^{\epsilon}}{r}\right)=0, \\
u^{\epsilon} v_\theta^{\epsilon}+r v^{\epsilon} v_r^{\epsilon}-\left(u^{\epsilon}\right)^2+r p_r^{\epsilon}-\epsilon^2\left(\frac{v_{\theta \theta}^{\epsilon}}{r}+r v_{r r}^{\epsilon}+v_r^{\epsilon}-\frac{2}{r} u_\theta^{\epsilon}-\frac{v^{\epsilon}}{r}\right)=0, \\
u_\theta^{\epsilon}+\left(r v^{\epsilon}\right)_r=0, \\
\lt(u^{\epsilon},v^\e\rt)(\theta, r_0+1)=\lt((\o+\delta f(\theta))(r_0+1),0\rt), \\
\lt(u^{\epsilon},v^\e\rt)(\theta, r_0)=0,
\end{array}
\right.
\ee
where $\bl{u}^\e=(v^\e,u^\e)$. Also the Euler system \eqref{equeuler} in polar coordinates is reformulated into
\bes
\left\{
\begin{array}{l}
u^{e} u_\theta^{e}+r v^{e} u_r^{e}+u^{e} v^{e}+p_\theta^{e}=0, \\
u^{e} v_\theta^{e}+r v^{e} v_r^{e}-\left(u^{e}\right)^2+r p_r^{e}=0, \\
u_\theta^{e}+\left(r v^{e}\right)_r=0, \\
v^e(\theta, r_0)=v^e(\theta, r_0+1)=0,
\end{array}
\right.
\ees
where $\bl{u}^e=(v^e,u^e)$. In the current work, we consider the limiting Euler solution has no stagnation point $(|\bl{u}^e|\neq 0)$ in the interior of the annulus. From \cite[Theorem 1.3]{WangZ:2023ARXIV}, the Euler solution must to be a tangential shear flow $u^e(r)\bl{e}_\th$. Moreover, from the Prandtl-Batchelor theory, see for example \cite[Appendix]{FeiGLT:2023CMP}, we can obtain that
\bes
\Dl u^e(r)=0,
\ees
which indicates that the only possible Euler flow is the following Taylor-Couette flow
\bes
\bl{u}^e=\lt(Ar+B\f{1}{r}\rt)\bl{e}_\th.
\ees

The constants $A,B$ are deduced by the Batchelor-Wood formula. One can refer to \cite{DrivasIN:2023ARXIV} for its detailed derivation. We will also give an explanation later on. See Lemma \ref{lembw} and Section \eqref{sec431}. The principle of choosing $A$ and $B$ is

{\bf The integral average of the square of the tangential components to the solutions of the Navier-Stokes equations and the limiting Euler equations on the boundary must be the same.}

This indicates that
\be \label{BatchelorWood}
\lt\{
\begin{aligned}
&\f{1}{2\pi}\int^{2\pi }_0 \lt(Ar+B\f{1}{r}\rt)^2\Big|_{r=r_0}d\th=0,\\
&\f{1}{2\pi}\int^{2\pi }_0 \lt(Ar+B\f{1}{r}\rt)^2\Big|_{r=r_0+1}d\th=\f{1}{2\pi}\int^{2\pi }_0 \lt(\o+\dl f(\th)\rt)^2(r _0+1)^2d\th:=\t{\o}^2(r_0+1)^2.
\end{aligned}
\rt.
\ee
From \eqref{BatchelorWood}, we see that
\be\label{defAB}
A=\f{(r_0+1)^2\t{\o}}{2r_0+1},\q B=-Ar^2_0=-\f{(r_0+1)^2r^2_0\t{\o}}{2r_0+1}.
\ee
Next, we introduce the leading order steady boundary layer equations near the exterior boundary $r=r_0+1$. Denote the leading order term of the outer boundary layer expansions by $(u_p^{(0)},v_p^{(1)})$, then
\be\label{boundarylayeruppermain}
\left\{
\begin {array}{ll}
\big(\t{\o}(r_0+1)+u_p^{(0)}\big)\partial_\th u_p^{(0)}+\big( v_p^{(1)}- v_p^{(1)}(\th,0)\big)(r_0+1)\partial_\zeta u_p^{(0)}-(r_0+1)\partial^2_{\zeta}u_p^{(0)}=0,\\[5pt]
\partial_\th u_p^{(0)}+(r_0+1)\partial_\zeta v_p^{(1)}=0,\\[5pt]
u_p^{(0)}(\th,\zeta)=u_p^{(0)}(\th+2\pi,\zeta),\ v_p^{(1)}(\th,\zeta)=v_p^{(1)}(\th+2\pi,\zeta),\\[5pt]
u_p^{(0)}\big|_{\zeta=0}=(r_0+1)\lt(\o+\dl f(\th)-\t{\o}\rt)\thickapprox \dl,\\[5pt]
\lim\limits_{\zeta\rightarrow -\infty}u_p^{(0)}=\lim\limits_{\zeta\rightarrow -\infty}v_p^{(1)}=0,
\end{array}
\right.
\ee
where $\zeta:=\f{r-r_0-1}{\e}$ is the scaled variable. This boundary layer equations \eqref{boundarylayeruppermain} is derived by the procedure of matched asymptotic expansion and its well-posedness is stated in Proposition \ref{propdcu0}.

Define the domain $$\O:=\bT\times[r_0,r_0+1],$$ where $\bT=[0,2\pi)$ is the $2\pi$ periodic interval. The following is the main result of our paper.
\begin{theorem}\label{thmain}
Assume that $f(\th)$ is a smooth $2\pi$-periodic function. Then there exist two constants $\e_0$ and $\dl_0$ such that for any $\e\in(0,\e_0]$ and $\dl\in(0,\dl_0]$, the system \eqref{nspolar} has a solution $(u^\e,v^\e)$ satisfying
\be\label{estimain1}
\bali
\lt\|u^\e(\th,r)-u^e(r)-u_p^{(0)}(\th,\zeta)\rt\|_{L^\i(\O)}+\lt\|v^\e(\th,r)\rt\|_{L^\i(\O)}\leq C\e r_0,
\eali
\ee
where $u^e(r)=Ar+\f{B}{r}$, $C$ is independent of $\e$ and $\dl$, and $u_p^{(0)}(\th,\zeta)$ is the solution of \eqref{boundarylayeruppermain}. \qed
\end{theorem}

 \begin{remark}
  If we change the boundary condition between the inner and outer boundary, i.e. the flow is subjected to a perturbation of the rotation of rigid body at the inner boundary and stays at rest at the outer boundary, we can obtain a similar result as Theorem \ref{thmain} when $r_0$ is suitably away from the origin. The proof is completely the same as that for Theorem \ref{thmain}. The only difference is presented in Remark \ref{remboundaryc}.
\end{remark}

\begin{remark}
   The authors in \cite{FeiGLT:2024ADV} also considered  the  vanishing viscosity limit of the 2D stationary Navier-Stokes equations in the annulus. The differences between theirs and ours
  are mainly stays in the following two aspects which cause many more difficulties in analysis:

  One is the degeneration of the inner boundary condition. This degeneration results in different boundary expansions at the outer and inner boundary. The thickness of the boundary layer will be different at the outer and inner boundary and more subtle asymptotic expansion will be constructed.

  The other is the expansion of the annulus in our paper, which means that $r_0$ in our setting is a large number which can go to infinity. This causes heavily difficult problems when we show existence of the error equations after the approximate solution is constructed. More dedicate and complicated estimates are needed to close the  energy functionals. Our result gives a quantitative dependence on the radius of the error between the solutions of the Navier-Stokes equations and the Euler equations when the annulus expands.
\end{remark}

The rigorous justification of vanishing viscosity limit from the Navier-Stokes equations to the Euler equations, particularly under the non-slip boundary condition, remains a significant and challenging problem in fluid dynamics, seeing  \cite{OleinikS:1999,SchlichtingG:2017}. This issue arises due to the mismatch between the boundary conditions of the viscous Navier-Stokes equations and the inviscid Euler equations. A classical approach to address this discrepancy is through the Prandtl boundary layer theory, first introduced by Prandtl in 1904 \cite{Prandtl:1904}, which provides a formal asymptotic expansion to describe the boundary layer behavior.

In recent decade, there are some important results for the steady vanishing viscosity limit problem. Guo-Nguyen \cite{GuoN:2017ANNPDE} considered the Prandtl boundary expansion of the steady Navier-Stokes equations in a moving plate in $(0,L)\times \bR_+$, where they developed methodologies based on positivity estimates. And then Guo-Iyer \cite{GuoI:2023CPAM} showed the Prandtl boundary layer expansion for the motionless plate with the Euler flow being a shear flow. Gao-Zhang advanced this progress in \cite{GaoZ:2023SCM} by introducing a good unknown-the ratio of the stream function to the approximate solution-enabling them to validate Prandtl boundary layer asymptotic expansions across arbitrarily large horizontal domains. In subsequent work \cite{GaoZ:2024JMP}, they employed innovative weighted forms and domain decomposition techniques to estimate the material derivative of the stream function under non-shear Euler flows, thereby extending the applicability of boundary layer asymptotic expansions. Recently, Gao-Xin demonstrated in \cite{GX2023} that in an infinitely long two-dimensional converging channel, if the boundary curve satisfies a decreasing curvature condition, then point-source Navier-Stokes flows with negative mass flux adhere to Prandtl boundary layer theory as viscosity diminishes. Additionally, relevant work can be found in references \cite{Iyer:2017ARMA,Iyer:2019SIAM,Iyer:2019Peking,Gerard:2019ARMA}, which provide further insights and discussions on this topic. Furthermore, we mention some results in \cite{Xin:2004ADV,Xin:2024SIAM,WangZ:2021AIHP,Iyer:2020ARMA,WangZ:2023MATHANN,Jia:2024ARXIV,Gao:2025ARXIV,GuoWZ:2023ANNPDE,DM:2019,ShenWZ:2021} on the well-posedness, asymptotic behavior, dynamic stability, and boundary layer separation of the stationary Prandtl system.

In the study of the steady inviscid limit problem, the identification of an appropriate Euler flow remains a fundamental challenge, particularly for general domains where no precise principle for selection exists despite notable advancements. However, in the context of closed regions, significant insights have been derived from the Prandtl-Batchelor theory. Prandtl, in his seminal work \cite{Prandtl:1904}, demonstrated that in simply-connected domains, the vorticity of the steady Euler flow must be constant within closed streamlines for viscous incompressible fluids. This pivotal discovery was later reaffirmed and extended by Batchelor and Wood in \cite{Batchelor:1956} and \cite{Wood:1957JFM}, respectively, who further elucidated that the exact value of this vorticity constant is determined by the boundary conditions. This kind of result is now referred to as Prandtl-Batchelor theory in the literature. See some related researches on the Prandtl-Batchelor theory in \cite{Kim-thesis, Kim:1998SIAM, KimC:2001SIAM, WV:2007} and references therein. Building upon this theoretical framework, Fei-Gao-Lin-Tao in \cite{FeiGLT:2023CMP, FeiGLT:2024ADV} considered the Prandtl boundary expansion in a disk and an annulus respectively, where the boundary conditions for the swirl component of the velocity is a small perturbation of the rigid-rotation. Furthermore, Fei-Pan-Zhao \cite{FeiPZ:2023ARXIV} extended this analysis to a periodic strip with a motionless boundary on one side. Very recently, Chen-Fei-Lin-Zhao \cite{ChenFLZ:24arxiv} and Pan-Zhao \cite{PanZ:2025ARXIV} have addressed the vanishing viscosity limit in domains of an interior disk with a center vortex and  an exterior disc,  respectively, which further advanced our understanding of inviscid limit phenomena in complex geometries.

In the case of the non-stationary solution, existing rigorous mathematical analyses of the Prandtl boundary layer expansion have been established within the analytic framework \cite{FTZ2018, KVW2020, M2014, NN2018, SC1998-1, SC1998-2, WWZ2017} and references therein. However, validation of the boundary layer expansion in the Sobolev framework remains an open  problem in the non-stationary context. Notably, several instability results concerning the Prandtl expansion of shear flow type in Sobolev spaces have been presented in \cite{G2000, GGN2016, GN2019}. The nonlinear asymptotic stability and transition threshold analysis near Taylor-Couette flow in an annulus for the Navier-Stokes equations in \cite{AnHL:2024CMP} was mentioned here for its relation with our works. We also mention some results on the vanishing viscosity limit of the stationary and non-stationary Navier-Stokes equations with Navier-slip boundary in \cite{MasmoudiR:2012ARMA,WangXZ:2012JMFM,BedrossianHIW:2023ARXIV,BedrossianHIW:2024ARXIV,PanZ:2025preprint} and references therein.

Next we present some ideas of proving the main theorems. In order to prove Theorem \ref{thmain}, we firstly construct an approximate solution to the system \eqref{ns} by the method of asymptotic expansion, which share the same boundary condition with \eqref{ns},  and then estimate the error around the constructed approximate solution. A brief view of the strategy of the proof will be given as follows.

{\bf\noindent Construction of an approximation solution}
We first build an approximate solution $(u^a,v^a)$ by the asymptotic expansion, which contain the Euler part $(u^a_e, v^a_e)$, the outer boundary layer part $(u^a_p,v^a_p)$ and inner boundary layer part $(\hat{u}^a_p,\hat{v}^a_p)$. The Euler part of the approximate solution satisfies the following estimates
\begin{align}
&\lt\|\p^j_\th\p^k_r\lt(u^a_e-\lt(Ar+\f{B}{r}\rt)\rt)\rt\|_\infty\leq C_{j,k}\epsilon r^{1-k}_0(\dl+\epsilon),\q \|\p^j_\th\p^k_rv^a_e\|_\infty\leq C_{j,k} r^{-k}_0\epsilon(\dl+\epsilon). \label{eulerapprox}
\end{align}
While for the boundary outer layer part, we have for any $j,k,\ell\in \{0\}\cup\bN$,
\begin{align}
\|\zeta^\ell\p^{j}_\th\partial_\zeta^k u_p^a\|_\infty\leq C_{j,k,\ell} r_0(\dl+\epsilon), \ \|\zeta^\ell\p^{j}_\th\partial_\zeta^k v_p^a\|_\infty\leq C_{j,k,\ell} \epsilon(\dl+\epsilon), \label{prandtlapprox1}
\end{align}
and for the inner boundary layer parts,
\begin{align}
\|\eta^\ell\p^{j}_\th\partial_\eta^k \hat{u}_p^a\|_\infty\leq C_{j,k,\ell} r_0\e(\dl+\epsilon^{2/3}),  \ \|\eta^\ell\p^{j}_x\partial_\eta^k v_p^a\|_\infty\leq C_{j,k,\ell} \epsilon^{1+2/3}(\dl+\epsilon^{2/3}),\label{prandtlapprox2}
\end{align}
where $\zeta:=\f{r-r_0-1}{\e}$ is the scaled variable near the outer boundary and $\eta:=\f{r-r_0}{\e^{2/3}}$ is the scaled variable near the inner boundary. Estimates \eqref{eulerapprox}, \eqref{prandtlapprox1} and \eqref{prandtlapprox2} will be frequently used in our next error estimates. The construction and the asymptotic behavior of the approximate solution are rather long and very technical. In order to keep fluent and show clearly the proof line of our paper, we postpone this construction to Section \ref{secappro}.

{\bf\noindent ${H}^1$ stability estimate of the linearized error Navier-Stokes equations}

In the polar coordinate system \eqref{nspolar}, denote the error function by
\bes
u:=u^\epsilon-u^a,\ v:=v^\epsilon-v^a,\ p:=p^\epsilon-p^a,
\ees
where $p^a$ is the constructed pressure. Then the error function will satisfy the following linear forced Navier-Stokes equations

\begin{align}\label{errorequation}
\left\{
\begin{array}{ll}
-\epsilon^2\left(\frac{u_{\theta \theta}}{r}+r u_{r r}+u_r+\frac{2}{r} v_\theta-\frac{u}{r}\right)+\p_\th p +S_u=R_u, &(\th,r)\in \O,\\[5pt]
-\epsilon^2\left(\frac{v_{\theta \theta}}{r}+r v_{r r}+v_r-\frac{2}{r} u_\theta-\frac{v}{r}\right)+r\p_rp+S_v=R_v,&(\th,r)\in \O,\\[5pt]
\p_\th u+\p_r(rv)=0,  &(\th,r)\in \O,\\[5pt]
( u^a,v^a)(\th,r)=( u^a,v^a)(\th+2\pi,r), &(\th,r)\in \O, \\[5pt]
u(\th,1)=v(\th,1)=u(\th,+\i)=v(\th,+\i)=0, &\th\in \bT,
 \end{array}
\right.
\end{align}
where
\bes
\begin{aligned}
S_u:&=u^a \p_\th u +v^ar\p_ru+u \p_\th u^a+vr\p_ru^a+uv^a+vu^a,\\
S_v:&=u^a \p_\th v +v^ar\p_rv+u \p_\th v^a+v r\p_r v^a-2u u^a,
\end{aligned}
\ees
 along with $R_u$ and $R_v$ being the remainders.

We will first give a $\dot{H}^1$ stability estimate for the above linearized Navier-Stokes equations as follows.
\be\label{linearxx4}
\e^4\int_\O (u^2_r+v^2_r)+ (u^2_\th+v^2_\th)(r-r_0)+r^{-2}_0 (u^2_\th+v^2_\th)\ls\cdots.
\ee
where $``\cdots"$ represents the remainders terms. This stability estimate contain three parts:

One is the weighted positive estimate from the transport term $S_u$ and $S_v$. By multiplying  \eqref{errorequation}$_{1,2}$ with $u_\th$ and $v_\th$, then integrating the resulting equation on $\O$, we can achieve the following estimate with the help of the estimate for the approximation solution
\be\label{linear1}
\bali
&\int (u^2_\th+v^2_\th)(r-r_0)\ls \e^2\dl\int u^2_{0,r}+(u^2_{\neq,r}+v^2_r)(r-r_0) \cdots.
\eali
\ee
 Here $u_0$ and $u_{\neq}$ represent the zero and non-zero frequencies of $u$ in $\th$ direction, defined by
$$u_0(r):=\f{1}{2\pi}\int^{2\pi}_{0}u(\th,r)d\th:=\fint^{2\pi}_{0}u(\th,r)d\th,\q u_{\neq}=u-u_0. $$

The second one is the energy estimate for the first-order normal derivatives which comes from the diffusion terms. The energy estimate is decomposed into the zero frequency and non-zero frequency estimate. The zero frequency estimates can be directly from \eqref{errorequation}$_1$ by integration on $\th$. While for the non-zero frequency estimates, we need to perform a weighted estimates.
However, it is hard to deal with the pressure for system \eqref{errorequation} when performing the weighted estimates. Therefore, we first utilize the incompressibility to eliminate the pressure and reformulate the system \eqref{errorequation} in the form of vorticity function.

Define the vorticity $\o=\f{v_\th-(ru)_r}{r}$. Subtracting $\p_r$\eqref{errorequation}$_1$ from $\f{1}{r}\p_\th$\eqref{errorequation}$_2$, we can obtain that

\begin{align}
&-\epsilon^2 r\Dl \o+\f{1}{r}\p_\th S_v-\p_rS_u=\f{1}{r}\p_\th R_v-\p_rR_u. \label{errorequation1}
\end{align}

The zero frequency and non-zero frequency energy estimates will be presented as the following, for a fixed small $\al\in(0,0.01]$
\be\label{linearx1}
\e^2\int u^2_{0,r} +(r-r_0)^\al(u^2_{\neq,r}+v^2_r)\ls\al^{-6}\int(r-r_0) (u^2_\th+v^2_\th) \cdots.
\ee
Actually, estimates \eqref{linear1} and \eqref{linearx1} are closed for the weighted $\dot{H}^1$ estimates, namely, we have
\be\label{linearx2}
\int (r-r_0)(u^2_{\th}+v^2_\th)+\e^2\int u^2_{0,r} +(r-r_0)^\al(u^2_{\neq,r}+v^2_r)\ls\cdots.
\ee
Based on the above weighted estimates, the third one is the non-weighted energy estimates for non-zero frequency of the first-order normal derivatives, which will be achieved by  multiplying  \eqref{errorequation1} with $\e^2 u$ and $\e^2 \Phi_{\neq}$, where $\Phi_{\neq }$ is the non-zero frequency of the stream function. Then integrating the resulting equation on $\O$ indicates that
\be\label{linearxx3}
\e^4\int (u^2_{\neq,r}+v^2_r)+r^{-2}_0(u^2_{\th}+v^2_\th)\ls \int (r-r_0)(u^2_{\th}+v^2_\th)+\cdots.
\ee
Combining \eqref{linearxx3} and \eqref{linearx2}, we can obtain the linear stability $\dot{H}^1$ estimate in \eqref{linearxx4}.

{\bf\noindent ${H}^3$ estimate of the linearized error Navier-Stokes equations}

In order to obtain $L^\i$ estimate and close the nonlinear terms for the error function $(u,v)$, we still need $\dot{H}^3$ derivatives, which is achieved by the following two steps

1. {\it Tangential derivatives' estimates.} Similar  as \eqref{linearxx3}, multiplying  \eqref{errorequation1} with $-\e^2(rv_{\th\th \th})$ and then integrating the resulting equation indicate that
 \be\label{linearx4}
 \e^4\|(u_{\th\th r},v_{\th\th r})\|^2_{L^2(\O)}+ \e^4\|(u_{\th\th \th},v_{\th\th \th})\s{r-r_0}\|^2_{L^2}+ \e^4r^{-2}_0\|(u_{\th\th \th},v_{\th\th \th})\|^2_{L^2}\ls \cdots.
 \ee
Unfortunately, the tangential derivatives estimates from the energy integral can not induce the estimate for $u_{rr}$, $u_{rr\th}$ and $u_{rrr}$. The estimates for $u_{rr}$ can only be obtained from the equation \eqref{errorequation}$_1$.

2. {\it Normal derivatives' estimates.} Using Bogovski mapping Lemma to obtain the $L^2$ estimate of $p_\th$ and $p_{\th\th}$, which results in the following estimate
 \be\label{linearx5}
 \e^8\|(u_{rr},v_{rr})\|^2_{L^2(\O)}+ \e^8\|(u_{rr\th},v_{rr\th})\|^2_{L^2(\O)}\ls+\cdots.
 \ee
 At last directly from the vorticity equation \eqref{errorequation1}, we can control the pure normal derivatives,
 \be \label{linearxx6}
 \e^{10} \|(u_{rrr},v_{rrr})\|^2_{L^2(\O)}\ls+\cdots.
 \ee
Then estimates \eqref{linearxx4}, \eqref{linearx4}, \eqref{linearx5} and \eqref{linearxx6} together can close the $H^3$ linear weighted estimates for the nonlinear error equation.

{\bf\noindent Contraction mapping of the nonlinear error Navier-Stokes equations}

 By using the deduced a priori estimates up to third-order derivatives, the standard contraction mapping implies the existence of $L^\i$ solution to the system \eqref{errorequation}.

Throughout the paper, $C_{a,b,c,...}$ denotes a positive constant depending on $a,\,b,\, c,\,...$ which may vary from line to line. We also apply $A\lesssim_{a,b,c,\cdots} B$ to denote $A\leq C_{a,b,c,...}B$, while $A \approx_{a,b,c,...}B$ means $A\leq C_{a,b,c,...}B$ and $B\leq C_{a,b,c,...}A$. For a norm $\|\cdot\|$, we use $\|(f,g,\cdots)\|$ to denote $\|f\|+\|g\|+\cdots$.  For a function $f(\th,r)$ and $1\leq p,q \leq +\i$,  define
\bes
\|f\|_{L^p_\th L^q_r}:=\lt(\int^{+\i}_1 \lt(\int_{\bT} |f|^p d\th\rt)^{q/p}dr \rt)^{1/q}.
\ees
If $p=q$, we simply write it as $\|f\|_{L^p}$ or $\|f\|_{p}$. If $p=q=2$, we use $\|f\|$ to denote the $\|f\|_{L^2}$.

Our paper is organized as follows. In section \ref{sec2}, we first present the equations satisfied by the approximate solution and its asymptotic behavior, and derive some useful lemmas for later use. In section \ref{sec3}, based on the approximate solution's property, we derive the error equations and establish the linear stability $H^1$ estimates of the error solution. In section \ref{sec4}, we derive the tangential and normal derivatives' estimates up to third order of the error solutions. Then, in Section \ref{sec5}, we estimate the nonlinear terms, prove the existence of the error solution by contraction mapping theorem, which finishes the proof of Theorem \ref{thmain}. At last, in Section \ref{secappro}, we come back to give the construction of the approximate solution.

\section{Preliminary}\label{sec2}

\indent
In this section, we derive the error equation and some useful lemmas for later use. The properties and some detailed estimates for the approximate solution will be given. Also we will present a Hardy inequality for our domain.
First we present the following approximate solution $(u^a,v^a,p^a)$  which will be constructed in Section \ref{secappro}:

\begin{eqnarray}\label{app equationd}
\left\{
\begin{array}{ll}
u^{a} u_\theta^{a}+r v^{a} u_r^{a}+u^{a} v^{a}+p_\theta^{a}-\epsilon^2\left(\frac{u_{\theta \theta}^{a}}{r}+r u_{r r}^{a}+u_r^{a}+\frac{2}{r} v_\theta^{a}-\frac{u^{a}}{r}\right)=R_u^a, &(\th,r)\in \O,\\[5pt]
u^{a} v_\theta^{a}+r v^{a} v_r^{a}-\left(u^{a}\right)^2+r p_r^{a}-\epsilon^2\left(\frac{v_{\theta \theta}^{a}}{r}+r v_{r r}^{a}+v_r^{a}-\frac{2}{r} u_\theta^{a}-\frac{v^{a}}{r}\right)=R_v^a, &(\th,r)\in \O,\\[5pt]
 \p_\th u^a+\p_r(rv^a)=0,  &(\th,r)\in \O, \\[5pt]
( u^a,v^a)(\th,r)=( u^a,v^a)(\th+2\pi,r), &(\th,r)\in \O,\\[5pt]
u^a(\th, r_0+1)=(\o+\dl f(\th))(r_0+1), \ v^a(\th,r_0+1)=0, &\th\in \bT, \\[5pt]
 u^a(\th,r_0)=0,\ v^a(\th,r_0)=0, & \th\in \bT.
\end{array}
\right.
\end{eqnarray}
where the forced terms $(R^a_u,R^a_v)$ satisfy the following estimates
 \begin{align}
 \|\p^i_\th\p^j_r(R_u^a, R_v^a)\|_2\leq C_{i,j}\epsilon^{16-j} r^2_0.\label{remainderesti}
 \end{align}
Besides we have the following lemma which will be shown in Section \ref{secappro} to state the asymptotic behavior of the approximate solution $(u^a,v^a)$ near and away from the boundary of the strip $\O$.
\begin{lemma}\label{lemdetail}
The constructed approximate solution $(u^a,v^a)$ is decomposed as
\be\label{uaconstruct}
u^a:=u^a_e+\chi(r)u^a_p+\hat{\chi}(r)\hat{u}^a_p+\e^{16} h(\th,r),\q v^a:=v^a_e+\chi(r)v^a_p+\hat{\chi}(r)\hat{v}^a_p,
\ee
where $\chi(r)\in C^\i_c((-\i,r_0+1])$ is a function satisfying
\bes
\chi(r)=\lt\{
\bali
&0,\q r\in \lt[-\i,r_0+0.25\rt],\\
&1,\q r\geq r_0+0.75,
\eali
\rt.
\ees
and $\hat{\chi}(r)=1-\chi(r)$. The Euler part $(u^a_e,v^a_e)$ satisfies,
\begin{align}
&\|\p^j_\th\p^{k}_r(u^a_e-u^e(r))\|_\infty\leq C_{j,k}r^{1-k}_0\epsilon(\dl+\epsilon),\q \|\p^j_\th\p^k_rv^a_e\|_\infty\leq C_{j,k}r^{-k}_0 \epsilon (\dl+\epsilon). \label{eulerapproxd}
\end{align}
The outer boundary layer part $(u^a_p,v^a_p)$ satisfies
\begin{align}
\|\zeta^\ell\p^{j}_\th\partial_\zeta^k u_p^a\|_\infty\leq C_{j,k,\ell} r_0(\dl+\epsilon), \ \|\zeta^\ell\p^{j}_\th\partial_\zeta^k {v}_p^a\|_\infty\leq C_{j,k,\ell} \epsilon(\dl+\epsilon), \label{prandtlapprox1d}
\end{align}
and the inner boundary layer $(\hat{u}^a_p,\hat{v}^a_p)$ satisfies
\begin{align}
\|\eta^\ell\p^{j}_\th\partial_\eta^k \hat{u}_p^a\|_\infty\leq C_{j,k,\ell} r_0\e(\dl+\epsilon^{2/3}),  \ \|\eta^\ell\p^{j}_\th\partial_\eta^k \hat{v}_p^a\|_\infty\leq C_{j,k,\ell} \epsilon^{1+2/3}(\dl+\epsilon^{2/3}).\label{prandtlapprox2d}
\end{align}
While $h(\th,r)$ satisfies
\bes
h(\th,r_0)=h(\th,r_0+1)=0,\q \|\p^j_\th\p^k_r h(\th,r)\|_{L^\i}\leq C_{j,k}.
\ees
\qed
\end{lemma}
Set the error function by
\bes
u:=u^\epsilon-u^a,\ v:=v^\epsilon-v^a,\ p:=p^\epsilon-p^a,
\ees
then there holds
\begin{align}\label{errorequ}
\left\{
\begin{array}{ll}
-\epsilon^2\left(\frac{u_{\theta \theta}}{r}+r u_{r r}+u_r+\frac{2}{r} v_\theta-\frac{u}{r}\right)+\p_\th p +S_u=R_u, &(\th,r)\in \O,\\[5pt]
-\epsilon^2\left(\frac{v_{\theta \theta}}{r}+r v_{r r}+v_r-\frac{2}{r} u_\theta-\frac{v}{r}\right)+r\p_rp+S_v=R_v,&(\th,r)\in \O,\\[5pt]
\p_\th u+\p_r(rv)=0,  &(\th,r)\in \O,\\[5pt]
( u^a,v^a)(\th,r)=( u^a,v^a)(\th+2\pi,r), &(\th,r)\in \O, \\[5pt]
u(\th,r_0)=v(\th,r_0)=u(\th,r_0+1)=v(\th,r_0+1)=0, &\th\in \bT,
 \end{array}
\right.
\end{align}
where
\begin{align}
S_u:&=u^a \p_\th u +v^ar\p_ru+u \p_\th u^a+vr\p_ru^a+uv^a+vu^a, \label{sudef}\\[5pt]
S_v:&=u^a \p_\th v +v^ar\p_rv+u \p_\th v^a+v r\p_r v^a-2u u^a, \label{svdef}
\end{align}
and the remainders
\bes
R_u:=R_u^a-(u\p_\th+v r\p_r)u-uv,\ R_v:=R_v^a-(u\p_\th+v r\p_r)v+u^2.
\ees
In order to avoid the trouble caused by the pressure, sometimes it is more convenient to reformulate our system in the vorticity form and eliminate the pressure.

In polar coordinates, the vorticity $\o$ is defined by
\bes
\o=\f{v_\th-(ru)_r}{r}
\ees
From the incompressibility and $\int^{2\pi}_0 rvd\th=0$, there exists a $\th$-periodic stream function $\Phi$ such that
\bes
\Phi_r=u,\q \Phi_\th=-rv,
\ees
Actually we can take
\bes
\Phi(\th,r):=\int^r_{r_0+1} u(\th,\bar{r})d\bar{r},
\ees
which satisfies $\Phi(\th,r_0+1)=0$. Direct calculation implies that $-\Dl\Phi=\o$.

Subtracting $\p_r$\eqref{errorequ}$_1$ from $\f{1}{r}\p_\th$\eqref{errorequ}$_2$, we can obtain that
\begin{align}
&-\epsilon^2r\Dl\o+\f{1}{r}\p_\th S_v-\p_rS_u=\f{1}{r}\p_\th R_v-\p_rR_u. \label{errorequcut1}
\end{align}
Define the zero frequency and non-zero frequency of a $\th$-periodic function $f(\th,r)$ by
 \bes
 f_0=\f{1}{2\pi}\int^{2\pi}_0 f(\th,r)d\th:=\fint^{2\pi}_0 f(\th,r)d\th \q f_{\neq}=f-f_0.
 \ees
 The boundary condition of $v|_{r=r_0}=0$ and the incompressibility indicate that
\bes
\p_\th\Phi(\th,r_0)=\int^{r_0}_{r_0+1} u_\th(\th,r)dr=-\int^{r_0}_{r_0+1} (rv)_rdr=rv(\th,r)\big|^{r_0+1}_{r_0}=0,
\ees
which indicates that $\Phi(\th,r_0)=constant$. We have the following boundary condition for $\Phi$,
\bes
\Phi|_{r=r_0+1}=0, \q \p_\th \Phi|_{r=r_0 \text{ and } r_0+1}=0,\q \p_r\Phi|_{r=r_0\text{ and }r_0+1}=0.
\ees
and noting that $\Phi_{\neq}(\th,r_0)=\int^{r_0}_{r_0+1} (u-u_0)dr=0$, we have the following boundary condition for $\Phi_{\neq}$,
\bes
(\Phi_{\neq},\p_r\Phi_{\neq}, \p_\th\Phi_{\neq})|_{r=r_0 \text{ and } r_0+1}=0.
\ees

Next, we present a Lemma to give more detailed and useful estimates of the approximate solution from Lemma \ref{lemdetail}.
\begin{lemma} \label{lemdetail1}
By using the boundary conditions  in \eqref{app equationd} and the asymptotic behaviors in \eqref{eulerapproxd}, \eqref{prandtlapprox1d}, \eqref{prandtlapprox2d} satisfied by $(u^a,v^a)$, we have the following estimates.\\
{\noindent i) Estimates for $u^a$:}
{\small
\begin{align}
&\q u^a\thickapprox \t{\o}r_0(r-r_0),\q |\p^{i}_\th (u^a-u^e(r))|\ls \dl r_0(r-r_0). \label{uadetail1}\\
&|\p^{i}_\th (u^a-u^e(r))|\ls_{i,k,\ell} \e\dl r_0+\lt(\f{\e^{2/3}}{r-r_0}\rt)^k\e\dl r_0+\lt(\f{\e}{r_0+1-r}\rt)^\ell\dl r_0(r-r_0)^2, \label{uadetail3}\\
&|\p^{i}_\th\p^{j}_r (u^a-u^e(r))|\ls_{i,j,k,\ell} \e\dl r_0+\lt(\f{\e^{2/3}}{r-r_0}\rt)^k\e^{1-\f{2}{3}j}\dl r_0+\lt(\f{\e}{r_0+1-r}\rt)^\ell\e^{-j}\dl r_0(r-r_0)^2. \label{uadetail2}
\end{align}
}
{\noindent ii) Estimates for $v^a$:}
\begin{align}
&|\p^{i}_\th v^a|\ls_{i,\ell} \e\dl(r-r_0),\label{vadetail2} \\
&|\p^{i}_\th\p^j_r v^a|\ls_{i,k,\ell} \e\dl + \lt(\f{\e^{2/3}}{r-r_0}\rt)^k\e^{\f{5}{3}-\f{2}{3}j}\dl +\lt(\f{\e}{|r-r_0-1|}\rt)^\ell\e^{1-j}\dl(r-r_0)^2. \label{vadetail1}
\end{align}
\end{lemma}
\begin{proof} We utilize decomposition in \eqref{uaconstruct} and estimates in \eqref{eulerapproxd}, \eqref{prandtlapprox1d} and \eqref{prandtlapprox2d} to the above estimates.

{\noindent\bf Proof of \eqref{uadetail1}.} First we show the second one of \eqref{uadetail1}, then the first one follow from that
\be\label{uathird}
u^e(r)\approx \t{\o}r_0(r-r_0).
\ee
 From \eqref{uaconstruct}, for $i\in\bN$,
\bes
|\p^i_\th(u^a-u^e(r))|\leq |\p^i_\th\lt[u^a_e-u^e(r)+\hat{\chi}\hat{u}^a_p+\e^{16}h\rt]|+|{\chi}\p^i_\th{u}^a_p|.
\ees
Since $u^a_e+\hat{\chi}\hat{u}^a_p+\e^{14}h-u^e(r)|_{r=r_0}=0$, by using \eqref{eulerapproxd} and \eqref{prandtlapprox2d}, we have
\begin{align}
&|\p^i_\th\lt[u^a_e-u^e(r)+\hat{\chi}\hat{u}^a_p+\e^{16} h\rt]|\nn\\
\leq&\|\p^i_\th(u^a_e-u^e(r)+\hat{\chi}\hat{u}^a_p+\e^{16}h)_r\|_{L^\i}(r-r_0)\nn\\
\ls&  \lt[\e\dl r_0+\e^{-2/3}|\p^i_\th\hat{u}^a_{p,\eta}|+\e^{16}\rt](r-r_0)\nn\\
\ls& \lt[\e\dl r_0+\e^{1/3}\dl r_0 +\e^{16}\rt] (r-r_0)\ls \e^{1/3}\delta r_0 (r-r_0). \label{uafirst}
\end{align}

Also, from  \eqref{eulerapproxd} and \eqref{prandtlapprox1d}, it is easy to see that
\be\label{uasecond}
|{\chi}\p^i_\th{u}^a_p|\leq C{\chi}\dl \leq C\dl (r-r_0),\text{ for } (r-r_0)\in[0.25, 1].
\ee

Combining \eqref{uafirst} and \eqref{uasecond}, we see the second one of \eqref{uadetail1} is valid. Remembering \eqref{uathird}, we see that the first one of \eqref{uadetail1} is valid.

{\noindent\bf Proof of \eqref{uadetail3}.} Also from \eqref{uaconstruct}, and by using \eqref{eulerapproxd}, \eqref{prandtlapprox1d} and \eqref{prandtlapprox2d}, we see that, for $i\in\bN$,
\begin{align}
&|\p^{i}_\th (u^a-u^e(r))|=\lt|\p^i_\th\lt[u^a_e-u^e(r)+\hat{\chi}\hat{u}^a_p+{\chi}\p^i_\th{u}^a_p+\e^{16}h\rt]\rt|\nn\\
\ls& \e\dl r_0+ \hat{\chi} \eta^{-k}\eta^k |\p^i_\th \hat{u}^a_p+ {\chi} \zeta^{-\ell}\zeta^\ell |\p^i_\th {u}^a_p|+ \e^{16} |\p^i_\th h| \nn\\
\ls&\e\dl r_0+\hat{\chi}\eta^{-k} \e\dl r_0+ \chi \zeta^{-\ell}\dl r_0\nn\\
\ls& \e\dl r_0+ \lt(\f{\e^{2/3}}{r-r_0}\rt)^k\e\dl r_0+\lt(\f{\e}{|r-r_0-1|}\rt)^\ell\dl r_0(r-r_0)^2,\nn
\end{align}
which is the third one of \eqref{uadetail3}.

{\noindent\bf Proof one of \eqref{uadetail2}}. From \eqref{uaconstruct}, Leibniz formula, and by using \eqref{eulerapproxd}, \eqref{prandtlapprox1d} and \eqref{prandtlapprox2d}, we see that, for $i,j\in\bN$,
\begin{align}
&|\p^{i}_\th\p^{j}_r (u^a-u^e(r))|\leq |\p^{i}_\th\p^{j}_r (u^a_e-u^e(r))|+|\p^{i}_\th\p^{j}_r (\hat{\chi}\hat{u}^a_p+{\chi}u^a_p)|+\e^{16}|\p^{i}_\th\p^{j}_r h|\nn\\
\ls &\e\dl r^{1-j}_0+ \hat{\chi}|\p^{i}_\th\p^{j}_r \hat{u}^a_p|+\sum^{j}_{j'=1}|\hat{\chi}^{(j')}\p^{i}_\th \p^{j-j'}_r \hat{u}^a_p|+{\chi}\p^{i}_\th|\p^{j}_r {u}^a_p|+\sum^{j+1}_{j'=1}|{\chi}^{(j')}\p^{i}_\th\p^{j-j'}_r{u}^a_p|+\e^{16}\nn\\
\ls&\e\dl r_0+\hat{\chi}\eta^{-k}\e^{-\f{2}{3}(j)}\eta^{k}|\p^{i}_\th\p^{j}_\eta \hat{u}^a_p|+{\chi}|\zeta|^{-\ell}\e^{-j}|\zeta^{\ell}\p^{i}_\th\p^{j}_\zeta {u}^a_p|\nn\\
\ls& \e\dl r_0+ \lt(\f{\e^{2/3}}{r-r_0}\rt)^k\e^{1-\f{2}{3}j}\dl r_0+\lt(\f{\e}{|r-r_0-1|}\rt)^\ell\e^{-j}\dl r_0(r-r_0)^2,\nn
\end{align}
which is the second one of \eqref{uadetail2}.

{\noindent\bf Proof of \eqref{vadetail2}.} This will be similar with that of the second of \eqref{uadetail1}. From \eqref{uaconstruct}, for $i\in\bN$,
\be
|\p^i_\th v^a|\leq |\p^i_\th\lt[v^a_e+\hat{\chi}\hat{v}^a_p\rt]|+|{\chi}\p^i_\th{v}^a_p|. \nn
\ee
Since $v^a_e+\hat{\chi}\hat{v}^a_p|_{r=r_0}=0$, by using \eqref{eulerapproxd} and \eqref{prandtlapprox2d}, we have
\begin{align}
&|\p^i_\th\lt[v^a_e+\hat{\chi}\hat{v}^a_p\rt]|\leq\|\p^i_\th(v^a_e+\hat{\chi}\hat{v}^a_p)_r\|_{L^\i}(r-r_0)\ls \lt[\e\dl +\e^{-2/3}|\p^i_\th\hat{v}^a_{p,\eta}|\rt](r-r_0)\ls \e\delta(r-r_0).\nn
\end{align}
Also it is easy to see that $|{\chi}\p^i_\th{v}^a_p|\ls \e\dl (r-r_0)$. Then \eqref{vadetail2} is valid .

{\noindent\bf Proof of \eqref{vadetail1}.} From \eqref{uaconstruct} and by using \eqref{eulerapproxd}, \eqref{prandtlapprox1d} and \eqref{prandtlapprox2d}, we see that, for $i\in\bN$,
\begin{align}
&|\p^{i}_\th\p^{j}_rv^a|\leq |\p^{i}_\th\p^{j}_r v^e(r)|+|\p^{i}_\th\p^{j}_r (\hat{\chi}\hat{v}^a_p+{\chi}v^a_p)|\nn\\
\ls &\e\dl r^{-j}_0+ \hat{\chi}|\p^{i}_\th\p^{j}_r \hat{v}^a_p|+\sum^{j}_{j'=1}|\hat{\chi}^{(j')}\p^{i}_\th \p^{j-j'}_r \hat{v}^a_p|+{\chi}|\p^{i}_\th\p^{j}_r {v}^a_p|+\sum^{j}_{j'=1}|{\chi}^{(j')}\p^{i}_\th\p^{j-j'}_r{v}^a_p|\nn\\
\ls&\e\dl +\hat{\chi}\eta^{-k}\e^{-\f{2}{3}j}\eta^{k}|\p^{i}_\th\p^{j}_\eta \hat{u}^a_p|+{\chi}|\zeta|^{-\ell}\e^{-j}|\zeta^{\ell}\p^{i}_\th\p^{j+1}_\zeta {v}^a_p|\nn\\
\ls& \e\dl + \lt(\f{\e^{2/3}}{r-r_0}\rt)^k\e^{\f{5}{3}-\f{2}{3}j}\dl +\lt(\f{\e}{|r-r_0-1|}\rt)^\ell\e^{1-j}\dl(r-r_0)^2,\nn
\end{align}
which is \eqref{vadetail1}.

\end{proof}

Next, we present two types of the Hardy inequalities as follows.
\begin{lemma}\label{lemhardy}
For function $f(r)\in C^1([r_0,r_0+1])$ with $f(r_0+1)=0$, we have
\be\label{hardy1}
\int^{r_0+1}_{r_0} (r-r_0)^{\kappa} f^2 dr \leq \f{4}{(\kappa+1)^2}\int^{r_0+1}_{r_0} (r-r_0)^{\kappa+2} (f_r)^2dr\q \text{for } \kappa>-1.
\ee
For function $f(r)\in C^1([r_0,r_0+1])$ with $f(r_0)=0$, we have
\be\label{hardy1plus}
\int^{r_0+1}_{r_0} (r_0+1-r)^{\kappa} f^2 dr \leq \f{4}{(\kappa+1)^2}\int^{r_0+1}_{r_0} (r_0+1-r)^{\kappa+2} (f_r)^2dr\q \text{for } \kappa>-1.
\ee
If $f(r)\in C^1([r_0,r_0+1])$ with $f|_{r=r_0}=f|_{r=r_0+1}=0$, we have, for $-3<\kappa<-1$,
\be\label{hardy2}
\bali
&\int^{r_0+1}_{r_0} (r-r_0)^{\kappa}f^2 dr \leq \f{4}{(\kappa+1)^2}\int^{r_0+1}_{r_0} (r-r_0)^{\kappa+2}(f_r)^2dr,\\
&\int^{r_0+1}_{r_0}  (r_0+1-r)^{\kappa}f^2 dr \leq \f{4}{(\kappa+1)^2}\int^{r_0+1}_{r_0}  (r_0+1-r)^{\kappa+2}(f_r)^2dr.
\eali
\ee
If $f(r)\in C^2([r_0,r_0+1])$ with $(f,f_r)|_{r=r_0 \text{ and } r_0+1}=0$, we have, for $-5<\kappa<-3$,
\be\label{hardy3}
\bali
&\int^{r_0+1}_{r_0} (r-r_0)^{\kappa}f^2 dr \leq \f{24}{(\kappa+3)^2}\int^{r_0+1}_{r_0} (r-r_0)^{\kappa+4} f^2_{rr}dr,\\
&\int^{r_0+1}_{r_0}  (r_0+1-r)^{\kappa} f^2 dr \leq \f{24}{(\kappa+3)^2}\int^{r_0+1}_{r_0} (r_0+1-r)^{\kappa+4} f^2_{rr} dr.
\eali
\ee

\end{lemma}
\begin{proof}
Using integration by parts, the Cauchy inequality and H\"{o}lder inequality, we have
\begin{align}
\int^{r_0+1}_{r_0} (r-r_0)^{\kappa} f^2 dr =&\f{1}{1+\kappa}\int^{r_0+1}_{r_0}  f^2 d (r-r_0)^{\kappa+1}=-\f{1}{1+\kappa}\int^{r_0+1}_{r_0} 2(r-r_0)^{\kappa+1} f f_r dr\nn\\
                       \leq & \f{2}{1+\kappa}\lt(\int^{r_0+1}_{r_0} (r-r_0)^{\kappa} f^2 dr\rt)^{1/2}\lt(\int^{r_0+1}_{r_0} (r-r_0)^{\kappa+2} (f_r)^2 dr\rt)^{1/2},\label{hardyproof2}
\end{align}
which indicates \eqref{hardy1}. The proof of \eqref{hardy1plus} is the same, we omit the detail.

For the proof of \eqref{hardy2}, also by using integration by part, we have
\begin{align}
\int^{r_0+1}_{r_0} (r-r_0)^{\kappa} f^2 dr=&\f{1}{\kappa+1}\int^{r_0+1}_{r_0}  f^2 d(r-r_0)^{\kappa+1}\nn\\
                =& \f{1}{\kappa+1} f^2(r)(r-r_0)^{\kappa+1}\Big|^{r_0+1}_{r_0}-\f{2}{\kappa+1}\int^{r_0+1}_{r_0} (r-r_0)^{\kappa+1} f f_r dr.\nn
\end{align}
Using the facts $f(r_0)=0$ and $f(r)\in C^1([r_0,r_0+1])$ and the mean value formula, we see that
\bes
|f(r)|\leq\|\p_r f\|_{L^\i}(r-r_0)\ls (r-r_0),
\ees
which indicate that $f^2(r)(r-r_0)^{\kappa+1}\big|_{r=r_0}=0$, then the same as \eqref{hardyproof2}, we have the first one \eqref{hardy2}. Proof of the second one is the same as that of the first one of \eqref{hardy2}.

Proof of \eqref{hardy3} is similar as that of \eqref{hardy2} with two times' integration by parts and Cauchy inequality, we omit the details.
\end{proof}

By using the above Hardy inequality Lemma, we give the following useful Lemma which will be used in the energy estimates.
\begin{lemma}
For $f(r)\in C^1([r_0,r_0+1])$ with $f(r_0)=0$, then for $\al\in(0,\f{1}{100})$, we have
\be\label{hardy4}
\int^{r_0+1}_{r_0} f^2(r-r_0)^{\al-1}dr\ls \al\int^{r_0+1}_{r_0} f^2_r(r-r_0)^\al dr+\f{1}{\al^3} \int^{r_0+1}_{r_0} f^2(r-r_0) dr.
\ee
\end{lemma}
\begin{proof}
Let $\chi(r)\in C^\i_c([r_0,+\i))$ be a function satisfying
\bes
\chi(r)=\lt\{
\bali
&1,\q r\in \lt[r_0, r_0+\al^3\rt],\\
&0,\q r\geq r_0+2\al^3,
\eali
\rt.
\ees
and $|\chi'|\ls \f{1}{\al^3}$.  Then by using Hardy inequality \eqref{hardy1}, we can obtain that
\begin{align*}
&\int^{r_0+1}_{r_0} f^2\left(r-r_0\right)^{\alpha-1} dr\nn\\
=&\int^{r_0+1}_{r_0} f^2\chi^2\left(r-r_0\right)^{\alpha-1}dr+\int^{r_0+1}_{r_0} f^2(1-\chi^2)\left(r-r_0\right)^{\alpha-1} dr\\
\ls& \f{1}{\al^2} \int^{r_0+1}_{r_0} \lt[f_r^2\chi^2+f^2(\chi')^2\rt]\left(r-r_0\right)^{\alpha+1}dr+\int^{r_0+1}_{r_0} f^2(1-\chi^2)\left(r-r_0\right)^{\alpha-1}dr\\
\ls&  \f{1}{\al^2} \int^{r_0+1}_{r_0} f_r^2\chi^2\left(r-r_0\right)^{\alpha+1}dr+\int^{r_0+1}_{r_0} f^2\lt[(\chi')^2\left(r-r_0\right)^{\alpha+1}+(1-\chi^2)\left(r-r_0\right)^{\alpha-1}\rt]dr\\
\ls& \al\int^{r_0+1}_{r_0} f_r^2\left(r-r_0\right)^{\alpha}dr+\f{1}{\al^6}\int^{r_0+1}_{r_0} f^2\left(r-r_0\right)dr.
\end{align*}
At the last line of the above inequality, we have used the following facts
\bes
|r-r_0|\ls \al^3 \text{ in } \text{spt}\chi,\q |\chi'|\ls \f{1}{\al^3},\q |r-r_0|\gtrsim \al^3 \text{ in } \text{spt}(1-\chi^2).
\ees
\end{proof}

\section{$H^1$ stability estimates of the linear error equations}\label{sec3}

\indent In this section, we give the $H^1$ stability estimate of \eqref{errorequ}.  The main estimate for the linear Stokes system \eqref{errorequ} is the following.
\begin{proposition}\label{proplinearh1}
 Let $(u,v)$ be a smooth solution of (\ref{errorequ}). Then there exist $\epsilon_0>0, \dl_0>0$ and a small  constant $\al\in(0,0.01]$, such that for any $\epsilon\in (0,\epsilon_0), \dl\in(0,\dl_0)$, there holds
\begin{align}\label{linearstabilityh1}
&\e^4\lt\{\int_{\O}\lt(u^2_{r}+v^2_r\rt)+\lt(u^2_{\th}+v^2_\th\rt)(r-r_0)+r^{-2}_0\lt(u^2_{\th}+v^2_\th\rt)\rt\}\nn\\
&\leq Cr^{-1}_0\lt|\int_{\O} \lt( R_u u_{\th}+R_v v_{\th}\rt)\rt|+Cr^{-1}_0\lt|\int_{\O} R_u u_0\rt|\nn\\
&+Cr^{-1}_0\lt|\int_{\O} \lt( R_u u_{\neq}+R_v v\rt)\rt|+ Cr^{-1}_0\lt|\int_{\O}  R_u \Phi_{\neq}(r-r_0)^{\al-1}\rt|.
\end{align}
Here the constant $\al$ is independent of $\e$ and $\dl$.
\end{proposition}
Proof of Proposition \ref{proplinearh1} contains three estimates: one is a positive estimate in Section \ref{subspositive} to obtain the $\e$ independently tangential derivative estimate for the second term on the left hand of \eqref{linearstabilityh1}, the second one is the weighted normal energy estimates in Section \ref{subsenergy} to close the positive estimates. The third one is the non-weighted energy estimate in  Section \ref{sub3.3} for the normal derivatives for the first term on the left hand of \eqref{linearstabilityh1}.

\subsection{The $\boldsymbol{\e}$ independently positivity estimate}\label{subspositive}

\indent

In this subsection, we give the weighted $\e$ independently horizontal derivatives' estimate for the solution $(u,v)$. We have the following lemma.

\begin{proposition}\label{proppositive}
 Let $(u,v)$ be a smooth solution of (\ref{errorequ}), then there exist $\epsilon_0>0, \dl_0>0$ such that for any $\epsilon\in (0,\epsilon_0), \dl\in(0,\dl_0)$, there holds
\begin{align}\label{estipositive}
&\int_{\O}(r-r_0) (u^2_\th+v^2_\th) d\th dr\ls \e^2\dl\int_{\O}u^2_{0,r}+(r-r_0)(u^2_{\neq r}+v^2_r)+r^{-1}_0\lt|\int_{\O} (R_u u_{\th}+R_v v_{\th})\rt|.
\end{align}
\end{proposition}

\pf Multiplying \eqref{errorequ}$_1$ by $u_\th$, \eqref{errorequ}$_2$ by $v_\th$, integrating the resulted equations on $\O$  and adding them together to obtain that
\begin{align}
\int_{\O} u^a(u^2_\th+v^2_\th)=&\underbrace{-\int_{\O}(p_\th u_\th+rp_rv_\th)}_{I_1}+\underbrace{\epsilon^2\int_{\O}\left(\frac{u_{\theta \theta}}{r}+r u_{r r}+u_r+\frac{2}{r} v_\theta-\frac{u}{r}\right)u_\th}_{I_2}\nn\\
&\underbrace{+\epsilon^2\int_{\O}\left(\frac{v_{\theta \theta}}{r}+r v_{r r}+v_r-\frac{2}{r} u_\theta-\frac{v}{r}\right)v_\th}_{I_3}\nn\\
&-\underbrace{\int_{\O}(v^ar u_r+u u^a_\th+vr u^a_r+uv^a+u^av)u_\th}_{I_4}\underbrace{-\int_{\O}(v^ar v_r+u v^a_\th+vr v^a_r-2uu^a)v_\th}_{I_5}, \label{positiveesti13}\\
&+\int_{\O} (R_u u_{\th}+R_v v_{\th}).\nn
\end{align}

We estimate the righthand term by term.

{\noindent\bf Term $I_1$}: By using integration by parts and the incompressibility, we see that
\begin{align*}
I_1:=& -\int_{\O}(p_\th u_\th+rp_rv_\th)=\int_{\O}( u_{\th\th}+(rv_{\th})_r)p=0.
\end{align*}
{\noindent\bf Term $I_2+I_3$}: By using integration by parts, we see that
\begin{align*}
I_2+I_3=&\int_{\O}\lt\{ \f{1}{2r}(u^2_\th+v^2_\th)_\th-\f{r}{2}(u^2_r+v^2_r)_\th-\f{1}{2r}(u^2+v^2)_\th\rt\}=0.
\end{align*}

{\noindent\bf Term $I_4$}:

{\it\noindent \underline{Estimate of $\int_{\O}v^ar u_r u_\th d\th dr$}}

Using the estimate \eqref{vadetail2} and Cauchy inequality, we have
\begin{align}
\lt|\int_{\O}v^ar u_r u_\th\rt|\ls& \e\dl \int_{\O}(r-r_0)r |u_r u_\th|\ls \e^2\dl r_0 \int_{\O}(r-r_0) u^2_r+\dl r_0\int_{\O}(r-r_0) u^2_\th. \label{positive4}
\end{align}

{\it\noindent \underline{Estimate of $\int_{\O}(u^a_\th+v^a) u u_\th d\th dr$}}

Using the estimate \eqref{uadetail1} and \eqref{vadetail2},  Cauchy inequality, Poincar\'{e} inequality with zero mean in $\th$ variable, we have
\begin{align}
&\lt|\int_{\O}(u^a_\th+v^a) u_{\neq} u_\th\rt|\ls \dl r_0\int_{\O}(r-r_0)|u_{\neq0} u_\th|\nn\\
\ls& \dl r_0 \int_{\O}(r-r_0) u^2_{\neq}+\dl r_0 \int_{\O}(r-r_0) u^2_\th\ls \dl r_0\int_{\O}(r-r_0) u^2_\th. \label{positiveesti2}
\end{align}
Using the estimate \eqref{uadetail3} and \eqref{vadetail1}, we have
\bes
|u^a_\th |+|v^a|\ls \e\dl r_0+\e\dl r_0\f{(r-r_0)^2}{r_0-1-r}.
\ees
Then by Cauchy inequality, Poincar\'{e} inequality and Hardy inequality \eqref{hardy2}, we have
\begin{align}
\lt|\int_{\O}(u^a_\th+v^a) u_0 u_\th\rt|\ls& \e\dl r_0 \int_{\O}\lt( 1+ \f{(r-r_0)^2}{r_0+1-r} \rt)|u_0 u_\th|\nn\\
\ls& \e^2\dl r_0 \int_{\O}\lt(\f{u^2_0}{(r-r_0)}+\f{((r-r_0) u_0)^2}{(r_0+1-r)^2}\rt)+\dl r_0 \int_{\O} (r-r_0) u^2_\th\nn\\
\ls& \e^2\dl r_0 \int_{\O}\lt(\f{u^2_0}{(r-r_0)^{2}}+\f{((r-r_0) u_0)^2}{(r_0+1-r)^2}\rt)+\dl r_0 \int_{\O}(r-r_0) u^2_\th\nn \\
\ls& \e^2\dl r_0 \int_{\O} u^2_{0,r}+\dl r_0 \int_{\O}(r-r_0) u^2_\th. \label{positiveesti3}
\end{align}
Combining estimates \eqref{positiveesti2} and \eqref{positiveesti3}, we can obtain
\be\label{positiveesti4}
\lt|\int_{\O}(u^a_\th+v^a) u u_\th\rt|\ls\e^2\dl r_0 \int_{\O} u^2_{0,r}+\dl r_0 \int_{\O}(r-r_0) u^2_\th.
\ee

{\it\noindent \underline{Estimate of $\int_\O (u^a_r vr+u^av) u_\th d\th dr$}}

First, using the divergence free condition, integration by parts and the fact that
\bes
\lt(u^e_{r}+\f{u^e}{r}\rt)_r=0,
\ees we have
\begin{align}
&-\int_\O (u^a_r vr+u^av) u_\th=\int_\O (u^a_r vr+u^av) (vr)_r\nn\\
=&\f{1}{2}\int_\O \lt(u^a_r+\f{u^a}{r}\rt) (vr)^2_{r}=-\f{1}{2}\int_\O\lt(u^a_r+\f{u^a}{r}\rt)_r(vr)^2\nn\\
=&-\f{1}{2}\int_\O\lt(u^e_{r}+\f{u^e}{r}\rt)_r(vr)^2-\f{1}{2}\int_\O\lt[(u^a-u^e)_{rr}+\lt(\f{u^a-u^e}{r}\rt)_r\rt] (vr)^2\nn\\
=&-\f{1}{2}\displaystyle\int_\O\lt[(u^a-u^e)_{rr}+\lt(\f{u^a-u^e}{r}\rt)_r \rt](vr)^2. \label{positiveesti1}
\end{align}

Using the estimate \eqref{uadetail1} and \eqref{uadetail2}, we have
\begin{align*}
&|(u^a-u^e)_{rr}|+\lt|\lt(\f{u^a-u^e}{r}\rt)_r\rt|\ls \e\dl r_0+\dl r_0\f{1}{\s{r-r_0}}+\dl r_0\f{(r-r_0)^2}{(r_0+1-r)^2}.
\end{align*}
Inserting the above estimates into \eqref{positiveesti1}, using Hardy inequality and the divergence free condition, we can obtain that
\begin{align}
&\lt|\int_\O (u^a_r vr+u^av) u_\th\rt|\ls r_0\int_\O \lt(\f{\dl}{\s{r-r_0}}+\e\dl+\dl \f{(r-r_0)^2}{(r_0+1-r)^2}\rt) (vr)^2\nn\\
\ls& \dl r_0 \int_\O (r-r_0)^{3/2}(vr)^2_r +\e\dl r_0 \int_\O(r-r_0)^2(vr)^2_r+\dl r_0 \int_\O (r-r_0)^2(vr)^2_r\nn\\
\ls&  \dl r_0 \int_\O (r-r_0)(vr)^2_r\ls \dl r_0 \int_\O (r-r_0)u^2_\th. \label{positiveesti10}
\end{align}

{\noindent\bf Term $I_5$}:

{\it\noindent \underline{Estimate of $\int_\O (v^a r v_r+v^a_\th u) v_\th d\th dr$}}

Using the estimate \eqref{vadetail2},  Cauchy inequality and Hardy inequality \eqref{hardy2}, we have
\begin{align}
&\lt|\int_\O (v^a r v_r+v^a_\th u) v_\th\rt|\ls \e\dl  \int_{\O} (r_0|v_r|+ |u|)(r-r_0) v_\th\nn\\
\ls &\dl r_0 \int_{\O} (r-r_0)v^2_\th+ \e^2\dl r_0\int_\O (r-r_0) v^2_r +\e^2\dl\f{1}{r_0}\int_\O (r-r_0)u^2\nn\\
\ls &\dl r_0\int_{\O} (r-r_0)v^2_\th+ \e^2\dl r_0\int_\O (r-r_0) v^2_r +\e^2\dl\f{1}{r_0}\int_\O (r-r_0)^3 u^2_r\nn\\
\ls& \dl r_0 \int_{\O}(r-r_0)v^2_\th+ \e^2\dl r_0\int_\O (r-r_0)(u^2_r+v^2_r).  \label{positiveesti11}
\end{align}
{\it\noindent \underline{Estimate of $\int_\O vrv^a_r v_\th d\th dr$}}

Using \eqref{uadetail1}, \eqref{vadetail2} and the incompressibility, we see that
\bes
r|v^a_r|\leq |u^a_\th|+|v^a|\ls \dl(r-r_0),
\ees
Utilizing the above estimate, Cauchy inequality and Poincar\'{e} inequality with zero mean in $\th$ variable, we obtain that
\begin{align}
&\lt|\int_{\O}v^a_r vr v_\th\rt|\ls  \dl r_0\int_{\O} (r-r_0)|v v_\th|\ls \dl r_0\int_{\O}(r-r_0)(v^2+v^2_\th)\ls \dl r_0\int_{\O}(r-r_0)v^2_\th.\label{positiveesti12}
\end{align}

{\it\noindent \underline{Estimate of $2\int_\O u^a u v_\th d\th dr$}}

First, by using integration by parts and the incompressibility, we have
\begin{align}
&2\int_\O u^a u v_\th =-2\int_\O u^a_\th u v -2\int_\O u^a u_\th v\nn\\
=& -2\int_\O u^e(r) u_\th v -2\int_\O (u^a-u^e) u_\th v -2\int_\O u^a_\th u v\nn\\
=& 2\int_\O u^e(r) (rv)_r v -2\int_\O (u^a-u^e) u_\th v -2\int_\O u^a_\th u v\label{positiveesti5}\\
=& \int_\O \f{u^e}{r} (rv)^2_r -2\int_\O (u^a-u^e) u_\th v -2\int_\O u^a_\th u v\nn\\
=& -\int_\O \lt(\f{u^e}{r}\rt)_r (rv)^2 -2\int_\O (u^a-u^e) u_\th v -2\int_\O u^a_\th u v.\nn
\end{align}
Using the estimate \eqref{uadetail1}, Cauchy inequality and Poincar\'{e} inequality, it is easy to see that
\begin{align}
\lt|\int_\O (u^a-u^e) u_\th v\rt|\ls \dl r_0 \int_\O (r-r_0)(u^2_\th+v^2)\ls \dl r_0 \int_\O (r-r_0)(u^2_\th+v^2_\th). \label{positiveesti6}
\end{align}
The same estimate as \eqref{positiveesti2} and \eqref{positiveesti3}, we can have
\begin{align}
\lt|\int_{\O}u^a_\th u v\rt|\ls\e^2\dl r_0 \int_{\O} u^2_{0,r}+\dl r_0 \int_{\O}(r-r_0) (u^2_\th+v^2_\th). \label{positiveesti7}
\end{align}
Since $u^e=A\lt(r-\f{r^2_0}{r}\rt)$, then
\begin{align}
-\int_\O \lt(\f{u^e}{r}\rt)_r (rv)^2= -2Ar^2_0\int_\O \f{(rv)^2}{r^3}\leq 0. \label{positiveesti8}
\end{align}
Inserting \eqref{positiveesti6}, \eqref{positiveesti7} and \eqref{positiveesti8} into \eqref{positiveesti5} implies that
\begin{align}
&2\int_\O u^a u v_\th \ls  \e^2\dl r_0 \int_{\O} u^2_{0,r}+\dl r_0 \int_{\O}(r-r_0) (u^2_\th+v^2_\th). \label{positiveesti9}
\end{align}

\begin{remark} \label{remboundaryc}

 If we change the boundary condition between the inner and outer boundary, i.e. the flow is subjected to a perturbation of the rotation of rigid body at the inner boundary and stays at rest at the outer boundary, then the limiting Euler $\bl{u}^e$ is given by $\bl{u}^e=u^e\bl{e}_\th$, where  $u^e=\f{r^2_0}{2r_0+1}\t{o}\lt(\f{(r_0+1)^2}{r}-r\rt)$. In this case, the only difference is the estimate of \eqref{positiveesti8}. In this case, although the sign of this term changes, however if $r_0$ is suitably large, it can be absorbed. We calculate it as follows. By using Hardy iequality \eqref{hardy1} and the incompressibility

\begin{align}
&-\int_\O \lt(\f{u^e}{r}\rt)_r (rv)^2= \f{2r^2_0(r_0+1)^2}{2r_0+1}\t{\o}\int_\O \f{(rv)^2}{r^3} \leq  \f{2(r_0+1)^2}{(2r_0+1)r_0}\t{\o}\int_\O (rv)^2\leq \f{8(r_0+1)^2}{(2r_0+1)r_0}\t{\o}\int_\O (r-r_0)(rv)^2_r \nn\\
\leq&  9\t{\o} \int_\O (r-r_0)u^2_\th.\nn
\end{align}
This term can be absorbed by  $\int_\O u^au^2_\th$ if $r_0$ is suitably large.
\end{remark}

At last, inserting estimates in \eqref{positive4}, \eqref{positiveesti4},  \eqref{positiveesti10},  \eqref{positiveesti11},  \eqref{positiveesti12},  \eqref{positiveesti9} into  \eqref{positiveesti13} indicates
 the following positive estimate
\begin{align}
\int_{\O} u^a(u^2_\th+v^2_\th) \ls&\e^2\dl r_0 \int_{\O}u^2_{0,r}+ (r-r_0) (u^2_{r}+v^2_r)+\dl r_0\int_{\O}(r-r_0) \lt(u^2_\th+v^2_\th\rt)\nn\\
                                 &+\lt|\int_{\O} (R_u u_{\th}+R_v v_{\th})\rt|.\nn
\end{align}
Remembering \eqref{uadetail1}, then there exists a small $\dl_0$ such that when $\dl\in (0,\dl_0]$, such that \eqref{estipositive} stands.

In order to close the linear estimate, we still need to give the control of the first term on the right hand of \eqref{estipositive}, which will be achieved by the following lemma through basic energy estimates.

\subsection{Weighted energy estimates}\label{subsenergy}
\begin{proposition}\label{proplinearenergy}
 Let $(u,v)$ be a smooth solution of (\ref{errorequ}), then there exist $\epsilon_0>0, \dl_0>0$ such that for any $\epsilon\in (0,\epsilon_0), \dl\in(0,\dl_0)$  and a small constant $\al\in (0,0.1]$, there holds
\begin{align}\label{estilinearwei}
&\e^2\int_{\O} u^2_{0,r}+(r-r_0)^\al (u^2_{\neq,r}+v^2_r)\nn\\
&\ls \al^{-6}\int_{\O}(r-r_0)(u^2_\th+v^2_\th)+Cr^{-1}_0\lt|\int_{\O} R_u u_0\rt|\nn\\
&+Cr^{-1}_0\lt|\int_{\O} (R_u u_{\neq}+R_v v)\rt|+ Cr^{-1}_0\lt|\int_{\O} (r-r_0)^{\al-1}R_u \Phi_{\neq}\rt|.
\end{align}
\end{proposition}

\pf The proof is divided into two parts. The first part deals with  the non weighted estimate of zero frequency of $u$, namely, $u_0$. The other part is the weighted estimate of the non-zero frequency of $u$, namely, $u_\neq$.

\subsubsection{Energy estimate of $u_{0}$}
\indent
\begin{lemma}\label{lemlinearzero}
 Let $(u,v)$ be a smooth solution of (\ref{errorequ}), then there exist $\epsilon_0>0, \dl_0>0$ such that for any $\epsilon\in (0,\epsilon_0), \dl\in(0,\dl_0)$, there holds
\begin{align}
\e^2 \int_{\O} u^2_{0,r} \ls\dl \int_{\O} (r-r_0) u^2_{\th}+Cr^{-1}_0\lt|\int_{\O} R_u  u_0\rt|.\nn
\end{align}
\end{lemma}

Now we go back to the first equation of \eqref{errorequ}.
\be\label{errorenergys1}
-\epsilon^2\left(\frac{u_{\theta \theta}}{r}+r u_{r r}+u_r+\frac{2}{r} v_\theta-\frac{u}{r}\right)+\p_\th p +S_u=R_u.
\ee
Now multiplying \eqref{errorenergys1} by $ u_0$ and then integrating on $\O$, we can obtain that
\begin{align}
&\underbrace{-\int_{\O}\epsilon^2\left(\frac{u_{\theta \theta}}{r}+r u_{r r}+u_r+\frac{2}{r} v_\theta-\frac{u}{r}\right) u_0}_{J_1}\nn\\
&=\underbrace{-\int_{\O} \lt(u^au_\th+uu^a_\th+p_\th\rt) u_0}_{J_2}\underbrace{-\int_{\O} \lt(v^a ru_r+vru^a_r+uv^a+u^av\rt) u_0}_{J_3}+\int_{\O} R_u  u_0.\nn
\end{align}
{\noindent\bf Term $J_1$}: Using the boundary condition $u|_{r=r_0,r_0+1}=u_0|_{r=r_0,r_0+1}=0$ and integration by parts, we have
\begin{align}
J_1=&-\int_{\O}\epsilon^2\left(\frac{u_{\theta \theta}}{r}+r u_{r r}+u_r+\frac{2}{r} v_\theta-\frac{u}{r}\right) u_0\nn\\
=&-\e^2 \int_{\O}\left(r u_{r r}+u_r-\frac{u}{r}\right) u_0=-\e^2 \int_{\O}\lt( (ru_r)_r-\f{u}{r}  \rt)    u_0\nn\\
=&\e^2 \int_{\O}r  u^2_{0,r}+\f{u^2_0}{r}\geq \e^2r_0\int_{\O} u^2_{0,r}. \label{energyk1}
\end{align}
{\noindent\bf Term $J_2$}: By using integration by parts in $\th$ and $u_0$ being independent of $\th$, it is easy to see that
\be\label{energyk2}
J_2=0.
\ee
{\noindent\bf Term $J_3$}: By using the facts that
\bes
\int^{2\pi}_0 v^ad\th=\int^{2\pi}_0 vd\th=0,
\ees
and $u_0$ being independent on $\th$, we have
\begin{align}
J_3=&\int_{\O}  (r{u}_{\neq} )_r v^a  u_0+ [ru^a_{\neq }]_r v  u_0\nn\\
   =&\int_{\O} -r{u}_{\neq}  (v^a_ru_0+v^au_{0,r}) +[ru^a_{\neq }]_r v u_0. \label{energyk3}
\end{align}
Using the incompressibility, estimates \eqref{uadetail3}, \eqref{vadetail2} and \eqref{vadetail1}, we have
\bes
|rv^a_r|\ls |u^a_\th|+|v^a| \ls \e\dl r_0+\e\dl r_0 \f{(r-r_0)}{r_0+1-r}, \text{and } |v^a|\ls \e\dl(r-r_0).
\ees
Then by using Cauchy inequality, Poincar\'{e} inequality with zero mean, and Hardy inequality \eqref{hardy2}, we have
\begin{align}
&\lt|\int_{\O} r{u}_{\neq} (v^a_ru_0+v^au_{0,r})\rt|\nn\\
\ls& \e\dl r_0 \int_{\O} \lt(1+\f{(r-r_0)}{r_0+1-r}\rt)|{u}_{\neq} u_0|+\e\dl r_0 \int_{\O} (r-r_0) |u_{\neq}u_{0,r}| \nn\\
\ls& \dl r_0\int_{\O} (r-r_0)u^2_{\neq}+\e^2\dl r_0\int_{\O} (r-r_0)^{-2}\lt(u^2_0+\f{(r-r_0)^2u^2_0}{(r_0+1-r)^2}\rt) +\e^2\dl r_0 \int_{\O}  u^2_{0,r}\nn\\
\ls& \dl r_0\int_{\O} (r-r_0)u^2_{\th}+\e^2\dl r_0\int_{\O} u^2_{0,r}.\label{energyk3f}
\end{align}

Using the estimate \eqref{uadetail2} and \eqref{uadetail1}, we see that
\bes
|(ru^a_{\neq })_r|\ls r|(u^a-u^e)_r|+|(u^a-u^e)|\ls r\e\dl r_0\lt(\f{1}{r-r_0}+\f{(r-r_0)^2}{(r_0+1-r)^2}\rt).
\ees
Then by using Cauchy inequality, Hardy inequality \eqref{hardy1}, \eqref{hardy2}, and the incompressibility, we have
\begin{align}
&\lt|\int_{\O} (ru^a_{\neq })_r v u_0\rt|\ls \e\dl r_0\int_\O \lt(\f{1}{r-r_0}+\f{(r-r_0)^2}{(r_0+1-r)^2}\rt) |rv u_0|\nn\\
\ls &\dl r_0\int_\O \lt(1+\f{(r-r_0)^{2}}{(r_0+1-r)^2}\rt) (rv)^2+\e^2\dl r_0\int_\O (r-r_0)^{-2}\lt(1+\f{(r-r_0)^{4}}{(r_0+1-r)^2}\rt) u^2_0\nn\\
\ls &\dl r_0\int_\O (r-r_0)^{2} (rv)^2_r+ \e^2\dl r_0\int_\O  u^2_{0,r}\ls \dl r_0\int_\O (r-r_0) u^2_\th + \e^2\dl r_0\int_\O u^2_{0,r}.\label{energyk3s}
\end{align}
Inserting \eqref{energyk3f} and \eqref{energyk3s} into \eqref{energyk3}, we see that
\be\label{energyk3four}
J_3\ls \dl r_0\int_\O (r-r_0) u^2_\th + \e^2\dl r_0\int_\O  u^2_{0,r}.
\ee
Then combining estimates \eqref{energyk1}, \eqref{energyk2} and \eqref{energyk3four}, we obtain that there exists small $\dl_0$, when $\dl\leq \dl_0$,
\begin{align}
&\e^2 \int_{\O} u^2_{0,r} \ls\dl \int_{\O} (r-r_0) u^2_{\th}+Cr^{-1}_0\lt|\int_{\O} R_u  u_0\rt|.\nn
\end{align}
 \qed

{
Next, we give the weighted energy estimates for the non-zero frequency. In order to get rid of the trouble caused by the pressure, we will proceed it in the vorticity equation \eqref{errorequcut1}.
\subsubsection{Weighted energy estimates for $u_{\neq}$ and $v$ }
\begin{lemma}\label{lemlinearnonzero}
 Let $(u,v)$ be a smooth solution of (\ref{errorequ}), then there exist $\epsilon_0>0, \dl_0>0$ such that for any $\epsilon\in (0,\epsilon_0), \dl\in(0,\dl_0)$ and small $\al\in(0,0.01]$, there holds
\begin{align}\label{estiunonzero}
&\e^2 \int_{\O} (r-r_0)^\al\lt[u^2_{\neq,r}+v^2_r+ r^{-2}_0(u^2_\th+v^2_\th)\rt]\nn\\
 \ls& \al^{-6} \int_{\O} (r-r_0) \lt(u^2_{\th}+v^2_\th\rt)+\e^2\dl \int u^2_{0,r}+Cr^{-1}_0\lt|\int_{\O} R_u  u_{\neq}+ R_v  v\rt|+ \al r_0^{-1}\lt| \int_\O R_u \Phi_{\neq}(r-r_0)^{\al-1}\rt|.
\end{align}
\end{lemma}
\pf Multiplying \eqref{errorequcut1} by $\Phi_{\neq}(r-r_0)^\al$ to obtain
\be\label{wenergyj1zer}
\begin{aligned}
 &\underbrace{-\epsilon^2 \int r\Delta\o \Phi_{\neq} (r-r_0)^\al}_{K_1}+\underbrace{\int\left(\frac{1}{r} \partial_\th S_v-\p_rS_u\right)\Phi_{\neq} (r-r_0)^\al}_{K_2}=\underbrace{\int\left(\frac{1}{r} \partial_\th R_v-\partial_r R_u\right)\Phi_{\neq} (r-r_0)^\al}_{K_3}.
\end{aligned}
\ee
\underline{{\bf\noindent Estimates of $K_1$}}

Using integration by parts and the relation among $\o$, $u_{\neq}$ and $\Phi_{\neq}$ , we have
\begin{align}
& K_1=-\epsilon^2 \int\left[\p_r(r\o_r)+\f{\o_{\th\th}}{r}\right] \Phi_{\neq} (r-r_0)^\al \nn\\
=&\epsilon^2 \int r \o_{r}\left(u_{\neq}\left(r-r_0\right)^\alpha+\Phi_{\neq} \alpha\left(r-r_0\right)^{\alpha-1}\right) +\epsilon^2 \int \left(r-r_0\right)^\alpha v_\th \o \nn\\
=&\epsilon^2 \int \left(r-r_0\right)^\alpha v_\th \o-\epsilon^2 \int \o(ru_{\neq})_r\left(r-r_0\right)^\alpha-\e^2\al \int  \o\lt[ru_{\neq} \left(r-r_0\right)^{\alpha-1} \rt]\nn\\
 &-\epsilon^2 \al \int  \o \left[(r\Phi_{\neq})_r\left(r-r_0\right)^{\alpha-1}+r\Phi_{\neq}(\al-1)\left(r-r_0\right)^{\alpha-2}\right]\nn\\
=&\epsilon^2 \int r\o^2_{\neq}(r-r_0)^\al-\e^2\al \int  \o\lt[(2ru_{\neq}+\Phi_{\neq})\left(r-r_0\right)^{\alpha-1}+r\Phi_{\neq}(\al-1)\left(r-r_0\right)^{\alpha-2}\right].\nn
\end{align}

Using Cauchy inequality and Hardy inequalities \eqref{hardy2} and \eqref{hardy3}, we see that
\begin{align}
K_1\geq& \epsilon^2 (1-\al)\int r\o^2_{\neq}(r-r_0)^\al -C \e^2\al r_0\int \lt(u^2_{\neq}\left(r-r_0\right)^{\alpha-2}+\Phi^2_{\neq}\left(r-r_0\right)^{\alpha-4}\rt)\nn\\
    \geq & \epsilon^2 (1-\al)\int r\o^2_{\neq}(r-r_0)^\al-C \e^2\al r_0\int  u^2_{\neq,r}\left(r-r_0\right)^{\alpha}.\label{wenergyj1fi}
\end{align}
Remembering that $\o_{\neq}=\f{v_\th-(ru_{\neq})_r}{r}$, by direct calculation, we see that
\bes
r\o^2=\f{r^2u^2_{\neq,r}+v^2_\th+u^2}{r}+(u^2_{\neq})_r-2\f{v_\th(ru_{\neq})_r}{r}.
\ees
Then by integration by parts, the incompressibility, we see that
\begin{align}
\int r\o^2_{\neq}(r-r_0)^\al=&\int \f{r^2u^2_{\neq,r}+v^2_\th+(u_{\neq})^2}{r}(r-r_0)^\al\nn\\
 &-\al \int u^2_{\neq}(r-r_0)^{\al-1}-2\int ru_\th\lt[\lt(\f{v_r}{r}-\f{v}{r^2}\rt)(r-r_0)^\al+\f{v}{r}\al(r-r_0)^{\al-1}\rt]\nn\\
 =&\int \f{r^2u^2_{\neq,r}+r^2v^2_r+(v_\th-u_{\neq})^2+(u_\th+v)^2}{r}(r-r_0)^\al\nn\\
 &-\al \int (u^2_{\neq}+v^2+2u_\th v)(r-r_0)^{\al-1}.\nn
\end{align}
Then by using  the incompressibility, integration by parts and Cauchy inequality, we can obtain that
\bes
\int r\o^2_{\neq}(r-r_0)^\al
 \geq \int \lt\{r u^2_{\neq,r}+r v^2_r+\f{v^2_\th+u^2_\th}{2r} \rt\}(r-r_0)^\al-Cr_0\int (u^2_{\neq}+v^2)(r-r_0)^{\al-1}.
\ees
Using \eqref{hardy4} and Poincar\'{e} inequality with zero mean, we see that
\begin{align}
\int r\o^2_{\neq}(r-r_0)^\al
 \geq &\int \lt\{r u^2_{\neq,r}+r v^2_r+\f{v^2_\th+u^2_\th}{2r} \rt\}(r-r_0)^\al\nn\\
      &-C\al r_0\int (u^2_{\neq,r}+v^2_r)(r-r_0)^{\al}-C\al^{-6}r_0 \int (u^2_{\th}+v^2_\th)(r-r_0).  \label{wenergyj1se}
\end{align}

For $\al\in (0,1)$ sufficiently small close to zero, inserting \eqref{wenergyj1se} into \eqref{wenergyj1fi}, we obtain
\be
K_1\geq \f{1}{4}\e^2\int \lt\{r (u^2_{\neq,r}+ v^2_r)+\f{v^2_\th+u^2_\th}{r} \rt\}(r-r_0)^\al-C\al^{-6}\e^2r_0 \int (u^2_{\th}+v^2_\th)(r-r_0).  \label{wenergyj1th}
\ee
}

\noindent\underline{{\bf Estimates of $K_2$}}

Using integration by parts and the relation among $\Phi_{\neq}$ and $rv,\, u_{\neq}$, we have
\begin{align}
&\int_\O (r-r_0)^\al \Phi_{\neq} \lt(\f{1}{r}\p_\th S_v-\p_rS_u\rt)\nn\\
=&-\int_\O (r-r_0)^\al \f{1}{r} \Phi_{\neq,\th} S_v-S_u \lt((r-r_0)^\al\Phi_{\neq,r}+\al\Phi_{\neq}(r-r_0)^\al\rt)\nn\\
=&\int_\O (r-r_0)^\al \f{1}{r} (rv) S_v+S_u \lt((r-r_0)^\al u_{\neq}+\al\Phi_{\neq}(r-r_0)^{\al-1}\rt)\nn\\
=&\int_\O  (S_u u_{\neq}+S_v v)(r-r_0)^\al+\al S_u\Phi_{\neq}(r-r_0)^{\al-1}.\label{wenergyj5fi}
\end{align}
We estimate the right-hand of \eqref{wenergyj5fi} term by term.

\noindent\underline{{\it Estimates of $\int_\O S_u u_{\neq}(r-r_0)^\al$}:}
Remembering the formula of $S_u$, we write it as follows.
\begin{align}
\int_\O S_u u_{\neq}(r-r_0)^\al=&\underbrace{\int_\O (u^a u_\th+v^a r u_r) u_{\neq}(r-r_0)^\al}_{K_{21}}+\underbrace{\int_\O (u u^a_\th+uv^a) u_{\neq}(r-r_0)^\al}_{K_{22}}\nn\\
                             &+\underbrace{\int_\O (vr u^a_r+vu^a) u_{\neq}(r-r_0)^\al.}_{K_{23}} \label{wenergyj5se}
\end{align}
Using integration by parts and the incompressibility, we have
\begin{align}
|K_{21}|=&\lt|\int_\O (u^a u_{\neq,\th}+v^a r u_{\neq,r}) u_{\neq}(r-r_0)^\al+\int_\O v^a r u_{0,r} u_{\neq}(r-r_0)^\al\rt|\nn\\
      =&\lt|-\int_\O \f{\al}{2} u^2_{\neq} rv^a(r-r_0)^{\al-1} +\int_\O v^a r u_{0,r} u_{\neq0}(r-r_0)^\al\rt|  \nn\\
      \ls & \al\int_\O r|v^a| u^2_{\neq}(r-r_0)^{\al-1} +\int_\O r|v^a| |u_{0,r} u_{\neq}|(r-r_0)^\al.\nn
\end{align}
Using the estimate \eqref{vadetail2}, Cauchy inequality and Hardy inequality \eqref{hardy2}, we see that
\begin{align}
  &|K_{21}|\ls \dl r_0 \int_\O u^2_{\neq}(r-r_0)+\e^2\dl r_0\int_\O (r-r_0)^{\al-2} u^2\nn\\
    \ls& \dl r_0 \int_\O u^2_{\neq}(r-r_0)+\e^2\dl r_0\int_\O (r-r_0)^{\al} u^2_{r}. \label{wenergyj5th}
\end{align}
Using estimates in \eqref{uadetail3} and \eqref{vadetail2}, Cauchy inequality, Poincar\'{e} inequality with zero mean and Hardy inequality \eqref{hardy2}, we obtain that
\begin{align}
|K_{22}|\ls&\e\dl r_0\int \lt(1+\f{r-r_0}{r_0+1-r} \rt)u (r-r_0)^\al u_{\neq}\nn\\
     \ls &\dl r_0 \int (r-r_0)u^2_{\neq}+ \e^2\dl r_0 \lt(u^2+\f{(r-r_0)^2u^2}{(r_0+1-r)^2} \rt) (r-r_0)^{\al-2}\nn\\
     \ls &\dl r_0 \int (r-r_0)u^2_{\th}+ \e^2\dl r_0 \int u^2_r(r-r_0)^{\al}. \label{wenergyj5fo}
\end{align}

Using the estimate \eqref{uadetail1}, \eqref{uadetail2}, we first have
\bes
\f{|u^a_r+u^a|}{r}=2A+\f{|(u^a-u^e)_r|+|u^a-u^e|}{r}\ls 2A +\dl+\f{(r-r_0)}{r_0+1-r},
\ees
Then by using Cauchy inequality, Hardy inequality, Poincar\'{e}, inequality and the incompressibility, we see that
we have
\begin{align}
|K_{23}|\ls& \lt|\int_\O  \lt(2A +\dl+\f{\chi(r)}{r_0+1-r} \rt)|rv| |u_{\neq}|(r-r_0)^\al\rt|\nn\\
   \ls & \int_\O (rv)^2(r-r_0)^{\al-1}+ \int_\O u^2_{\neq}(r-r_0)+\int_\O\f{((r-r_0)rv)^2}{(r_0+1-r)^2} \nn\\
       \ls & \al^{-2}\int_\O (r-r_0)^{\al+1}(rv)^2_r+ \int_\O u^2_{\th}(r-r_0)\ls \al^{-2}\int_\O u^2_{\th}(r-r_0).\label{wenergyj5fiv}
\end{align}
Inserting \eqref{wenergyj5th}, \eqref{wenergyj5fo} and \eqref{wenergyj5fiv} into \eqref{wenergyj5se}, wee achieve that
\be\label{wenergyj5six}
\lt|\int S_u u_{\neq}(r-r_0)^\al\rt|\ls \al^{-2} r_0\int_\O u^2_{\th}(r-r_0)+ \e^2\dl r_0 \int_\O u^2_{\neq,r}(r-r_0)^\al.
\ee
\noindent\underline{{\it Estimates of $\int_\O S_v v_{\neq}(r-r_0)^\al$}:} Remembering the formula of $S_v$, we write it as follows.
\begin{align}
\int_\O S_v v(r-r_0)^\al=&\underbrace{\int_\O (u^a v_\th+v^a r v_r) v(r-r_0)^\al}_{K_{24}}+\underbrace{\int_\O (u v^a_\th-2uu^a) v(r-r_0)^\al}_{K_{25}}\nn\\
                             &+\underbrace{\int_\O vr v^a_r v (r-r_0)^\al}_{K_{26}}.\label{wenergyk2fir}
\end{align}
Using integration by parts, the incompressibility, estimate \eqref{vadetail2} and Poincar\'{e} inequality with zero mean, we have
\begin{align}
|K_{24}|=&\f{\al}{2} \lt|\int_\O v^ar v^2(r-r_0)^{\al-1}\rt|\ls \e\dl r_0\int_\O (r-r_0)^\al v^2_{\th}\ls \dl r_0\int_\O (r-r_0) v^2_{\th}.  \label{wenergyk2sec}
\end{align}
First we decompose $K_{25}$ as the following basic estimate
\begin{align}
|K_{25}|\ls \lt|\int u^e(r) uv(r-r_0)^\al\rt|+\int (|v^a_\th|+|u^a-u^e(r)|)|uv|(r-r_0)^\al\nn\\
   =\lt|\int u^e(r) u_{\neq}v(r-r_0)^\al\rt|+\int (|v^a_\th|+|u^a-u^e(r)|)|uv|(r-r_0)^\al. \nn
\end{align}
Then using \eqref{vadetail1}, \eqref{uadetail3}, we see that
\bes
|v^a_\th|+|u^a-u^e(r)|\ls \e\dl r_0+\e\dl r_0\f{r-r_0}{r_0+1-r}.
\ees
Then by using Cauchy inequality, Hardy inequality, and Poincar\'{e} inequality with zero mean, we obtain
\begin{align}
|K_{25}|\ls &r_0\int_\O  (u^2_{\neq}+v^2)(r-r_0)^{\al+1}+\dl r_0\int_\O v^2 (r-r_0)+ \e^2\dl r_0 \int \lt(u^2+\f{(r-r_0)^2u^2}{(r_0+1-r)^2}\rt)(r-r_0)^{\al-2}\nn\\
  \ls &r_0\int_\O  (u^2_{\th}+v^2_\th)(r-r_0)+ \e^2\dl r_0 \int u^2_r(r-r_0)^{\al}.  \label{wenergyk2fou}
\end{align}

Using estimate \eqref{uadetail1} and \eqref{vadetail2}, we see that $|rv^a_r|=|u^a_\th+v^a|\ls \e\dl r_0(r-r_0)$, which indicates that

\begin{align}
|K_{26}|\ls & \dl r_0\int_\O v^2 (r-r_0)^{\al+1}\ls \dl r_0\int_\O v^2_\th (r-r_0). \label{wenergyk2fiv}
\end{align}
Inserting \eqref{wenergyk2sec}, \eqref{wenergyk2fou} and \eqref{wenergyk2fiv} into \eqref{wenergyk2fir}, we obtain that

\begin{align}
\lt|\int_\O S_v v(r-r_0)^\al\rt|=r_0\int_\O  (u^2_{\th}+v^2_\th)(r-r_0)+ \e^2\dl r_0 \int u^2_r(r-r_0)^{\al}.   \label{wenergyk2six}
\end{align}
\noindent\underline{{\it Estimates of $\al\int_\O S_u \Phi_{\neq}(r-r_0)^{\al-1}$}:} Using the representation of $S_u$, we first have
\begin{align}
\al\int_\O S_u \Phi_{\neq}(r-r_0)^{\al-1}=&\al\int_\O \lt[(uu^a)_\th+ v^a (ru)_r+v(ru^a)_r  \rt]\Phi_{\neq}(r-r_0)^{\al-1}.\nn
\end{align}

Using integration by parts, estimate \eqref{uadetail3}, Cauchy inequality, Hardy inequality, Poincar\'{e} inequality with zero mean and the incompressibility, we see that

\begin{align}
&\lt|\al\int_\O (uu^a)_\th\Phi_{\neq}(r-r_0)^{\al-1}\rt|=\lt |\al\int_\O uu^a rv(r-r_0)^{\al-1}\rt|\nn\\
\leq& \lt|\al\int_\O u_{\neq}u^e(r) rv(r-r_0)^{\al-1}\rt|+\lt|\al\int_\O u(u^a-u^e(r)) rv(r-r_0)^{\al-1}\rt|\nn\\
\ls & r_0\int_\O  u^2_{\neq}(r-r_0)+  r_0\int_\O (rv)^2(r-r_0)^{\al-1}+ \e^2\dl \al^2 r_0 \int \lt(u^2+\f{(r-r_0)^2u^2}{(r_0+1-r)^2}\rt)(r-r_0)^{\al-1}\nn\\
  \ls & \al^{-2}r_0\int_\O  (u^2_{\th}+v^2_\th)(r-r_0)+ \e^2\dl  r_0 \int u^2_r(r-r_0)^{\al}.  \label{wenergyk2sev}
\end{align}

Using estimate \eqref{vadetail2}, Cauchy inequality, Hardy inequality, Poincar\'{e} inequality with zero mean and and the incompressibility, we see that

\begin{align}
&\lt|\al \int_\O v^a (ru)_r\Phi_{\neq}(r-r_0)^{\al-1}\rt|\ls \e\dl\al \lt |\int_\O(ru)_r\Phi_{\neq}(r-r_0)^{\al}\rt|\nn\\
\ls& \e\dl\al r_0 \lt |\int_\O(u_r+\f{u}{r})\Phi_{\neq}(r-r_0)^{\al}\rt|\nn\\
\ls & \e^2 \dl r_0\int_\O  (u^2_{r}+u^2)(r-r_0)^{\al}+  \dl r_0\int_\O \Phi^2_{\neq}(r-r_0)^{\al}\nn\\
  \ls & \e^2 \dl r_0\int_\O  (u^2_{r}+u^2)(r-r_0)^{\al}+\dl r_0\int_\O (rv)^2(r-r_0)^{\al} \nn\\
  \ls & \e^2 \dl r_0\int_\O  u^2_{r}(r-r_0)^{\al}+\dl r_0\int_\O u^2_\th(r-r_0).  \label{wenergyk2eig}
\end{align}

Using estimate \eqref{uadetail2}, Cauchy inequality, Hardy inequality, Poincar\'{e} inequality with zero mean and the incompressibility, we see that

\begin{align}
&\lt|\al\int_\O v(ru^a)_r\Phi_{\neq}(r-r_0)^{\al-1}\rt|=\lt|\al\int_\O v(ru^e(r))_r\Phi_{\neq}(r-r_0)^{\al-1}\rt|+\lt|\al\int_\O v(r(u^a-u^e(r)))_r\Phi_{\neq}(r-r_0)^{\al-1}\rt|\nn\\
\ls& \lt|\al\int_\O rv\Phi_{\neq}(r-r_0)^{\al-1}\rt|+ \e\dl r_0\lt|\al\int_\O \lt(\f{1}{r-r_0}+\f{(r-r_0)^2}{(r_0+1-r)^2}\rt)rv\Phi_{\neq}(r-r_0)^{\al-1}\rt|\nn\\
\ls & \int_\O  (rv)^2(r-r_0)^{\al-1}+ \e^2\dl \al^2 r_0 \int \lt(\Phi^2_{\neq}+\f{(r-r_0)^4\Phi^2_{\neq}}{(r_0+1-r)^4}\rt)(r-r_0)^{\al-3}\nn\\
  \ls & \al^{-2}r_0\int_\O  u^2_{\th}(r-r_0)+ \e^2\dl  r_0 \int u^2_r(r-r_0)^{\al}.  \label{wenergyk2nin}
\end{align}

Combining estimates \eqref{wenergyk2sev}, \eqref{wenergyk2eig} and \eqref{wenergyk2nin}, we see that
\be\label{wenergyk2ten}
\lt|\al\int_\O S_u \Phi_{\neq}(r-r_0)^{\al-1}\rt|\ls \al^{-2}r_0\int_\O (u^2_{\th}+v^2_\th)(r-r_0)+ \e^2\dl  r_0 \int u^2_r(r-r_0)^{\al}.
\ee
Combining estimates in \eqref{wenergyj5six}, \eqref{wenergyk2six} and \eqref{wenergyk2ten}, we see that
\be\label{wenergyk2ele}
|K_2|\ls \al^{-2} r_0\int (r-r_0)(u^2_\th+v^2_\th)+ \e^2\dl r_0\int (r-r_0)^\al u^2_r.
\ee
\noindent\underline{{\bf Estimates of $K_3$}}

The same as \eqref{wenergyj5fi}, we have
\begin{align}
|K_3|=&\lt|\int_\O (r-r_0)^\al \Phi_{\neq} \lt(\f{1}{r}\p_\th R_v-\p_rR_u\rt)\rt|\nn\\
=&\lt|\int_\O  (R_u u_{\neq}+R_v v)(r-r_0)^\al+\al R_u\Phi_{\neq}(r-r_0)^{\al-1}\rt|.\label{wenergyk3}
\end{align}
Inserting \eqref{wenergyj1th}, \eqref{wenergyk2ele} and \eqref{wenergyk3} into \eqref{wenergyj1zer}, we can obtain \eqref{estiunonzero}.
\begin{remark}
The estimate for the term $K_{23}$ in \eqref{wenergyj5fiv}term show that we need weighted energy estimate and it also indicates that the energy and the positive estimates must perform the same order. Later we will see that the nonlinear term determine how many order derivatives we need to perform. \qed
\end{remark}

\subsection{Non-weighted energy estimates}\label{sub3.3}

Next we give the non-weighted estimates for the nonzero frequency of $u$. We have the following Proposition.
\begin{proposition}\label{proplinearnonw}
 Let $(u,v)$ be a smooth solution of (\ref{errorequ}), then there exist $\epsilon_0>0, \dl_0>0$ such that for any $\epsilon\in (0,\epsilon_0), \dl\in(0,\dl_0)$, there holds
\begin{align}\label{estiunonw}
&\e^4\int_{\O} \lt[u^2_{\neq,r}+v^2_r+ r^{-2}_0(u^2_\th+v^2_\th)\rt]\nn\\
\ls&  \int_{\O} (r-r_0) \lt(u^2_{\th}+v^2_\th\rt)+\e^4\dl \int u^2_{0,r}+C\e^2 r^{-1}_0\lt|\int_{\O} R_u  u_{\neq}+ R_v  v\rt|.
\end{align}
\end{proposition}

\pf Multiplying \eqref{errorequcut1} by $\e^2\Phi_{\neq}$ to obtain
\be \label{nonwfir}
\begin{aligned}
 &\underbrace{-\epsilon^4 \int r\Delta\o \Phi_{\neq} }_{L_1}+\underbrace{\e^2\int\left(\frac{1}{r} \partial_\th S_v-\p_rS_u\right)\Phi_{\neq}}_{L_2}=\underbrace{\e^2\int\left(\frac{1}{r} \partial_\th R_v-\partial_r R_u\right)\Phi_{\neq}}_{L_3}.
\end{aligned}
\ee
\underline{{\bf\noindent Estimates of $L_1$}}

Essentially the same as and much simpler than \eqref{wenergyj1fi}, we can see that
\begin{align}
L_1=&\e^4\int r\o^2_{\neq}=\e^4\int r(u^2_{\neq,r}+v^2_r)+\f{(v_\th-u_{\neq})^2+(u_\th+v)^2 }{r}\nn\\
   \geq& \e^4\int r(u^2_{\neq,r}+v^2_r)+\f{v^2_\th+u_\th^2 }{2r}-C(r-r_0)(u^2_\th+v^2_\th).  \label{nonwl1fir}
\end{align}

\noindent\underline{{\bf Estimates of $L_2$}}

Using integration by parts and the relation among $\Phi_{\neq}$ and $rv,\, u_{\neq}$, we have
\begin{align}
&L_2=\e^2\int_\O (r-r_0)^\al \Phi_{\neq} \lt(\f{1}{r}\p_\th S_v-\p_rS_u\rt)=\e^2\int_\O  (S_u u_{\neq}+S_v v).\nn
\end{align}
Essentially the same as and much simpler than \eqref{wenergyj5six} and \eqref{wenergyk2six}, we can see that
\be\label{nonwl2sec}
\lt|L_2\rt|=\lt|\e^2\int_\O  (S_u u_{\neq}+S_v v)\rt|\ls  r_0\int (r-r_0)(u^2_\th+v^2_\th)+ \e^4 \dl r_0\int u^2_r.
\ee
\noindent\underline{{\bf Estimates of $L_3$}}

Direct integration by parts indicates that
\begin{align}
|L_3|=\e^2\lt|\int_\O  (R_u u_{\neq}+R_v v)\rt|.\label{nonwl3}
\end{align}
Inserting \eqref{nonwl1fir}, \eqref{nonwl2sec} and \eqref{nonwl3} into \eqref{nonwfir}, we can obtain \eqref{estiunonw}. \qed

{\bf\noindent Proof of Proposition \ref{proplinearh1}}:

Combining the estimate \eqref{estipositive} in Proposition \ref{proppositive} and the estimate \eqref{estilinearwei} in Proposition \ref{proplinearenergy}, we obtain that
\begin{align}
&\int (r-r_0)(u^2_\th+v^2_\th)+ \e^2\int_{\O} u^2_{0,r}+(r-r_0)^\al (u^2_{\neq,r}+v^2_r)\nn\\
\leq& Cr^{-1}_0\lt|\int_{\O} (R_u u_{\neq}+R_v v)\rt|+Cr^{-1}_0\lt|\int_{\O} (r-r_0)^{\al-1}R_u \Phi_{\neq}\rt| \nn\\
&+Cr^{-1}_0\lt|\int_{\O} R_u u_0\rt|+r^{-1}_0\lt|\int_{\O} (R_u u_{\th}+R_v v_{\th})\rt|. \label{linearweight}
\end{align}

At last, by combining the above estimate \eqref{linearweight} and \eqref{estiunonw} in Proposition \ref{proplinearnonw}, we arrive at \eqref{linearstabilityh1} in Proposition \ref{proplinearh1}. \qed

\section{$H^3$ stability estimates of the linear error equations}\label{sec4}

In order to obtain the $L^\i$ estimate and close the nonlinear energy estimates, we will present the following $H^3$ estimates.
\begin{proposition}\label{proplinearh3}
 Let $(u,v)$ be a smooth solution of (\ref{errorequ}).Then there exist $\epsilon_0>0, \dl_0>0$ and a small constant $\al$, such that for any $\epsilon\in (0,\epsilon_0), \dl\in(0,\dl_0)$, there holds
\begin{align}\label{linearh3}
&\e^{10}\lt\{\sum_{i+j=2}\|\p^i_\th\p^j_r(u_r,v_r)\|^2+ r^{-2}_0 \|(u_{\th\th\th},u_{\th\th\th})\|^2+\|(u_{\th\th\th},u_{\th\th\th})\s{r-r_0}\|^2\rt\}\nn\\
&\leq C\e^6r^{-1}_0\lt|\int_{\O} \lt( R_u \p^5_\th u+R_v  \p^5_\th v\rt)\rt|+\e^4\|(u_r,v_r)\|^2\nn\\
&+C\e^6r^{-1}_0\lt|\int_{\O} \lt( R_u  \p^4_\th u+R_v  \p^4_\th v\rt)\rt|+ C\e^6r^{-1}_0\lt|\int_{\O}  R_u \p^3_\th(rv)(r-r_0)^{\al-1}\rt|\nn\\
&+\e^4\lt(r^{-2}_0\lt\|(R_u,\p_\th R_u,\p_r R_u)\rt\|^2+r^{-4}_0\lt\|(R_v, \p_\th R_v)\rt\|^2\rt).
\end{align}
\end{proposition}

Proof of Proposition \ref{proplinearh3} contain the tangential derivatives' estimates whose proof is the same as \eqref{linearstabilityh1} in Proposition \ref{proplinearh1}, and the normal derivatives' estimates which will be achieved by the Bogovski Lemma and the vorticity equation.

\subsection{Tangential derivatives' estimates}

The similar estimate as \eqref{linearstabilityh1} in Proposition \ref{proplinearh1}, we have the following Proposition.
\begin{proposition}\label{proplinearh3tan}
 Let $(u,v)$ be a smooth solution of (\ref{errorequ}). Then there exist $\epsilon_0>0, \dl_0>0$ and a small constant $\al$, such that for any $\epsilon\in (0,\epsilon_0), \dl\in(0,\dl_0)$, there holds
\begin{align}\label{linearh3tan}
&\e^4\lt\{\int_{\O}\lt(u^2_{\th\th r}+v^2_{\th\th r}\rt)+\lt(u^2_{\th\th\th}+v^2_{\th\th\th}\rt)(r-r_0)+r^{-2}_0\lt(u^2_{\th\th\th}+v^2_{\th\th\th}\rt)\rt\}\nn\\
&\leq Cr^{-1}_0\lt|\int_{\O} \lt( R_u \p^5_\th u+R_v  \p^5_\th v\rt)\rt|\nn\\
&+Cr^{-1}_0\lt|\int_{\O} \lt( R_u  \p^4_\th u+R_v  \p^4_\th v\rt)\rt|+ Cr^{-1}_0\lt|\int_{\O}  R_u \p^3_\th(rv)(r-r_0)^{\al-1}\rt|.
\end{align}
\end{proposition}

Since the proof is almost the same as that for Proposition \ref{proplinearh1}, we omit the details. \qed

\subsection{Normal derivatives' estimates} The normal derivatives' estimates is first by applying the Bogovskii Lemma to obtain the estimates of $u_{rr}$ and $u_{rr\th}$. Then with the help of the vorticity equation, we obtain estimates of $u_{rrr}$.

\subsubsection{Estimates of $u_{rr},  u_{rr\th}$}  We have the following Lemma.

\begin{lemma}\label{lemmnormalurr}
 Let $(u,v)$ be a smooth solution of (\ref{errorequ}). Then there exist $\epsilon_0>0, \dl_0>0$ and a small constant $\al$, such that for any $\epsilon\in (0,\epsilon_0), \dl\in(0,\dl_0)$, there holds
\begin{align}\label{urr7}
&\e^8\|(u_{rr},v_{rr}, u_{rr\th},v_{rr\th},v_{rrr})\|^2\nn\\
\ls & \e^4\lt( \|(u_{\th\th\th}, v_{\th\th\th})(r-r_0)\|^2+\|(u_r, v_{r})\|^2+ \|(u_{r\th\th}, v_{r\th\th})\|^2+r^{-2}_0\|(u_{\th\th\th},v_{\th\th\th})\|^2\rt)\nn\\
                       &+\e^4\lt(r^{-2}_0\|(R_u,\p_\th R_u)\|^2+r^{-4}_0\|(R_v, \p_\th R_v) \|^2\rt).
\end{align}
\end{lemma}

From \eqref{errorequ}$_1$, we need  the weighted $L^2$ estimate of $p_{\th}$, which can be obtained by applying the following Bogovskii Lemma.

\begin{lemma}\label{lembogov}
 Define the domain $\O:=\bT \times[r_0,r_0+1]$, then for $f\in L^2(\O)$ with
\bes
\int_{\O} f(\th,r)d\th dr=0,
\ees
there exists a vector $V=(V^\th,V^r)\in H^1_0(\O)$, s. t.
\bes
\bali
&\partial_\theta V^\th+\partial_r V^r=f,\q \|\nabla V\|_{L^2(\O)}\leq  C \|f\|_{L^2(\O)}.
\eali
\ees
where $\na=(\p_\th,\p_r)$ and the constant $C$ is independent of $r_0$.
\end{lemma}

Here we only stated a special case of Lemma III.3.1 of \cite{Galdi:2011} which works for general domains and the constant $C$ depends on the ratio of the diameter and inner radius of the domain. For the above domain $\O$, the diameter and inner radius are fixed, so we have an absolute constant. Note that the Bogovskii function may not be unique. For our purpose, we will choose and fix one for the relevant domain.

{\bf\noindent Proof of Lemma \ref{lemmnormalurr}} By taking $\th$ derivative on \eqref{errorequ}$_{1,2}$, we obtain that
\be\label{urr}
\lt\{
\begin{aligned}
& -\epsilon^2\left(r u_{rr\th}+u_{r\th}+\frac{u_{\th\th\th}}{r}-\frac{u_{\th}}{r}+\frac{2}{r} v_{\th\th}\right)+\p^2_\th p+\partial_\th S_u=\p_\th R_u \\
& -\e^2\left(v_{rr \theta}+\frac{v_{r  \theta}}{r}+\frac{v_{\theta \theta  \theta}}{r^2}-\frac{2}{r^2} u_{\th \theta}\right)+\p_r\p_\th p+\f{1}{r}\partial_\th S_v=\f{1}{r}\p_\theta R_v \\
\end{aligned}
\rt.
\ee
Using Lemma \ref{lembogov} for $f=p_\th$, there exists a vector $V=(V^\th,V^r)\in H^1_0(\O)$, s. t.
\be\label{bogov2}
\bali
&\partial_\theta V^\th+\partial_r V^r=p_\th,\q \|\nabla V\|_{L^2(\O)}\leq  C \|p_\th\|_{L^2(\O)}.
\eali
\ee

Multiplying \eqref{urr}$_1$ by$-V^\th $, \eqref{urr}$_2$ by$-V^r $, and integrations on $\O$ to obtain
\begin{align}
& \int_{\O} p^2_{\theta}=\int p_\th  \lt(\partial_\theta V^\th+\partial_r V^r\rt)=  \int_{\O} \left(-p_{\th\th}V^{\th}-p_{r\theta} V^r\right)\nn\\
=& -\epsilon^2\int_{\O} \left(r u_{rr\th}+u_{r\th}+\frac{u_{\th\th\th}}{r}-\frac{u_{\th}}{r}+\frac{2}{r} v_{\th\th}\right) V^\th \nn\\
 &-\epsilon^2\int_{\O} \left(v_{rr \theta}+\frac{v_{r \theta}}{r}+\frac{v_{\theta \theta  \theta}}{r^2}-\frac{2}{r^2} u_{\th \theta}\right) V^r \label{ptheta}\\
& + \int_{\O} \lt(\partial_\theta S_u V^\th +\f{1}{r}\partial_\theta S_v V^r\rt)-\int_{\O} \lt(\partial_\theta R_u V^\th+\f{1}{r}\partial_\theta R_v V^r\rt).\nn
\end{align}
By using integration by parts, Hardy inequality, Poincar\'{e} inequality and \eqref{bogov2}, we have
\begin{align}
&-\e^2\int_{O_k}\left(r u_{rr\th}+u_{r\th}+\frac{u_{\th\th\th}}{r}-\frac{u_{\th}}{r}+\frac{2}{r} v_{\th\th}\right) V^\th  \nn\\
\ls& \e^2\lt(\|r u_{r\th}\|_{L^2(\O)}+\|r^{-1}(u_{\th\th},v_{\th\th})\|_{L^2(\O)} \rt)\|\nabla V^\th\|_{L^2(\O)}\nn\\
\ls& \e^2\lt(\|r u_{r\th}\|_{L^2(\O)}+\|r^{-1}(u_{\th\th},v_{\th\th})\|_{L^2(\O)} \rt)\|p_\th\|_{L^2(\O)}. \label{ptheta1}
\end{align}
The same, we can obtain that
\begin{align}
&-\e^2\int_{O_k}\left( v_{rr\th}+\f{v_{r\th}}{r}+\frac{v_{\th\th\th}}{r}-\frac{v_{\th}}{r}-\frac{2}{r} u_{\th\th}\right) V^r \nn\\
\ls& \e^2\lt(\|v_{r\th}\|_{L^2(\O)}+\|r^{-1}(u_{\th\th},v_{\th\th})\|_{L^2(\O)} \rt)\|\nabla V^r\|_{L^2(\O)}\nn\\
\ls& \e^2\lt(\|v_{r\th}\|_{L^2(\O)}+\|r^{-1}(u_{\th\th},v_{\th\th})\|_{L^2(\O)} \rt)\|p_\th\|_{L^2(\O)}.  \label{ptheta2}
\end{align}
By using integration by parts, Hardy inequality, Poincar\'{e} inequality and \eqref{bogov2}, we have
{\small
\be \label{ptheta3}
\begin{array}{ll}
\begin{aligned}
&\int_{\O} \lt(\partial_\theta S_u V^\th +\f{1}{r}\partial_\theta S_v V^r\rt)\\
\ls& \|S_u \|_{L^2(\O)}\|\p_\th V^\th\|_{L^2(\O)}+\|r^{-1}  S_v \|_{L^2(\O)}\|\p_\th V^r\|_{L^2(\O)}\\
\ls&  \lt(\|S_u \|_{L^2(\O)}+\|r^{-1} S_v \|_{L^2(\O)}\rt)\|p_\th \|_{L^2(\O)}.
\end{aligned}
&
\begin{aligned}
&\int_{\O} \lt(\partial_\theta R_u V^\th +\f{1}{r}\partial_\theta R_v V^r\rt)\\
\ls& \|R_u \|_{L^2(\O)}\|\p_\th V^\th\|_{L^2(O_k)}+\|r^{-1}  R_v \|_{L^2(\O)}\|\p_\th V^r\|_{L^2(\O)}\\
\ls&  \lt(\|R_u \|_{L^2(\O)}+\|r^{-1} R_v \|_{L^2(\O)}\rt)\|p_\th \|_{L^2(\O)}.
\end{aligned}
\end{array}
\ee
}

Inserting \eqref{ptheta1}, \eqref{ptheta2} and \eqref{ptheta3} into \eqref{ptheta}, we have
\begin{align}
\|p_\th \|^2_{L^2(\O)}\ls& \e^4\lt(r^2_0\| (v_{r\th},u_{r\th})\|^2_{L^2(\O)}+r^{-2}_0\|(u_{\th\th},v_{\th\th})\|^2_{L^2(\O)} \rt) \nn\\
                    &+\lt(\|S_u \|^2_{L^2(\O)}+r^{-2}_0\| S_v \|^2_{L^2(\O)}\rt) +\lt(\|R_u\|^2_{L^2(\O)}+r^{-2}_0\| R_v \|^2_{L^2(\O)}\rt).\label{ptheta4}
\end{align}
Then from the equation \eqref{errorequ}$_1$, we have
\begin{align}
\e^2 u_{rr}=-\epsilon^2 \left(r^{-1}u_{r}+\frac{u_{\th\th}}{r^2}-\frac{u}{r^2}+\frac{2}{r^2} v_{\th}\right)+r^{-1}\p_\th p+ r^{-1}S_u- r^{-1} R_u,\nn
\end{align}
which after using Hardy inequality, indicates that
\begin{align}
\e^4\| u_{rr}\|^2\ls \e^4 \lt(r^{-2}_0\| u_{r}\|^2+r^{-4}_0\|(u_{\th\th},v_{\th})\|^2\rt)+r^{-2}_0\|p_\th\|^2+r^{-2}_0\|S_u \|^2+r^{-2}_0\|R_u \|^2. \label{urr2}
\end{align}
Combining \eqref{ptheta4} and \eqref{urr2}, we can arrive at
\begin{align}
\e^8\|u_{rr}\|^2\ls& \e^8 \lt(\| u_{r}\|^2+\| (u_{r\th},v_{r\th})\|^2+r^{-4}_0\|(u_{\th\th},v_{\th\th})\|^2\rt)\nn\\
                          &+\e^4\lt(r^{-2}_0\|S_u \|^2+r^{-4}_0\|S_v \|^2\rt) +\e^4\lt(r^{-2}_0\|R_u\|^2+r^{-4}_0\| R_v \|^2\rt). \label{urr3}
\end{align}
From the representation of $S_u$ and $S_v$ in \eqref{sudef} and \eqref{svdef} using estimates in Lemma \ref{lemdetail}, it is not hard to obtain that
\be
r^{-2}_0\|S_u \|^2+r^{-4}_0\|S_v \|^2\ls \|(u_\th, v_\th)(r-r_0)\|^2+\|(u_r, v_{r})\|^2. \label{urr4}
\ee
Combining \eqref{urr3} and \eqref{urr4}, we see arrive at
\begin{align}
\e^8\|u_{rr}\|^2\ls &\e^8 \lt(\| u_{r}\|^2+\| (u_{r\th},v_{r\th})\|^2+r^{-4}_0\|(u_{\th\th},v_{\th\th})\|^2\rt)\nn\\
                       & +\e^4\lt( \|(u_\th, v_\th)(r-r_0)\|^2+\|(u_r, v_{r})\|^2\rt)\nn\\
                       &+\e^4\lt(r^{-2}_0\|R_u\|^2+r^{-4}_0\| R_v \|^2\rt). \label{urr5f}
\end{align}
The incompressibility indicates that
\be
 v_{rr}=\lt(\f{-u_\th-v}{r}\rt)_r=\lt[-\f{u_{r\th}+v_r}{r}+\f{u_\th+v}{r^2}\rt]. \label{urr5s}
\ee
Combining \eqref{urr5f} and \eqref{urr5s}, we can obtain that
\begin{align}
\e^8\|(u_{rr},v_{rr})\|^2\ls &\e^8 \lt(\| (u_{r\th},v_{r\th})\|^2+r^{-4}_0\|(u_{\th\th},v_{\th\th})\|^2\rt)\nn\\
                       & +\e^4\lt( \|(u_\th, v_\th)(r-r_0)\|^2+\|(u_r, v_{r})\|^2\rt)\nn\\
                       &+\e^4\lt(r^{-2}_0\|R_u\|^2+r^{-4}_0\| R_v \|^2\rt). \label{urr5}
\end{align}

By applying $\p_\th$ to \eqref{urr} and the same estimate as \eqref{urr5}, we can obtain that
\begin{align}
\e^8\|(u_{rr\th},v_{rr\th})\|^2\ls &\e^8 \lt(\| (u_{r\th\th},v_{r\th\th})\|^2+r^{-4}_0\|(u_{\th\th\th},v_{\th\th\th})\|^2\rt)\nn\\
                       & +\e^4\lt( \|(u_{\th\th}, v_{\th\th})(r-r_0)\|^2+\|(u_r, v_{r})\|^2+ \|(u_{r\th}, v_{r\th})\|^2\rt)\nn\\
                       &+\e^4\lt(r^{-2}_0\|\p_\th R_u\|^2+r^{-4}_0\| \p_\th R_v \|^2\rt). \label{urr6}
\end{align}

Combining \eqref{urr5} and \eqref{urr6} and using Poincar\'{e} inequality with zero mean and the incompressibility, we achieve that
\begin{align}
\e^8\|(u_{rr},v_{rr}, u_{rr\th},v_{rr\th},v_{rrr})\|^2\ls & \e^4\lt( \|(u_{\th\th\th}, v_{\th\th\th})(r-r_0)\|^2+\|(u_r, v_{r})\|^2+ \|(u_{r\th\th}, v_{r\th\th})\|^2+r^{-2}_0\|(u_{\th\th\th},v_{\th\th\th})\|^2\rt)\nn\\
                       &+\e^4\lt(r^{-2}_0\|(R_u,\p_\th R_u)\|^2+r^{-4}_0\|(R_v, \p_\th R_v) \|^2\rt),\nn
\end{align}
which is \eqref{urr7}.

\subsubsection{Estimates of $u_{rrr}$} we have the following Lemma.
\begin{lemma}\label{lemmnormalurrr}
 Let $(u,v)$ be a smooth solution of (\ref{errorequ}). Then there exist $\epsilon_0>0, \dl_0>0$ and a small constant $\al$, such that for any $\epsilon\in (0,\epsilon_0), \dl\in(0,\dl_0)$, there holds
\begin{align}
\e^{10}\|u_{rrr}\|^2\ls &\e^4\lt( \|(u_{\th\th\th}, v_{\th\th\th})(r-r_0)\|^2+\|(u_r, v_{r})\|^2+ \|(u_{r\th\th}, v_{r\th\th})\|^2+r^{-2}_0\|(u_{\th\th\th},v_{\th\th\th})\|^2\rt)\nn\\
                  &+\e^4\lt(r^{-2}_0\lt\|(R_u,\p_\th R_u,\p_r R_u)\rt\|^2+r^{-4}_0\lt\|(R_v, \p_\th R_v)\rt\|^2\rt). \label{urrr}
\end{align}
\end{lemma}

\pf From the representation of $\o$, direct calculation implies that
\bes
-\Dl \o=u_{rrr}+\f{3}{r}u_{rr}+\f{v_{rr\th}}{r}-\f{\o_r}{r}+\f{\o_{\th\th}}{r^2}.
\ees
Inserting the above equality into \eqref{errorequcut1}, we obtain that
\begin{align*}
\e^2 \p^3_r u=& -\e^2\lt[\f{3}{r}u_{rr}+\f{v_{rr\th}}{r}-\f{\o_r}{r}+\f{\o_{\th\th}}{r^2}\rt]\\
 &-\frac{1}{r^2}\p_\th S_v+\f{1}{r}\p_r S_u+\frac{1}{r^2} \partial_\theta R_v-\f{1}{r}\partial_r R_u.
\end{align*}
By using estimates in Lemma \ref{lemdetail}, it is not hard to deduce that
\begin{align}
\lt\|-\frac{1}{r^2}\p_\th S_v+\f{1}{r}\p_r S_u\rt\|\ls \|(u_r,u_{r\th}, \e u_{rr},r^{-1}_0v_{\th\th})\|.\nn
\end{align}
Then from the representation of $\o$, Poincar\'{e} inequality, Hardy inequality, we can obtain that
\begin{align}
\e^2\|u_{rrr}\|\ls &\e^2\|u_r, u_{rr},u_{r\th\th},v_{r\th}, r^{-1}_0 v_{\th\th\th}\|+\|(u_r,u_{r\th}, \e u_{rr},r^{-1}_0v_{\th\th})\|\nn\\
                  &+r^{-1}_0\lt\|\p_r R_u\rt\|+r^{-2}_0\lt\| \p_\th R_v\rt\|. \label{urrrx}
\end{align}
Multiplying \eqref{urrrx} by $\e^3$ and using estimate \eqref{urr7}, we can achieve that
\begin{align}
\e^{10}\|u_{rrr}\|^2\ls &\e^4\lt( \|(u_{\th\th\th}, v_{\th\th\th})(r-r_0)\|^2+\|(u_r, v_{r})\|^2+ \|(u_{r\th\th}, v_{r\th\th})\|^2+r^{-2}_0\|(u_{\th\th\th},v_{\th\th\th})\|^2\rt)\nn\\
                  &+\e^4\lt(r^{-2}_0\lt\|(R_u,\p_\th R_u,\p_r R_u)\rt\|^2+r^{-4}_0\lt\|(R_v, \p_\th R_v)\rt\|^2\rt).\nn
\end{align}
This is \eqref{urrr}. \qed

At last, combining estimate \eqref{urr7} in Lemma \ref{lemmnormalurr}, estimate \eqref{urrr} in Lemma \ref{lemmnormalurrr} and estimate \eqref{linearh3tan} in Proposition \ref{proplinearh3tan}, we obtain \eqref{linearh3} in Proposition \ref{proplinearh3}. \qed

\section{ Existence of the error equation and proof of the main result}\label{sec5}

Combining $\dot{H}^1$ and $\dot{H}^3$ estimates in Proposition \ref{proplinearh1} and \ref{proplinearh3}, we can close the nonlinear estimates and prove the main result of this paper, Theorem \ref{thmain}. We first give estimates of the nonlinear terms in \eqref{errorequ} and then using contract mapping to show existence of solutions to system \eqref{errorequ}.

\subsection{Estimates of nonlinear terms}

From \eqref{linearh3}, by using Cauchy inequality, Hardy inequality and the incompressibility, we see that
\begin{align}\label{nonlinear1}
&\e^{10}\lt\{\sum_{i+j=2}\|\p^i_\th\p^j_r(u_r,v_r)\|^2+ r^{-2}_0 \|(u_{\th\th\th},u_{\th\th\th})\|^2+\|(u_{\th\th\th},u_{\th\th\th})\s{r-r_0}\|^2\rt\}\nn\\
&\leq C\e^2r^{-2}_0 \lt(\lt\|\p^2_\th (R_u-R^a_u) (r-r_0)^{-1/2}\rt\|^2+\lt\|\p^2_\th (R_v-R^a_v) (r-r_0)^{-1/2}\rt\|^2\rt)+\e^4\|(u_r,v_r)\|^2\nn\\
&+\e^2r^{-2}_0\lt(\lt\|(R_u,\p_\th R_u,\p_r R_u, \p^3_\th R^a_u)\rt\|^2+\lt\|(R_v, \p_\th R_v, \p^3_\th R^a_v)\rt\|^2\rt).
\end{align}
Also from \eqref{linearstabilityh1}, by using Cauchy inequality, Hardy inequality and the incompressibility, we see that

\begin{align}\label{nonlinear2}
&\e^{4}\lt\{\|(u_r,v_r)\|^2+r^{-2}_0 \|(u_{\th},u_{\th})\|^2+\|(u_{\th},u_{\th})\s{r-r_0}\|^2\rt\}\\
&\leq C\e^{-4}r^{-2}_0 \|(R_u,R_v,\p_\th R_u, \p_\th R_v)\|^2+\e^{4}\mu \|(u_{r},v_{r})\|^2,\nn
\end{align}
for small $\mu$.

Combining \eqref{nonlinear1} and \eqref{nonlinear2}, we see that
\begin{align}\label{nonlinear3}
&\e^{10}\lt\{\sum_{i+j\leq 2}\|\p^i_\th\p^j_r(u_r,v_r)\|^2+ r^{-2}_0 \|(u_{\th\th\th},u_{\th\th\th})\|^2+\|(u_{\th\th\th},u_{\th\th\th})\s{r-r_0}\|^2\rt\}\nn\\
&\leq C\e^2r^{-2}_0 \lt(\lt\|\p^2_\th (R_u-R^a_u) (r-r_0)^{-1/2}\rt\|^2+\lt\|\p^2_\th (R_v-R^a_v) (r-r_0)^{-1/2}\rt\|^2\rt)\nn\\
&+\e^{-4}r^{-2}_0\lt(\lt\|(R_u,\p_\th R_u,\p_r R_u, \p^3_\th R^a_u)\rt\|^2+\lt\|(R_v, \p_\th R_v, \p^3_\th R^a_v)\rt\|^2\rt).
\end{align}
\subsection{Existence for the error equations}
\indent

We apply the contraction mapping theorem to prove the existence for the error equations (\ref{errorequ}).

\begin{proposition}\label{existence and error estimate of error equation}
There exist $\epsilon_0>0,\dl_0>0$ such that for any $\epsilon\in (0,\epsilon_0), \dl\in (0,\dl_0)$, the error equations (\ref{errorequ}) have a unique solution $(u,v)$ which satisfies
$$\|(u,v)\|_\infty\leq C\epsilon r_0.$$
\end{proposition}
\begin{proof}
For a smooth function $(u,v)$ which satisfies
\begin{align}\label{iterative conditon}
\left\{
\begin{array}{ll}
 u_\th+(rv)_r=0,\\[5pt]
(u,v)(\th,r)=(u,v)(\th+2\pi,r),\\[5pt]
(u,v)(\th,r_0)=(u,v)(\th,r_0+1)=0,
\end{array}
\right.
\end{align}
we consider the following linear problem:
\begin{align*}
\left\{
\begin{array}{ll}
-\epsilon^2\left(\frac{\t{u}_{\theta \theta}}{r}+r \t{u}_{r r}+\t{u}_r+\frac{2}{r} \t{v}_\theta-\frac{\t{u}}{r}\right)+\p_\th \t{ p} +S_{\t{u}}=R_{u},\\[5pt]
-\epsilon^2\left(\frac{\t{v}_{\theta \theta}}{r}+r \t{v}_{r r}+\t{v}_r-\frac{2}{r} \t{u}_\theta-\frac{\t{v}}{r}\right)+r\p_r\t{p}+S_{\t{v}}=R_v,\\[5pt]
\p_\th \t{u}+\p_r(r\t{v})=0,\\[5pt]
(\t{u},\t{v})(\th,r)=(\t{u},\t{v})(\th+2\pi,r), \\[5pt]
(\t{u},\t{v})(\th,r_0)=(\t{u},\t{v})(\th,r_0+1)=0,
 \end{array}
\right.
\end{align*}
where the transversal terms are
\begin{align*}
S_{\t{u}}:&=u^a \p_\th \t{u} +v^ar\p_r\t{u}+\t{u} \p_\th u^a+\t{v}r\p_ru^a+\t{u}v^a+\t{v}u^a,\\[5pt]
S_{\t{v}}:&=u^a \p_\th \t{v} +v^ar\p_r\t{v}+\t{u} \p_\th v^a+\t{v} r\p_r v^a-2\t{u} u^a,
\end{align*}
and the remainders are
\bes
R_u:=R_u^a-(u\p_\th+v r\p_r)u-uv,\ R_v:=R_v^a-(u\p_\th+v r\p_r)v+u^2.
\ees
Set the total energy to be
\begin{align}
\|(u,v)\|_E^2:=&\e^{10}\lt\{\sum_{i+j\leq 2}\|\p^i_\th\p^j_r(u_r,v_r)\|^2+ r^{-2}_0 \|(u_{\th\th\th},u_{\th\th\th})\|^2+\|(u_{\th\th\th},u_{\th\th\th})\s{r-r_0}\|^2\rt\}.\nn
\end{align}
From estimate \eqref{nonlinear3}, we see that
\begin{align}
\|(\t{u},\t{v})\|_E^2\ls& \e^2r^{-2}_0 \lt(\lt\|\p^2_\th (R_u-R^a_u) (r-r_0)^{-1/2}\rt\|^2+\lt\|\p^2_\th (R_v-R^a_v) (r-r_0)^{-1/2}\rt\|^2\rt)\nn\\
&+\e^{-4}r^{-2}_0\lt(\lt\|(R_u,\p_\th R_u,\p_r R_u, \p^3_\th R^a_u)\rt\|^2+\lt\|(R_v, \p_\th R_v, \p^3_\th R^a_v)\rt\|^2\rt). \label{energyfinal}
\end{align}
By using Sobolev embedding, Cauchy inequality, Hardy inequality, Poincar\'{e} inequality and the incompressibility, we can show that
\begin{align}
&\lt\|\p^2_\th (R_u-R^a_u) (r-r_0)^{-1/2}\rt\|^2+\lt\|\p^2_\th (R_v-R^a_v) (r-r_0)^{-1/2}\rt\|^2\nn\\
&+\lt(\lt\|(R_u-R^a_u,\p_\th (R_u-R^a_u),\p_r (R_u-R^a_u))\rt\|^2+\lt\|((R_v-R^a_v, \p_\th (R_v-R^a_v)\rt\|^2\rt)\ls \e^{-20}r^2_0\|(u,v)\|^4_E. \label{energyfinal1}
\end{align}
Since this is a tedious calculation, we only show the estimate of the most trouble term, $$ \lt\|\p^2_\th (R_u-R^a_u) (r-r_0)^{-1/2}\rt\|^2,$$ then the other term follows the same line.

First from the representation of $R_u-R^a_u$, we see that
\begin{align}
 &\lt\|\p^2_\th (R_u-R^a_u) (r-r_0)^{-1/2}\rt\|^2=\int \f{\lt((uu_\th+vru_r+uv)_{\th\th}\rt)^2}{r-r_0}\nn\\
 \ls& \int \f{|uu_{\th\th\th}|^2+ |u_\th u_{\th\th}|^2+ |v_\th ru_{r}|^2+ |vr u_{r\th\th}|^2+ |v_{\th\th} ru_{r}|^2+ |u_{\th\th} v|^2+ |v_{\th\th} u|^2+|u_{\th} v_\th|^2}{r-r_0}. \label{energyfinal2}
\end{align}
Using H\"{o}lder inequality and Sobolev embedding, we see that
\begin{align}
\int \f{|uu_{\th\th\th}|^2}{r-r_0}\leq \lt\|\f{u}{r-r_0}\rt\|^2_{L^\i}\lt\|u_{\th\th\th}\s{r-r_0}\rt\|^2_{L^2}\ls \lt\|u_r\rt\|^2_{L^\i}\lt\|u_{\th\th\th}\s{r-r_0}\rt\|^2_{L^2}\ls \e^{-20}\|(u,v)\|^4_E. \label{energyfinal3}
\end{align}
Using H\"{o}lder inequality, Hardy inequality and Sobolev embedding, we see that
\begin{align}
\int \f{|u_\th u_{\th\th}|^2}{r-r_0}\leq \lt\|u_\th\s{r-r_0}\rt\|^2_{L^\i}\lt\|\f{u_{\th\th}}{r-r_0}\rt\|^2_{L^2}\ls \lt\|(u_{\th\th}\s{r-r_0},u_{r\th\th})\rt\|^2_{L^\i}\lt\|u_{r\th\th}\rt\|^2_{L^2}\ls \e^{-20}\|(u,v)\|^4_E. \label{energyfinal4}
\end{align}
Using H\"{o}lder inequality, Hardy inequality and Sobolev embedding, we see that
\begin{align}
\int \f{|v_\th ru_{r\th}|^2}{r-r_0}\leq r^2_0\lt\|\f{v_\th}{r-r_0}\rt\|^2_{L^\i_\th L^2_r}\lt\|u_{r\th}\rt\|^2_{L^2_\th L^\i_r}\ls r^2_0\lt\|v_{r\th\th}\rt\|^2_{L^2}\lt\|(u_{r\th},u_{rr\th})\rt\|^2_{L^2}\ls \e^{-20}r^2_0\|(u,v)\|^4_E. \nn
\end{align}
Using H\"{o}lder inequality,  Sobolev embedding, Hardy inequality, we see that
\begin{align}
\int \f{|v ru_{r\th\th}|^2}{r-r_0}\leq \lt\|\f{vr}{\s{r-r_0}}\rt\|^2_{L^\i}\lt\|u_{r\th\th}\rt\|^2_{L^2}\ls \lt\|u_{r\th\th}\rt\|^4_{L^2}\ls \e^{-20}\|(u,v)\|^4_E. \nn
\end{align}
Using H\"{o}lder inequality, Hardy inequality and Sobolev embedding and the incompressibility, we see that
\begin{align}
\int \f{|v_{\th\th} ru_{r}|^2}{r-r_0}\leq r^2_0 \lt\|u_r\rt\|^2_{L^\i}\lt\|\f{v_{\th\th}}{r-r_0}\rt\|^2_{L^2}\ls r^2_0\lt\|u_r\rt\|^2_{L^\i}\lt\|v_{r\th\th}\rt\|^2_{L^2}\ls \e^{-20}r^2_0\|(u,v)\|^4_E. \nn
\end{align}
The same as \eqref{energyfinal3},   \eqref{energyfinal4}, we can obtain that
\begin{align}
\int \f{|u_{\th\th} v|^2+ |v_{\th\th} u|^2+|u_{\th} v_\th|^2}{r-r_0}\leq r^2_0 \lt\|u_r\rt\|^2_{L^\i}\lt\|\f{v_{\th\th}}{r-r_0}\rt\|^2_{L^2}\ls r^2_0\lt\|u_r\rt\|^2_{L^\i}\lt\|v_{r\th\th}\rt\|^2_{L^2}\ls \e^{-20}r^2_0\|(u,v)\|^4_E. \label{energyfinal8}
\end{align}

Inserting \eqref{energyfinal3} to  \eqref{energyfinal8} into \eqref{energyfinal2}, we achieve the first one of \eqref{energyfinal1}.

Applying estimate \eqref{remainderesti}, we see that
\begin{align}
 \|(R_u^a,\p_\th R^a_u, \p^3_\th R^a_u, \p_r R^a_u,  R_v^a, \p_\th R^a_v, \p^3_\th R^a_u)\|^2\leq \epsilon^{30} r^4_0.\label{energyfinal9}
 \end{align}

Inserting \eqref{energyfinal1} and \eqref{energyfinal9} into \eqref{energyfinal}, we see that

\begin{align}
\|(\t{u},\t{v})\|_E^2\ls& \e^{-24}\|(u,v)\|^4_E+ \e^{26} r^2_0. \label{energyfinal10}
\end{align}
Let $E=\{(u,v)\in C^\infty: (u,v) \ satisfes \ (\ref{iterative conditon})\ and \ \|(u,v)\|_E<+\infty\}$.

Then there exist $\epsilon_0>0, \dl_0>0$ such that for any $\epsilon\in (0,\epsilon_0), \dl\in(0,\dl_0)$, the operator
\beas
(u,v)\mapsto (\t{u},\t{v})
\eeas
maps the ball $B:=\{(u,v): \|(u,v)\|^2_E\leq \epsilon^{25}r^2_0\}$  in $B$ into itself.

Moreover, for every two pairs $(u_1, v_1)$ and $(u_2, v_2)$ in the ball,  by considering the system of $\t{u}_1-\t{u}_2,\ \t{v}_1-\t{v}_2$ and following the estimate of \eqref{energyfinal10} line by line, we can achieve that
\begin{align}
\|(\t{u}_1-\t{u}_2, \t{v}_1-\t{v}_2)\|^2_E\leq C\epsilon^{-24}(\|(u_1,v_1)\|_E^2+\|(u_2,v_2)\|_E^2)\|(u_1-u_2, v_1-v_2)\|_E^2.\nn
\end{align}
Hence, there exist $\epsilon_0>0,\dl_0>0$ such that for any $\epsilon\in (0,\epsilon_0), \dl\in (0,\dl_0)$, the operator
\beas
(u,v)\mapsto (\t{u},\t{v})
\eeas
maps the ball $B:=\{(u,v): \|(u,v)\|^2_E\leq \epsilon^{25}r^2_0\}$ into itself and is a contraction mapping. This completes the proof of Proposition \ref{existence and error estimate of error equation}.
\end{proof}

Now we can give the proof of Theorem \ref{thmain}.
\begin{proof}
Combining Proposition \ref{existence and error estimate of error equation} and the constructed approximate solution  in (\ref{app equationd}), we easily obtain Theorem \ref{thmain}.
\end{proof}

\section{Construction of approximate solutions}\label{secappro}
\indent
In this section, we construct an approximate solution of the Navier-Stokes equations (\ref{nspolar}) by the procedure of matched asymptotic expansion which can be explained by the following picture.

{\footnotesize
\begin{tikzpicture}
  \matrix (m) [matrix of math nodes, row sep=1.5em, column sep=1em]
    { (u^{(0)}_p,v^{(1)}_p) & (u^{(1)}_p,v^{(2)}_p) & (u^{(1+\f{1}{3})}_p,v^{(1+\f{4}{3})}_p) &(u^{(1+\f{2}{3})}_p,v^{(1+\f{5}{3})}_p)  & (u^{(1+\f{3}{3})}_p,v^{(1+\f{6}{3})}_p)  & (u^{(1+\f{4}{3})}_p,v^{(1+\f{7}{3})}_p) & \cdots\cdots\\
      (u^{(0)}_e,v^{(0)}_e) & (u^{(1)}_e,v^{(1)}_e) &  (u^{(1+\f{1}{3})}_e,v^{(1+\f{1}{3})}_e) & (u^{(1+\f{2}{3})}_e,v^{(1+\f{2}{3})}_e)  & (u^{(1+\f{3}{3})}_e,v^{(1+\f{3}{3})}_e)  & (u^{(1+\f{4}{3})}_e,v^{(1+\f{4}{3})}_e) & \cdots\cdots\\
      & (\hat{u}^{(1)}_p,\hat{v}^{(1+\f{2}{3})}_p) & (\hat{u}^{(1+\f{1}{3})}_p,\hat{v}^{(1+\f{3}{3})}_p) & (\hat{u}^{(1+\f{2}{3})}_p,\hat{v}^{(1+\f{4}{3})}_p)  & (\hat{u}^{(1+\f{3}{3})}_p,\hat{v}^{(1+\f{5}{3})}_p)  & (\hat{u}^{(1+\f{4}{3})}_p,\hat{v}^{(1+\f{6}{3})}_p)  & \cdots\cdots\\ };
  { [start chain] \chainin (m-2-1);
    \chainin (m-1-1);
    \chainin (m-2-2) [join={node[above,labeled] {\text{upper bound.}}}];
    { [start branch] \chainin (m-1-2); \chainin (m-2-5); \chainin (m-1-5);}
    \chainin (m-3-2);
    \chainin (m-2-4)[join={node[above,labeled] {\text{lower bound.}}}];
    { [start branch] \chainin (m-1-4);}
     \chainin (m-3-4);
     \chainin (m-2-6);\chainin (m-3-6);
     }
  { [start chain] \chainin (m-2-4);\chainin (m-3-4);\chainin (m-2-6);\chainin (m-1-6); }

 { [start chain] \chainin (m-2-5);\chainin (m-3-5);}

 { [start chain] \chainin (m-3-3);\chainin (m-2-5);}

\end{tikzpicture}
}

We give some explanations for this picture.
\begin{itemize}
\item We use $(u_e,v_e)$ to denote the Euler asymptotic expansion away from the boundary, while $(u_p,v_p)$ to denote the asymptotic expansion near the outer boundary and $(\hat{u}_p,\hat{v}_p)$ to denote the asymptotic expansion near the inner boundary.
\item The superscript on the Euler and the boundary layer asymptotic expansion represents the matching order of $\e$.
\item The ``upper" and ``down" arrow mean the matched boundary condition between $u_e$ and $u_p,\ \hat{u}_p$, while the slant arrow means the matched boundary condition matching between $v_e$ and $v_p,\ \hat{v}_p$.
\item The equations satisfied by $(u_e,v_e)$, $(u_p,v_p)$ and $(\hat{u}_p,\hat{v}_p)$ will be obtained by solving \eqref{ns} by matching the order of $\e$. The asymptotic expansion will be solved column by column starting from $(u^{(0)}_e,v^{(0)}_e):=(Ar+B\f{1}{r}, 0)$.
\end{itemize}

Define the outer boundary layer variable $\zeta=\f{r-(r_0+1)}{\e}$ and the inner boundary layer variable $\eta:=\f{r-r_0}{\e^{\f{2}{3}}}$. Now we are ready to write the formal asymptotic expansion away from the boundary and near the boundary.
\begin{remark}\label{rem4.1}
 Let us now explain the reason why we choose the thickness of the inner boundary layer to be $\e^{2/3}$. One can observe that the leading order of our approximate solution, $u^{(0)}_e(\th,r)=u^e(r)=A r+B\f{1}{r}$ is vanishing ($u_e(r_0)=0$) at the inner boundary, which indicates the inner boundary layer expansion doesn't contain a leading zero-order  term compared to the outer boundary layer expansion. Therefore, we can not expect the strong boundary layer scaling of $\e$ as the outer boundary. Instead, we should expect a much weaker boundary layer of thickness $\e^{\beta}$, where $\beta\in (0,1]$, and also form high orders approximation by considering $\e^{k\beta}$-order expansion. The consequence of this milder scaling is that the boundary layer equations at the inner boundary will become now uniform in $\e$. The choosing $\beta$ is as follows.

Assuming that at the inner boundary, the scaling variable is $\eta=\f{r-r_0}{\e^\beta}$ for $\beta\in(0,1]$ and the leading expansion term is
\begin{align*}
u=Ar+B\f{1}{r}+\e\lt(u^{(1)}_e(x,y)+\hat{u}^{(1)}_p(x,\eta)\rt)\\
v=\e\lt(v^{(1)}_e(x,y)+\e^{\beta}\hat{v}^{(1+\beta)}_p(x,\eta)\rt).
\end{align*}
Inserting this expansion in \eqref{ns}$_1$ and taking the leading $\e$ order, we see that
\bes
\e^{1+\beta}2A\eta\p_\th\hat{u}^{(1)}_p(\th,\eta)+\e^{1+\beta}2Ar_0 \hat{v}^{(1+\beta)}_p(\th,\eta)-\e^{3-2\beta}\p^2_{\eta}r_0\hat{u}^{(1)}_p(\th,\eta)=0.
\ees
To match the order, we need $1+\beta=3-2\beta$, which indicates that
\bes
\beta=2/3.
\ees
\end{remark} \qed

Now we write down the Euler asymptotic expansion away from the boundary, the boundary layer asymptotic expansions near the boundary and the boundary conditions satisfied by these expansions.

{\noindent\bf Euler expansions away from the boundary}

Away from the boundary, we make the following formal expansions
\begin{align}
&u^{\e}(\th,r)=u^e(r)+\sum^{k_0}_{k=0}\e^{1+\f{k}{3}}u^{(1+\f{k}{3})}_e(\th,r)+\text{h.o.t.},\nn\\[5pt]
&v^{\e}(\th,r)=\sum^{k_0}_{k=0}\e^{1+\f{k}{3}}v^{(1+\f{k}{3})}_e(\th,r)+\text{h.o.t.},\label{eulerextension}\\[5pt]
&p^{\e}(\th,r)=\sum^{k_0}_{k=0}\e^{1+\f{k}{3}}p^{(1+\f{k}{3})}_e(\th,r)+\text{h.o.t.}.\nn
\end{align}
Here and in what follows, ``h.o.t." means higher order terms. Actually, $(u^{(1+\f{1}{3})}_e,v_e^{(1+\f{1}{3})})\equiv (h_{1+1/3}(r),0)$ is a shear flow.

{\noindent\bf Boundary layer expansions near the outer boundary $\{r=r_0+1\}$}

Near the outer boundary, we make the following formal expansions
\begin{align}
u^{\e}(\th,r)=&u^e(r)+u_p^{(0)}(\th,\zeta)+\sum^{k_0}_{k=0}\e^{1+\f{k}{3}}\lt(u^{(1+\f{k}{3})}_e(\th,r)+u_p^{(1+\f{k}{3})}(\th,\zeta)\rt)+\text{h.o.t.},\nn\\[5pt]
v^{\e}(\th,r)=&\sum^{k_0}_{k=0}\e^{1+\f{k}{3}}\lt(v^{(1+\f{k}{3})}_e(\th,r)+v^{(1+\f{k}{3})}_p(\th,\zeta)\rt)+\text{h.o.t.}, \label{boundaryextensionu}\\[5pt]
p^{\e}(\th,r)=&\sum^{k_0}_{k=0}\e^{1+\f{k}{3}}\lt(p^{(1+\f{k}{3})}_e(\th,r)+p^{(1+\f{k}{3})}_p(\th,\zeta)\rt)+\text{h.o.t.},\nn
\end{align}
where the followings are set to be zero
\bes
u^{(1+\f{1}{3})}_p=0,\q v^{(1+\f{1}{3})}_p=v^{(1+\f{2}{3})}_p=v^{(1+\f{4}{3})}_p=0,\q p^{(1+\f{1}{3})}_p=p^{(1+\f{2}{3})}_p=p^{(1+\f{4}{3})}_p\equiv0.
\ees
The boundary condition is matched by
\begin{align}
&u^e(r)\big|_{r=r_0+1}+u_p^{(0)}\big|_{\zeta=0}=\t{\o}+\dl f(\th),\nn\\
&u_e^{(1+\f{k}{3})}\big|_{r=r_0+1}+u_p^{(1+\f{k}{3})}\big|_{\zeta=0}=0,\q k\in\{0\}\cup\bN, \nn\\
&v_e^{(1+\f{k}{3})}\big|_{r=r_0+1}+v_p^{(1+\f{k}{3})}\big|_{\zeta=0}=0,\q k\in\{0\}\cup\bN.\nn
\end{align}

{\noindent\bf Boundary layer expansions near the inner boundary $\{r=r_0\}$}

Near the inner boundary, we make the following formal expansions
\begin{align}
u^{\e}(\th,r)=&u^e(r)+\sum^{k_0}_{k=0}\e^{1+\f{k}{3}}\lt(u^{(1+\f{k}{3})}_e(\th,r)+\hat{u}_p^{(1+\f{k}{3})}(\th,\eta)\rt)+\text{h.o.t.},\nn\\[5pt]
v^{\e}(\th,r)=&\sum^{k_0}_{k=0}\e^{1+\f{k}{3}}\lt(v^{(1+\f{k}{3})}_e(\th,r)+\hat{v}^{(1+\f{k}{3})}_p(\th,\eta)\rt)+\text{h.o.t.}, \label{boundaryextensionl}\\[5pt]
p^{\e}(\th,r)=&\sum^{k_0}_{k=1}\e^{1+\f{k}{3}}\lt(p^{(1+\f{k}{3})}_e(\th,r)+\hat{p}^{(1+\f{k}{3})}_p(\th,\eta)\rt)+\text{h.o.t.},\nn
\end{align}
where we make the following assumptions
\bes
\hat{v}^{(1)}_p=\hat{v}^{(1+\f{1}{3})}_p=0,\q \hat{p}^{(1)}_p=\hat{p}^{(1+\f{1}{3})}_p=0.
\ees
The boundary condition is matched by
\begin{align}
&u_e^{(1+\f{k}{3})}\big|_{r=r_0}+\hat{u}_p^{(1+\f{k}{3})}\big|_{\eta=0}=0,\q  k\in\{0\}\cup\bN, \nn\\
&v_e^{(1+\f{k}{3})}\big|_{r=r_0}+\hat{v}_p^{(1+\f{k}{3})}\big|_{\eta=0}=0,\q  k\in\{0\}\cup\bN.\nn
\end{align}
Next we deduce the equations satisfied by these expansions.
\subsection{Equations for lower order expansions}\label{lowerexpansion}
\subsubsection{Equations for $(u_p^{(0)},v_p^{(1)},p_p^{(1)})$}
\indent

By substituting the outer boundary layer expansion \eqref{boundaryextensionu} into (\ref{nspolar}) and collecting the $\epsilon-0$th order terms, we obtain
 the following steady boundary layer equations for $(u_p^{(0)},v_p^{(1)},p_p^{(1)})$
\be\label{boundarylayeru0}
\left\{
\begin {array}{ll}
\big(\t{\o}(r_0+1)+u_p^{(0)}\big)\partial_\th u_p^{(0)}+\big( v_p^{(1)}- v_p^{(1)}(\th,0)\big)(r_0+1)\partial_\zeta u_p^{(0)}-(r_0+1)\partial^2_{\zeta}u_p^{(0)}=0,\\[5pt]
\p_\zeta p^{(1)}_p=0,\\[5pt]
\partial_\th u_p^{(0)}+(r_0+1)\partial_\zeta v_p^{(1)}=0,\\[5pt]
(u_p^{(0)},v_p^{(1)})(\th,\zeta)=(u_p^{(0)},v_p^{(1)})(\th+2\pi,\zeta),\\[5pt]
u_p^{(0)}\big|_{\zeta=0}=(\o+\dl f(\th)-\t{\o})(r_0+1),\\[5pt]
\lim\limits_{\zeta\rightarrow -\infty}(u_p^{(0)},v_p^{(1)},p_p^{(1)})=0.
\end{array}
\right.
\ee
From the \eqref{boundarylayeru0}$_{2,6}$ for the requirement of $p^{(1)}_p$, we have $p^{(1)}_p\equiv 0$.
\subsubsection{Equations for $(u_e^{(1)},v_e^{(1)},p_e^{(1)})$}
\indent

Inserting the Euler expansion \eqref{eulerextension} into \eqref{nspolar} and collecting the $\epsilon-1$th order terms, we deduce that $(u_e^{(1)},v_e^{(1)},p_e^{(1)})$ satisfies the following linearized Euler equations
\be\label{middleeuler1}
\left \{
\begin{array}{ll}
u^e(r) \partial_\th u_e^{(1)}+rv_e^{(1)}\p_r u^e(r)+u^e(r)v^{(1)}_e+\partial_\th p_e^{(1)}=0,\\[5pt]
u^e(r) \partial_\th v_e^{(1)}-2u^e(r) u^{(1)}_e+r\partial_r p_e^{(1)}=0,\\[5pt]
\partial_\th u_e^{(1)}+r\partial_r v_e^{(1)}+v_e^{(1)}=0,\\
(u_e^{(1)},v_e^{(1)})(\th,r)=(u_e^{(1)},v_e^{(1)})(\th+2\pi,r),\\
 v_e^{(1)}|_{r=r_0+1}=-v_p^{(1)}|_{\zeta=0},\ v_e^{(1)}|_{r=r_0}=0,
\end{array}
\right.
\ee
where $v_p^{(1)}$ is obtained from \eqref{boundarylayeru0}.

\subsubsection{Equations for $(u_p^{(1)},v_p^{(2)},p_p^{(2)})$}

\indent

By substituting the the outer boundary layer expansion \eqref{boundaryextensionu} into (\ref{nspolar}) and collecting the $\epsilon-1$th order terms, we obtain
 the following steady boundary layer equations for $(u_p^{(1)},v_p^{(2)},p_p^{(2)})$
 \be\label{boundarylayeru1}
 \lt\{
 \begin{aligned}
 &(\t{\o}(r_0+1)+u^{(0)}_p)\p_\th u^{(1)}_p+(v^{(2)}_p-v^{(2)}_p(\th,0))(r_0+1)\p_\zeta u^{(0)}_p-(r_0+1)\p^2_\zeta u^{(1)}_p\\
 &+u^{(1)}_p\p_\th u^{(0)}_p+(v^{(1)}_e(\th,r_0+1)+v^{(1)}_p)(r_0+1)\p_\zeta u^{(1)}_p=f_{p,1}(\th,\zeta),\\
 &-(r_0+1)\p_\zeta p^{(2)}_p=g_{p,1}(\th,\zeta),\\
 &\p_\th u^{(1)}_p+ (r_0+1)\p_\zeta v^{(2)}_p+\p_\zeta (\zeta v^{(1)}_p)=0,
 \end{aligned}
 \rt.
 \ee
 where
\begin{align}
&{f}_{p,1}(\theta, \zeta):=  \zeta \partial^2_{ \zeta} {u}_p^{(0)}+\partial_\zeta {u}_p^{(0)} \nn\\
& -{u}_p^{(0)}\left(\partial_\theta u_e^{(1)}\left(\theta, r_0+1\right)+v_e^{(1)}\left(\theta, r_0+1\right)+{v}_p^{(1)}\right)-\left(\zeta(u^e)^{\prime}\left(r_0+1\right) +u_e^{(1)}\left(\theta, r_0+1\right)\right) \partial_\theta {u}_p^{(0)} \nn\\
& -\left(r_0+1\partial_r v_e^{(1)}\left(\theta, r_0+1\right)+v_e^{(1)}\left(\theta, r_0+1\right)+{v}_p^{(1)}\right) \zeta \partial_\zeta u_p^{(0)}-\left[u^e\left(r_0+1\right)+(r_0+1) (u^e)^{\prime}\left(r_0+1\right)\right] {v}_p^{(1)}.\nn
\end{align}
and
\begin{align}
g_{p,1}(\th,\zeta):=&u^{(0)}_p \p_\th v^{(1)}_e(x,r_0+1)+(u^e(r_0+1)+u^{(0)}_p)\p_\th v^{(1)}_p\nn\\
                    &+(r_0+1)(v^{(1)}_e(\th,r_0+1)+v^{(1)}_p)\p_\zeta v^{(1)}_p-(r_0+1)\p^2_\zeta v^{(1)}_p\nn\\
                    &-2\lt(u^e(r_0+1)+u^{(0)}_p\rt)u^{(1)}_p-2u^{(0)}_p\lt(u^e(r_0+1)+(u^e)'(r_0+1)\rt).\nn
\end{align}

\subsubsection{Equations for $(\hat{u}_p^{(1)},\hat{v}_p^{(1+\f{2}{3})},\hat{p}_p^{(1+\f{2}{3})})$}

\indent

By substituting the the inner boundary layer expansion \eqref{boundaryextensionl} into (\ref{nspolar}), collecting the $\epsilon-(1+\f{2}{3})$th order terms for the first equation and collecting the $\epsilon-1$th order terms for the second equation, we obtain the following steady boundary layer equations for $(\hat{u}_p^{(1)},\hat{v}_p^{(1+\f{2}{3})},\hat{p}_p^{(1+\f{2}{3})})$
 \be\label{boundarylayerl1}
 \lt\{
 \begin{aligned}
 &2A\eta \p_\th \hat{u}^{(1)}_p+2Ar_0\lt(\hat{v}^{(1+\f{2}{3})}_p-\hat{v}^{(1+\f{2}{3})}_p(\th,0)\rt)-r_0\p^2_{\eta}\hat{u}^{(1)}_p+\p_\th \hat{p}^{(1+\f{2}{3})}_{p}=0,\\
 &\p_\eta \hat{p}_p^{(1+\f{2}{3})}=0,\\
 &\p_\th \hat{u}^{(1)}_p+r_0\p_\eta\hat{v}^{(1+\f{2}{3})}_p=0,\\
 &\hat{u}^{(1)}_p\big|_{\eta=0}=-u^{(1)}_e(\th,0),\q \lim\limits_{\eta\rightarrow +\infty}\hat{v}^{(1+\f{2}{3})}_p=0.
 \end{aligned}
 \rt.
 \ee

\subsubsection{Equations for $(u_p^{(1+\f{1}{3})},v_p^{(1+\f{4}{3})},p_p^{(1+\f{4}{3})})$}
\indent

Since there isn't $v_e^{(1+\f{1}{3})}$ and the scaling of order $\e$ is 1,  We just take
\bes
(u_p^{(1+\f{1}{3})},v_p^{(1+\f{4}{3})},p_p^{(1+\f{4}{3})})\equiv (0,0,0).
\ees

\subsubsection{Equations for $(\hat{u}_p^{(1+\f{1}{3})},\hat{v}_p^{(2)},\hat{p}_p^{(2)})$}

\indent

By substituting the the inner boundary layer expansion \eqref{boundaryextensionl} into (\ref{nspolar}), collecting the $\epsilon-2$th order terms for the first equation and collecting the $\epsilon-1\f{1}{3}$th order terms for the second equation, we obtain the following steady boundary layer equations for $(\hat{u}_p^{(1+\f{1}{3})},\hat{v}_p^{(2)},\hat{p}_p^{(2)})$
 \be\label{boundarylayerl4/3}
 \lt\{
 \begin{aligned}
 &2A\eta \p_\th \hat{u}^{(1+\f{1}{3})}_p+2Ar_0\lt(\hat{v}^{(2)}_p-\hat{v}^{(2)}_p(\th,0)\rt)-r_0\p^2_{\eta}\hat{u}^{(1+\f{1}{3})}_p+\p_\th \hat{p}^{(2)}_{p}=\hat{f}_{p,4/3},\\
 &\p_\eta \hat{p}_p^{(2)}=0,\\
 &\p_\th \hat{u}^{(1+\f{1}{3})}_p+r_0\p_\eta\hat{v}^{(2)}_p=0,\\
 &\hat{u}^{(1+\f{1}{3})}_p\big|_{\eta=0}=0,\q \lim\limits_{\eta\rightarrow +\infty}\hat{v}^{(2)}_p=0,
 \end{aligned}
 \rt.
 \ee
where
\bes
\hat{f}_{p,4/3}:=r_0\lt(\eta (\p_r v^{(1)}_e) (\th,r_0)+\hat{v}_p^{(1+\f{2}{3})}-\hat{v}_p^{(1+\f{2}{3})}(\theta,r_0)\rt)\p_\eta \hat{u}_p^{(1)}-\hat{u}_p^{(1)}\p_\th\hat{u}_p^{(1)}-\p_\th(\hat{u}_e^{(1)}(\th,r_0)\hat{u}_p^{(1)}).
\ees
\subsubsection{Equations for $(u_e^{(1+\f{2}{3})},v_e^{(1+\f{2}{3})},p_e^{(1+\f{2}{3})})$}
\indent

Inserting the Euler expansion \eqref{eulerextension} into \eqref{ns} and collecting the $\epsilon-1\f{2}{3}$th order terms, we deduce that $(u_e^{(1+\f{2}{3})},v_e^{(1+\f{2}{3})},$ $ p_e^{(1+\f{2}{3})})$ satisfies the following linearized Euler equations

\be\label{middleeuler2}
\left \{
\begin{array}{ll}
u^e(r) \partial_\th u_e^{(1+2/3)}+rv_e^{(1+2/3)}\p_r u^e(r)+u^e(r)v^{(1+2/3)}_e+\partial_\th p_e^{(1+2/3)}=0,\\[5pt]
u^e(r) \partial_\th v_e^{(1+2/3)}-2u^e(r) u^{(1+2/3)}_e+r\partial_r p_e^{(1+2/3)}=0,\\[5pt]
\partial_\th u_e^{(1+2/3)}+r\partial_r v_e^{(1+2/3)}+v_e^{(2/3)}=0,\\
(u_e^{(1+2/3)},v_e^{(1+2/3)})(\th,r)=(u_e^{(1+2/3)},v_e^{(1+2/3)})(\th+2\pi,r),\\
 v_e^{(1+2/3)}|_{r=r_0+1}=0,\ v_e^{(1+2/3)}|_{r=r_0}=-\hat{v}_p^{(1+2/3)}|_{\eta=0},
\end{array}
\right.
\ee

where $\hat{v}_p^{(1+\f{2}{3})}$ is obtained from \eqref{boundarylayerl1}.

\subsection{Solvability of the lower order asymptotic expansions}
\indent

Now we give the solvability of the equations satisfied by the lower order asymptotic expansions in Section \ref{lowerexpansion}. The order in which we solve these equations are as follows
{\small
\begin{align*}
(u^e(r),0)\rightarrow (u_p^{(0)},v_p^{(1)})\rightarrow (u_e^{(1)},v_e^{(1)})\rightarrow(u_p^{(1)},v_p^{(2)})/(\hat{u}_p^{(1)},\hat{v}_p^{(1+\f{2}{3})})/(\hat{u}_p^{(1+\f{1}{3})},\hat{v}_p^{(2)})\rightarrow (u_e^{(1+\f{2}{3})},v_e^{(1+\f{2}{3})})\cdots.
\end{align*}
}
\subsubsection{The boundary layer system \eqref{boundarylayeru0} and its solvability}
\indent
Denote $\bar{u}_p^{(0)}=(r_0+1)^{-1}{u}_p^{(0)}$, then from \eqref{boundarylayeru0}, we see that $(\bar{u}_p^{(0)},v_p^{(1)})$ satisfies
\be\label{boundarylayeru0scaling}
\left\{
\begin {array}{ll}
\big(\t{\o}+\bar{u}_p^{(0)}\big)\partial_\th \bar{u}_p^{(0)}+\big(v_p^{(1)}- v_p^{(1)}(\th,0)\big)\partial_\zeta \bar{u}_p^{(0)}-\partial^2_{\zeta}\bar{u}_p^{(0)}=0,\\[5pt]
\partial_\th \bar{u}_p^{(0)}+\partial_\zeta v_p^{(1)}=0,\\[5pt]
(\bar{u}_p^{(0)},v_p^{(1)})(\th,\zeta)=(\bar{u}_p^{(0)},v_p^{(1)})(\th+2\pi,\zeta),\\[5pt]
\bar{u}_p^{(0)}\big|_{\zeta=0}=\o+\dl f(\th)-\t{\o},\\[5pt]
\lim\limits_{\zeta\rightarrow -\infty}(\bar{u}_p^{(0)},v_p^{(1)},p_p^{(1)})=0.
\end{array}
\right.
\ee

A necessary condition for the solvability of the system \eqref{boundarylayeru0scaling} is stated in the following Lemma. One can also refer to \cite[Lemma 2.1]{FeiGLT:2023CMP}.
\begin{lemma}\label{lembw}
If the system \eqref{boundarylayeru0scaling} has a solution $(\bar{u}_p^{(0)}, v_p^{(1)})$, which satisfies
\beas
&&\t{\o}+\bar{u}_p^{(0)}(\th,\zeta)>0, \q \|\bar{u}_p^{(0)}\|_\infty<+\i,\ \forall\ \zeta\leq 0,
\eeas
then there holds
\begin{align}
\t{\o}^2=\o^2+\frac{\o\dl}{\pi}\int_0^{2\pi}f(\th)d\th+\frac{\dl^2}{2\pi}\int_0^{2\pi}f^2(\th)d\th.\nn
\end{align}
\end{lemma}

Next we need to solve the steady boundary layer equations (\ref{boundarylayeru0}), one can refer to \cite[Corollary 2.4]{FeiGLT:2023CMP} or \cite[Proposition 2.3]{FeiPZ:2023ARXIV} for the solvability of system \eqref{boundarylayeru0scaling}, which can be proven by using iteration argument and contraction mapping theorem. After scaling back to $({u}_p^{(0)},{v}_p^{(1)})$, we have the following result.
\begin{proposition}\label{propdcu0}
 There exists $\dl_0>0$ such that for any $\dl\in (0,\dl_0)$ and any $j,k,\ell\in \mathbb{N}\cup \{0\}$, the system (\ref{boundarylayeru0}) has a unique solution $(u^{(0)}_p,v^{(1)}_p)$ which satisfies $\forall\ \zeta\leq 0$,
\begin{align}\label{decayofblu0}
\begin{aligned}
&\lt\|\partial_\th^j\partial_\zeta^k {u}_p^{(0)}\big<\zeta\big>^\ell\rt\|_{L^\i}\leq C_{j,k,\ell}\dl r_0, \q \lt\|\partial_\th^j\partial_\zeta^k {v}_p^{(1)}\big<\zeta\big>^\ell \rt\|_{L^\i}\leq C_{j,k,\ell}\dl,\q \int_0^{2\pi}{v}_p^{(1)}(\th,\zeta)d \th=0.
\end{aligned}
\end{align}
\end{proposition}

\subsubsection{The linearized Euler system for $(u_e^{(1)}, v_e^{(1)}, p_e^{(1)})$ and its solvability}
\begin{proposition}\label{propeulerorder1}
There exists $\dl_0>0$ such that for any $\dl\in(0,\dl_0)$, the linearized Euler equations (\ref{middleeuler1}) have a solution $(u_e^{(1)}, v_e^{(1)}, p_e^{(1)})$ which satisfies
\begin{align}\label{euler1esti}
&\|\partial^{j}_\th\partial^{k}_ru_e^{(1)}\|_{L^\i}\leq C_{j,k}\dl r^{1-k}_0,\q \|\partial^j_\th\partial^{k}_rv_e^{(1)}\|_{L^\i}\leq C_{j,k}\dl r^{-k}_0, \ \forall\ j,k\geq 0.
\end{align}
\end{proposition}
\begin{proof}
Subtracting $\p_\th(\ref{middleeuler1})_2$ from $r\p_r(\ref{middleeuler1})_1$ to eliminate the pressure $p_e^{(1)}$, we obtain that
\bes
u^e(r)\Dl(r v^{(1)}_e)=0.
\ees
where $\Dl=\p^2_r+\f{1}{r^2}\p^2_{\th}+\f{1}{r}\p_r$. Since $u^\e(r)$ is non-degenerate in the annulus, we obtain that
\bes
\Dl(r v^{(1)}_e)=0.
\ees
Then we can obtain $v_e^{(1)}$ by solving the following linear boundary value problem
\be\label{eulerfirstorder1}
\lt\{
\bali
&\Dl (rv^{(1)}_e)=0,\\
&rv_e^{(1)}\big|_{r=r_0}=0,\q rv_e^{(1)}\big|_{r=r_0+1}=-(r_0+1){v}_p^{(1)}\big|_{\zeta=0}.\\
&rv_e^{(1)}(\th,r)=rv_e^{(1)}(\th+2\pi,r).
\eali
\rt.
\ee

The above equations can be solved by using the method of variables separation. By noting that
\bes
\int^{2\pi}_0 rv_e^{(1)} d\th\big|_{r=r_0}=0,\q \int^{2\pi}_0 rv_e^{(1)} dx\big|_{r=r_0+1}=-\int^{2\pi}_0(r_0+1){v}_p^{(1)}d\th\big|_{\zeta=0}=0,
\ees
and integrating \eqref{eulerfirstorder1} with respect to $\th$ variable we get that
\bes
\lt(\int^{2\pi}_0 rv_e^{(1)}(\th,r) d\th\rt)^{''}+\f{1}{r}\lt(\int^{2\pi}_0 rv_e^{(1)}(\th,r) d\th\rt)^{'}=0 \q \forall\ r\in [r_0,r_0+1],
\ees
which indicates that
\bes
\int^{2\pi}_0 rv_e^{(1)}(\th,r) d\th=0 \q \forall\ r\in [r_0,r_0+1].
\ees

By using the inner boundary condition $rv_e^{(1)}|_{r=r_0}=0$ and the equation \eqref{eulerfirstorder1}$_1$, we can assume that
\be\label{eulerfirst2}
rv_e^{(1)}:=\sum^\i_{n=1} a^1_n \lt(\lt(\f{r}{r_0}\rt)^n-\lt(\f{r_0}{r}\rt)^{n}\rt)\cos n\th+b^1_n \lt(\lt(\f{r}{r_0}\rt)^n-\lt(\f{r_0}{r}\rt)^{n}\rt)\sin n\th,
\ee
where $(a^1_n,b^1_n)$ are constants.  From the boundary condition on the outer boundary, we have
\begin{align}
a^1_n \lt(\lt(\f{r_0+1}{r_0}\rt)^n-\lt(\f{r_0}{r_0+1}\rt)^{n}\rt)=-(r_0+1)\f{1}{\pi}\int^{2\pi}_0{v}_p^{(1)}\big|_{\zeta=0}\cos n\th d\th, \label{eulerfirstorder3}\\
b^1_n \lt(\lt(\f{r_0+1}{r_0}\rt)^n-\lt(\f{r_0}{r_0+1}\rt)^{n}\rt)=-(r_0+1)\f{1}{\pi}\int^{2\pi}_0{v}_p^{(1)}\big|_{\zeta=0}\sin n\th d\th. \label{eulerfirstorder4}
\end{align}

By square-difference formula, it is easi to see that
\be
r^{-1}_0 \leq \lt(\f{r_0+1}{r_0}\rt)^n-\lt(\f{r_0}{r_0+1}\rt)^{n}\leq n 2^{n} r^{-1}_0 . \label{eulerfirstorder5}
\ee
By noting estimates for ${v}_p^{(1)}$ in \eqref{decayofblu0}, integration by parts for \eqref{eulerfirstorder3} and \eqref{eulerfirstorder4}, and the above inequality \eqref{eulerfirstorder5}, we see that

\begin{align}
|a^1_n|+|b^1_n|\leq  C_k\f{\dl}{n^k} r^2_0, \q \forall\ k\in\bN. \label{eulerfirst1}
\end{align}

By taking $\p^j_\th\p^k_r$ to \eqref{eulerfirst2} and using \eqref{eulerfirst1} and \eqref{eulerfirstorder5}, we can see that for $0\leq k\in \bN$,
\be\label{eulerfirst3}
\|\p^j_\th\p^k_r(rv_e^{(1)})\|_{L^\i}\leq C_{j,k}\dl r^{1-k}_0.
\ee
Then Leibniz formula indicates that
\bes
\|\p^j_\th\p^k_r v_e^{(1)}\|_{L^\i}\leq C_{j,k} \dl r^{-k}_0\q \text{for } 0\leq k\in \bN.
\ees
After $v^{(1)}_e$ is given, we define
\bes
u^{(1)}_e(\th,r)=-\int^\th_0 \p_r(rv^{(1)}_e(\bar{\th},r))d\bar{\th},
\ees
which satisfies
\begin{eqnarray}
\left \{
\begin {array}{ll}
\partial_\th u_e^{(1)}+\partial_r(rv_e^{(1)})=0,\\[5pt]
u_e^{(1)}(\th,r)=u_e^{(1)}(\th+2\pi,r).\nonumber
\end{array}
\right.
\end{eqnarray}
Using \eqref{eulerfirst3}, we see that for $k\in\bN$
\bes
\|\p^j_\th\p^k_r u_e^{(1)}\|_{L^\i}\leq C_{j,k}\dl r^{1-k}_0.
\ees
After obtaining $(u_e^{(1)}, v_e^{(1)})$, we construct $p_e^{(1)}$ as following
\begin{align*}
p_e^{(1)}(\th,r):=\phi(r)- u^e(r)u_e^{(1)}(\th,r)-\p_r(ru^e(r)) \int^\th_0v_e^{(1)}(\bar{\th},r)d\bar{\th},
\end{align*}
which satisfies
\bes
\p_\th p_e^{(1)}+u^e(r)\p_{\th} u_e^{(1)}+\p_r(ru^e(r))v_e^{(1)}=0.
\ees
Let $\phi(r)$ be a function satisfying
\begin{align*}
r\phi'(r)+u^e(r)\p_\th v^{(1)}_e(0,r)-2u^e(r)u_e^{(1)}(0,r)=0.
\end{align*}
Combining the equations of $(u_e^{(1)}, v_e^{(1)})$, it's direct to obtain
\begin{align*}
u^e(r) \partial_\th v_e^{(1)}-2u^e(r) u^{(1)}_e+r\partial_rp_e^{(1)}=0.
\end{align*}
Hence, $(u_e^{(1)}, v_e^{(1)},p_e^{(1)})$ solves the equation (\ref{middleeuler1}) and satisfies (\ref{euler1esti}).

\end{proof}

\subsubsection{The linearized Prandtl system for  $(u_p^{(1)},v_p^{(2)})$ and its solvability}
\indent

In this subsection, we consider the solvability of \eqref{boundarylayeru1}. We have the following Proposition.

\begin{proposition}\label{propdcu1}
There exists $\dl_0>0$ such that for any $\dl\in(0,\dl_0)$, the equations (\ref{boundarylayeru1}) have a unique solution $(u_p^{(1)},v_p^{(2)})$ which satisfies
\begin{align*}
\begin{aligned}
&\lt\|\partial_\th^j\partial_\zeta^k \big({u}_p^{(1)}-A_{1}\big)\big<\zeta\big>^\ell\rt\|_{L^\i}\leq C_{j,k,\ell}\dl r_0,\\
&\lt\|\partial_\th^j\partial_\zeta^k v_p^{(2)}\big<\zeta\big>^\ell\rt\|_{L^\i}\leq C_{j,k,\ell}\dl,\q \int_0^{2\pi}v_p^{(2)}(\th,\zeta)dx=0, \ \forall\ \zeta\leq 0,
\end{aligned}
\end{align*}
where
$A_{1}:=\lim\limits_{\zeta\rightarrow -\infty} u_p^{(1)}(\th,\zeta)$ is a constant which satisfies $|A_{1}|\leq C\dl r_0.$
\end{proposition}

Let $\kappa\in C_c^\infty ((-\infty,0])$ satisfy
\begin{align*}
\kappa(0)=1,\ \int^0_{-\infty}\kappa(\zeta)d\zeta=0.
\end{align*}
For simplicity, we set
\begin{align*}
\bar{u}:&=\t{\o}+(r_0+1)^{-1}u_p^{(0)}, \ \bar{v}:=v_p^{(1)}+v^{(1)}_p(\th,r_0+1),\\
u:&=(r_0+1)^{-1} u_p^{(1)}+(r_0+1)^{-1}u_e^{(1)}(\th,r_0+1)\kappa(\zeta),\\
v:&=v_p^{(2)}-v_p^{(2)}(\th,0)+\zeta v_p^{(1)} -(r_0+1)^{-1}\p_\th u_e^{(1)}(\th,r_0+1)\int_0^\zeta \kappa(\bar{\zeta})d\bar{\zeta}.
\end{align*}
Then, the equations (\ref{boundarylayeru1}) reduce to
\be\label{boundarylayeru1r}
\left \{
\begin {array}{ll}
\bar{u}\partial_x u+\bar{v}\partial_\zeta u+u\partial_x \bar{u}+v\partial_\zeta\bar{u}-\partial^2_{\zeta}u=\bar{f},\\[7pt]
\partial_\th u+\partial_\zeta v=0,\\[5pt]
u(\th,\zeta)=u(\th+2\pi,\zeta),\ v(\th,\zeta)=v(\th+2\pi,\zeta)\\[5pt]
u|_{\zeta=0}=v|_{\zeta=0}=0,\  \lim\limits_{\zeta\rightarrow -\infty}\partial_\zeta u=0,
\end{array}
\right.
\ee
where $\bar{f}(\th,\zeta)$ is $2\pi$-periodic function and decays fast as $\zeta\rightarrow -\infty$ given as follows
\begin{align*}
\bar{f}=&(r_0+1)^{-2} f_{p,1}+\bar{u} (r_0+1)^{-1} \p_\th u_e^{(1)}(\th,r_0+1)\kappa(\zeta)\\
        &+\lt(\zeta v_p^{(1)} -(r_0+1)^{-1}\p_\th u_e^{(1)}(\th,r_0+1)\int_0^\zeta \kappa(\bar{\zeta})d\bar{\zeta}\rt)\p_\zeta \bar{u}\\
        &+(r_0+1)^{-1} u_e^{(1)}(\th,r_0+1)\lt(\p_\th \bar{u}\kappa(\zeta)+\bar{v}\kappa'-\kappa^{''}(\zeta)\rt)
\end{align*}
It is easy to see that all the estimates on $\bar{f}$ is uniform for $r_0$, which means
\bes
\lt\|\partial_\th^j\partial_\zeta^k \bar{f}\big<\zeta\big>^\ell\rt\|_{L^\i}\leq C_{j,k,\ell}\dl,\ \forall\ \zeta\leq 0.
\ees

This system \eqref{boundarylayeru1r} is exactly (2.47) in \cite{FeiGLT:2023CMP}. Reader can refer there for further detailed proof of Proposition \ref{propdcu1}. Here we omit the details.

Next, we construct the pressure $p_p^{(2)}(x,\zeta)$. Thanks to \eqref{boundarylayeru1}$_2$, we take
\bes
p_p^{(2)}(\th,\zeta)=\int^{-\i}_{\zeta} (r_0+1)^{-1}g_{p,1}(\th,\bar{\zeta})d\bar{\zeta}.
\ees
which satisfies
\begin{align*}
-(r_0+1)\p_\zeta p_p^{(2)}(\th,\zeta)=g_1(\th,\zeta), \quad \lim_{\zeta\rightarrow -\infty}p_p^{(2)}(\th,\zeta)=0,
\end{align*}

\subsubsection{The linearized boundary layer system for  $(\hat{u}_p^{(1)},\hat{v}_p^{(1+\f{2}{3})},\hat{p}^{(1+\f{2}{3})}_p)$ and its solvability}
\indent

In this subsection, we consider the solvability of \eqref{boundarylayerl1}. We have the following proposition.

\begin{proposition}\label{propdcl1}
There exists $\dl_0>0$ and a constant $\hat{A}_1$ such that for any $\dl\in(0,\dl_0)$, the equations (\ref{boundarylayerl1}) have a unique solution $(\hat{u}_p^{(1)},\hat{v}_p^{(1+\f{2}{3})},\hat{p}^{(1+\f{2}{3})}_p)$ which satisfies
\begin{align} \label{decayofbll1}
\begin{aligned}
&\lt\|\partial_\th^j\p^k_{\eta}\big(\hat{u}_p^{(1)}-\hat{A}_{1}\big)\eta^{\ell} \rt\|_{L^\i}\leq C_{j,k,\ell}\dl r_0, \\
&\lt\|\partial_\th^j\p^k_{\eta}\hat{v}_p^{(1+\f{2}{3})}\eta^{\ell} \rt\|_{L^\i}\leq C_{j,k,\ell}\dl ,\q \int_0^{2\pi}\hat{v}_p^{(1+\f{2}{3})}(\th,\eta)dx=0, \ \forall\ \ell\in \mathbb{N},\ \eta\geq 0,
\end{aligned}
\end{align}
where $\hat{A}_{1}=\lim\limits_{\eta\rightarrow+\i}\hat{u}_p^{(1)}$ satisfying $\hat{A}_{1}\ls C\dl r_0$.
\end{proposition}

\pf

We set
\bes
 \hat{p}_p^{(1+\f{2}{3})}=2Ar_0\int^\th_0\hat{v}^{(1+\f{2}{3})}_p(\bar{\th},0)d\bar{\th}.
\ees
Let $\kappa\in C_c^\infty ([0,+\i))$ satisfy
\begin{align*}
\kappa(0)=1,\ \int_0^{\infty}\kappa(\eta)d\eta=0,\q  \kappa^{(k)}(\eta)\leq  C_k.
\end{align*} define
\begin{align*}
&\hat{u}(\th,\eta)=r_0^{-1}\hat{u}^{(1)}_p(\th,\eta)+\kappa(\eta)r_0^{-1}u^{(1)}_e(\th,r_0),\q \hat{v}=\hat{v}^{(1+\f{2}{3})}_p-r_0^{-1}\int^{\eta_\ast}_0\kappa(\bar{\eta})d\bar{\eta}\p_\th u^{(1)}_e(\th,r_0).
\end{align*}
Then the system \eqref{boundarylayerl1} can be reformulated into
 \be\label{boundarylayerl1f}
 \lt\{
 \begin{aligned}
 &\lt(2Ar_0^{-1}\rt)\eta \p_\th \hat{u}+\lt(2A r_0^{-1}\rt)\hat{v}-\p^2_{\eta}\hat{u}=\hat{f}_{p,1},\\
 &\p_\th \hat{u}+\p_{\eta}\hat{v}=0,\\
 &\hat{u}\big|_{\eta=0}=0,\q \p_\eta\hat{u}\big|_{\eta=+\i}=\hat{v}\big|_{\eta=+\i}=0,
 \end{aligned}
 \rt.
 \ee

where
\begin{align}
\hat{f}_{p,1}:=&\lt(2Ar_0^{-1}\rt) r_0^{-1}\p_\th u^{(1)}_e(\th,r_0)\lt(\eta\kappa(\eta)-\int^{\eta}_0\kappa(\bar{\eta})d\bar{\eta}\rt)-\kappa^{''}(\eta)r_0^{-1}u^{(1)}_e(\th,r_0),
\end{align}
has compact support for $\eta$ and satisfying
\bes
\|\eta_\ast^\ell\p^j_\th\p^k_{\eta} \hat{f}_{p,1}\|_{L^\i}\leq C_{j,k,\ell}\dl .
\ees
From the definition of $A$ in \eqref{defAB}, we see that
\bes
2A r_0^{-1}\approx\t{\o}.
\ees

The system \eqref{boundarylayerl1f} is the same with (4.14) in \cite{FeiPZ:2023ARXIV} and has been solved there by Proposition 4.6. Actually we can obtain that there exists a Constant $\hat{B}_1=\lim\limits_{\eta\rightarrow+\i}\hat{u}$ such that

\begin{align*}
&\lt\|\partial_\th^j\p^k_{\eta}\big(\hat{u}-\hat{B}_{1}\big)\eta^{\ell} \rt\|_{L^\i}\leq C_{j,k,\ell}\dl, \\
&\lt\|\partial_\th^j\p^k_{\eta}\hat{v}\eta^{\ell} \rt\|_{L^\i}\leq C_{j,k,\ell}\dl,\q \int_0^{2\pi}\hat{v}_p^{(1+\f{2}{3})}(\th,\eta)d\th=0, \ \forall\ \ell\in \mathbb{N},\ \eta\geq 0,
\end{align*}
and $\hat{B}_{1}=\lim\limits_{\eta\rightarrow+\i}\hat{u}$ satisfying $\hat{B}_{1}\ls C\dl$. From the above estimates for $(\hat{u},\hat{v})$ and going back to $(\hat{u}_p^{(1)}, \hat{v}_p^{(1+2/3)})$, we have
\begin{align*}
&\hat{u}^{(1)}_p(\th,\eta)=r_0 \hat{u}(\th,\eta)-\kappa(\eta)u^{(1)}_e(\th,r_0),\\
&\hat{v}^{(1+\f{2}{3})}_p(\th,\eta)=\hat{v}(\th,\eta)+r_0^{-1}\int^{\eta}_0\kappa(\bar{\eta})d\bar{\eta}\p_\th u^{(1)}_e(\th,r_0).
\end{align*}
Then define $\hat{A}_{1}:=\hat{B}_{1}r_0$,
\begin{align*}
&\hat{u}^{(1)}_p(\th,\eta)-\hat{A}_{1}=r_0 \lt[\hat{u}(\th,\eta)-\hat{B}_{1}\rt]-\kappa(\eta) u^{(1)}_e(\th,r_0),\\
&\hat{v}^{(1+\f{2}{3})}_p(\th,\eta)=\hat{v}(\th,\eta)+r_0^{-1}\int^{\eta}_0\kappa(\bar{\eta})d\bar{\eta}\p_\th u^{(1)}_e(\th,r_0),
\end{align*}

which indicates that
\begin{align*}
&\lt|\eta^\ell\partial_\th^j\p^k_{\eta}\lt[\hat{u}^{(1)}_p(\th,\eta)-\hat{A}_{1}\rt]\rt|= r_0\lt|\eta^\ell\partial_\th^j\p^k_{\eta} \lt[\hat{u}(\th,\eta)-\hat{B}_{1}\rt]\rt| +\lt|\eta^\ell\partial_\th^j\p^k_{\eta}\kappa(\eta)u^{(1)}_e(\th,r_0)\rt|\leq C_{j,k,\ell}\dl r_0,
\end{align*}

and
\begin{align*}
&\lt|\eta^\ell\partial_\th^j\p^k_{\eta}\hat{v}^{(1+\f{2}{3})}_p(\th,\eta)\rt|=\lt|\eta_\ast^\ell\partial_\th^j\p^k_{\eta} \lt[\hat{v}(\th,\eta)+r_0^{-1}\int^{\eta}_0\kappa(\bar{\eta})d\bar{\eta}\p_\th u^{(1)}_e(\th,r_0)\rt]\rt|\leq  C_{j,k,\ell}\dl.
\end{align*}

 we can obtain \eqref{decayofbll1}. \ef

\subsubsection{The linearized Euler system for $({u}_p^{(1+\f{1}{3})},\ {v}_p^{(2)},\ {p}_p^{(2)})$ and its solvability}

It is taken to be $({u}_p^{(1+\f{1}{3})},\ {v}_p^{(2)},\ {p}_p^{(2)})\equiv (0,0,0)$.

\subsubsection{The linearized Euler system for $(\hat{u}_p^{(1+\f{1}{3})},\ \hat{v}_p^{(2)},\ \hat{p}_p^{(2)})$ and its solvability}

\indent

By setting $(\hat{u},\hat{v})=(r_0^{-1}\hat{u}_p^{(1+\f{1}{3})},\ \hat{v}_p^{(2)})$, we see $(\hat{u},\hat{v})$ satisfies the same system with \eqref{boundarylayerl1f}, so the solvability and estimates for $(\hat{u}_p^{(1+\f{1}{3})},\ \hat{v}_p^{(2)})$ are the same with Proposition \ref{propdcl1}. Here we have the following Proposition and omit the proof.

\begin{proposition}\label{propdcl4/3}
There exists $\dl_0>0$ and a constant $\hat{A}_{1+\f{1}{3}}$ such that for any $\dl\in(0,\dl_0)$, the equations (\ref{boundarylayerl4/3}) have a unique solution $(\hat{u}_p^{(1+\f{1}{3})},\hat{v}_p^{(2)},\hat{p}^{(2)}_p)$ which satisfies
\bes
\begin{aligned}
&\lt\|\partial_\th^j\p^k_{\eta}\big(\hat{u}_p^{(1+\f{1}{3})}-\hat{A}_{1+\f{1}{3}}\big)\langle\eta\rangle^{\ell} \rt\|_{L^\i}\leq C_{j,k,\ell}\dl r_0, \\
&\lt\|\partial_\th^j\p^k_{\eta}\hat{v}_p^{(2)}\langle\eta\rangle^{\ell} \rt\|_{L^\i}\leq C_{j,k,\ell}\dl,\q \int_0^{2\pi}\hat{v}_p^{(2)}(x,\eta)d\th=0, \ \forall\ \ell\in \mathbb{N},\ \eta\geq 0,
\end{aligned}
\ees
where $\hat{A}_{1+\f{1}{3}}=\lim\limits_{\eta\rightarrow+\i}\hat{u}_p^{(1+\f{1}{3})}$ satisfying $\hat{A}_{1+\f{1}{3}}\ls C\dl r_0$.
\end{proposition}

\subsubsection{The linearized Euler system for $(u_e^{(1+\f{2}{3})},\ v_e^{(1+\f{2}{3})},\ p_e^{(1+\f{2}{3})})$ and its solvability}
\indent

Since the system \eqref{middleeuler2} is completely the same as the system \eqref{middleeuler1}, we omit the details to prove the solvability.

\subsubsection{Corrections of the expansions to the boundary layer equations and the Euler equations}
\indent

Since we want the solutions of the boundary layer equations decaying fast at infinity, we define
\bes
\t{u}^{(1)}_p:=u^{(1)}_p-A_1,\q \t{\hat{u}}^{(1)}_p:=\hat{u}^{(1)}_p-\hat{A}_1,\q  \t{\hat{u}}^{(1+1/3)}_p:=\hat{u}^{(1+1/3)}_p-\hat{A}_{1+1/3}.
\ees
and use $\t{u}^{(1)}_p$, $\t{\hat{u}}^{(1)}_p$ and $\t{\hat{u}}^{(1+\f{1}{3})}_p$ to replace the original ${u}^{(1)}_p$ and ${\hat{u}}^{(1)}_p$ and ${\hat{u}}^{(1+\f{1}{3})}_p$ respectively. Then we need to correct $u^{(1)}_e$ and $u^{(1+\f{1}{3})}_e$ and to $\t{u}^{(1)}_e$  and $\t{u}^{(1+\f{1}{3})}_e,$ such that they satisfy
\begin{align}
&\t{u}^{(1)}_e\big|_{r=r_0+1}=-u^{(1)}_p\big|_{\zeta=0}+A_1={u}^{(1)}_e\big|_{r=r_0+1}+A_1\nn\\
&\t{u}^{(1)}_e\big|_{r=r_0}=-\hat{u}^{(1)}_p\big|_{\eta=0}+\hat{A}_1={u}^{(1)}_e\big|_{r=r_0}+\hat{A}_1,\nn\\
&{u}^{(1+\f{1}{3})}_e\big|_{r=r_0+1}=0, \q \text{we take } u_p^{(1+\f{1}{3})}\equiv 0,\nn\\
&\t{u}^{(1+\f{1}{3})}_e\big|_{r=r_0}=-\hat{u}^{(1+\f{1}{3})}_p\big|_{\eta=0}+\hat{A}_{1+\f{1}{3}}=\hat{A}_{1+\f{1}{3}},\nn
\end{align}
and $(\t{u}^{(1)}_e,v^{(1)}_e)$ is still a solution of the system \eqref{middleeuler1}.

Also, using divergence-free condition and a direct calculation, we see that
\bes
\p_\th \lt( r(\Dl-\f{1}{r^2})u_e^{(1)} +\f{2}{r}\p_\th v_e^{(1)}\rt)=0.
\ees
Let
\bes
\t{u}^{(1)}_e={u}^{(1)}_e+h_1(r),
\ees
where $h_1(r)$ satisfies the following ODE
\begin{align}
&r(\Dl-\f{1}{r^2})h_1(r)=-r(\Dl-\f{1}{r^2})u_e^{(1)} +\f{2}{r}\p_\th v_e^{(1)},\nn\\
&h_1(r_0)=\hat{A}_1 ,\q h_1(r_0+1)=A_1.\nn
\end{align}
By estimates of $u_e^{(1)}$ and $v_e^{(1)}$.  It is not hard to deduce that
\bes
\|\p^{k}_r h_1(r)\|_{L^\i}\leq C_{k} \dl r^{1-k}_0.
\ees

Obviously, $(\t{u}^{(1)}_e, v^{(1)}_e)$ is still a solution of the system \eqref{middleeuler1} and satisfies \eqref{euler1esti}. Also we have that
  \bes
 r(\Dl-\f{1}{r^2})\t{u}_e^{(1)} +\f{2}{r}\p_\th v_e^{(1)}\equiv 0.
  \ees
At last we choose
\bes
(u_e^{(1+1/3)}, v_e^{(1+1/3)})=\lt(\f{\hat{A}_{1+1/3}r_0^2}{-(2r_0+1)} \lt(r-\f{(r_0+1)^2}{r}\rt),0\rt),
\ees
which satisfies
\begin{align}
&r(\Dl-\f{1}{r^2})u_e^{(1+1/3)}=0,\nn\\
&u_e^{(1+1/3)}(r_0)=\hat{A}_{1+\f{1}{3}},\q u_e^{(1+1/3)}(r_0+1)=0,\nn\\
&\|\p^{k}_r u_e^{(1+1/3)}(r)\|_{L^\i}\leq C_{k} \dl r_0^{1-k}.\nn
\end{align}
  By letting $\t{u}_p^{1}$ and $\t{u}_e^{(1)}$ to replace  ${u}_p^{1}$ and $u_e^{(1)}$, the equations for the first order outer boundary layer is invariant, which are also true for the equations of the first order inner boundary when replacing $\hat{u}_p^{(1)}$ and $u_e^{(1)}$ with $\t{\hat{u}}_p^{(1)}$ and $\t{u}_e^{(1)}$, and true for the equations of the $(1+\f{1}{3})$th order inner boundary  when replacing $\hat{u}_p^{(1+\f{1}{3})}$ with $\t{\hat{u}}_p^{(1+\f{1}{3})}$.

\subsection{Equations for higher order expansions and its solvability}
\indent

After performing the above lower order extensions, we can perform the higher order expansions inductively as follows.

\subsubsection{Equations for $(u_e^{(1+\f{k}{3})},v_e^{(1+\f{k}{3})},p_e^{(1+\f{k}{3})})$}\label{sec431}
\indent

For $k\geq 2$, by collecting the $\epsilon-1+\f{k}{3}$th order terms of the Euler expansions in \eqref{eulerextension}, we deduce that $(u_e^{(1+\f{k}{3})},v_e^{(1+\f{k}{3})},$ $ p_e^{(1+\f{k}{3})})$ satisfies the following linearized Euler equations
\be\label{middleeulerk}
\left \{
\begin{array}{ll}
u^e(r) \partial_\th u_e^{(1+\f{k}{3})}+rv_e^{(1+\f{k}{3})}\p_r u^e(r)+u^e(r)v_e^{(1+\f{k}{3})} +\partial_\th p_e^{(1+\f{k}{3})}=f^{(1+\f{k}{3})}_e,\\[5pt]
u^e(r) \partial_\th v_e^{(1+\f{k}{3})}-2u^e(r)u_e^{(1+\f{k}{3})}+r\partial_r p_e^{(1+\f{k}{3})}=g^{(1+\f{k}{3})}_e,\\[5pt]
\partial_\th u_e^{(1+\f{k}{3})}+\partial_r(r v_e^{(1+\f{k}{3})})=0,\\
 v_e^{(1+\f{k}{3})}|_{r=r_0+1}=-{v}_p^{(1+\f{k}{3})}|_{\zeta=0},\ v_e^{(1+\f{k}{3})}|_{r=r_0}=-\hat{v}_p^{(1+\f{k}{3})}|_{\eta=0},\\
(u_e^{(1+\f{k}{3})}, v_e^{(1+\f{k}{3})})(\th,r)=(u_e^{(1+\f{k}{3})}, v_e^{(1+\f{k}{3})})(\th+2\pi,r),
\end{array}
\right.
\ee
where the forced term $f^{(1+\f{k}{3})}_e$ is defined as
\begin{align}
f^{(1+\f{k}{3})}_e:=&-\sum^{k}_{i=0} u_e^{(1+\f{i}{3})}\partial_\th u_e^{(\f{k-i}{3})}-r\sum^{k}_{i=0} v_e^{(1+\f{i}{3})}\partial_r u_e^{(\f{k-i}{3})}\nn\\
          &-\sum^{k}_{i=0} v_e^{(1+\f{i}{3})} u_e^{(\f{k-i}{3})}+r\lt(\Dl-\f{1}{r^2}\rt) u_e^{(1+\f{k}{3}-2)}+\f{2}{r}\p_\th v_e^{(1+\f{k}{3}-2)},\nn
\end{align}
and $g^{(1+\f{k}{3})}_e$ is defined as
\begin{align}
g^{(1+\f{k}{3})}_e:=&-\sum^{k}_{i=0} u_e^{(1+\f{i}{3})}\partial_\th v_e^{(\f{k-i}{3})}-r\sum^{k}_{i=0} v_e^{(1+\f{i}{3})}\partial_r v_e^{(\f{k-i}{3})}\nn\\
     &+\sum^{k-1}_{i=0} u_e^{(1+\f{i}{3})} u_e^{(\f{k-i}{3})}+\Dl (rv_e^{(1+\f{k}{3}-2)}).\nn
\end{align}
Here for notation simplification, we understand that
\begin{align}
u^{(-\f{1}{3})}_e=u^{(\f{1}{3})}_e=u^{(\f{2}{3})}_e\equiv 0,\q u^{(0)}_e=u^e(r),\q
v^{(-\f{1}{3})}_e=v^{(0)}_e=v^{(\f{1}{3})}_e=v^{(\f{2}{3})}_e\equiv 0. \nn
\end{align}

{\noindent\bf Here we give another explanation why we choose the Euler shear flow satisfying}
\bes
u^e(r)=Ar+\f{B}{r}.
\ees

When $k=3$, the system \eqref{middleeulerk} is simplified to
\be\label{middleeuler3}
\left \{
\begin{array}{ll}
u^e(r) \partial_\th u_e^{(2)}+v_e^{(2)}\p_r(r u^e(r))+\partial_\th p_e^{(2)}=f^{(2)}_e,\\[5pt]
u_e(r) \partial_\th v_e^{(2)}-2u^e(r)u_e^{(2)}+r\partial_r p_e^{(2)}=g^{(2)}_e,\\[5pt]
\partial_\th u_e^{(2)}+\partial_r(r v_e^{(2)})=0,\\
 v_e^{(2)}|_{r=r_0+1}=-{v}_p^{(2)}|_{\zeta=0},\ v_e^{(2)}|_{r=r_0}=-\hat{v}_p^{(2)}|_{\eta=0},\\
(u_e^{(2)}, v_e^{(2)})(\th,r)=(u_e^{(2)}, v_e^{(2)})(\th+2\pi,r),
\end{array}
\right.
\ee
where
\begin{align}
f^{(2)}_e:&=- u_e^{(1)}\partial_\th u_e^{(1)}-v_e^{(1)}\partial_r(r u_e^{(1)})+r(\Dl-\f{1}{r^2}) u^e(r),\label{bwx}\\
g^{(2)}_e:&=- u_e^{(1)}\partial_\th v_e^{(1)}-rv_e^{(1)}\partial_r v_e^{(1)}+(u_e^{(1)})^2.\nn
\end{align}

If $(u^e(r), 0)$ is the shear flow solution of the Euler equation, by taking $\th$ average of \eqref{middleeuler3}, we can see that it satisfies
\bes
\int^{2\pi}_0 f^{(2)}_e(\th,r)d\th\equiv 0.
\ees
Then from \eqref{bwx}, we see that
\begin{align*}
&2\pi r\lt(\p^2_r+\f{1}{r}\p_r-\f{1}{r^2}\rt)u^e(r)=\int_0^{2\pi}v_e^{(1)}\partial_r(ru_e^{(1)})d\th\\
=&\int_0^{2\pi}\partial_r(ru_e^{(1)})d(\int_0^\th v_e^{(1)}(\th',r)d\th')=-\int_0^{2\pi}\partial_{ r}(r\p_\th u_e^{(1)})\int_0^\th v_e^{(1)}(\th',r)d\th' d\th\\
=&\int_0^{2\pi}\partial_{r}(r\p_r(r v_e^{(1)}))\int_0^\th v_e^{(1)}(\th',r)d\th' d\th=\int_0^{2\pi}r\lt(\Dl-\f{\p^2_\th}{r^2}\rt)(rv_e^{(1)})\int_0^\th v_e^{(1)}(\th',r)d\th' d\th\\
=&-\int_0^{2\pi}\p^2_\th v_e^{(1)}\int_0^\th v_e^{(1)}(\th',r)d\th' d\th=\int_0^{2\pi}\p_\th v_e^{(1)} v_e^{(1)} d\th=0,
\end{align*}
which also gives the explanation of the choice of the leading Euler equation. \qed

Now we show the solvability of the system \eqref{middleeulerk}.
\begin{proposition}\label{propeulerorderk}
There exists $\dl_0>0$ such that for any $\dl\in(0,\dl_0)$, the linearized Euler equations (\ref{middleeulerk}) has a solution $(u_e^{(1+\f{1+k}{3})}, v_e^{(1+\f{1+k}{3})}, p_e^{(1+\f{1+k}{3})})$ which satisfies
\begin{align}\label{eulerkesti}
&\|\partial^{i}_\th\partial^{j}_r u_e^{(1+\f{k}{3})}\|_{L^\i}\leq C_{i,j}\dl r_0^{1-j},\q \|\partial^i_\th\partial^{j}_rv_e^{(1+\f{k}{3})}\|_{L^\i}\leq C_{i,j}\dl r_0^{-j}, \ \forall\ i,j\geq 0.
\end{align}
\end{proposition}
\indent

The idea of proof is the same as that of the system \eqref{middleeuler1}. First we cancel the pressure from \eqref{middleeulerk}$_1$ and \eqref{middleeulerk}$_2$ to obtain that
\be\label{middleeulerk1}
\lt\{
\bali
&-u^e(r) \Dl ( rv_e^{(1+\f{k}{3})}) =r\p_rf^{(1+\f{k}{3})}_e-\p_\th g^{(1+\f{k}{3})}_e,\\
&rv_e^{(1+\f{k}{3})}|_{r=R_2}=-R_2{v}_p^{(1+\f{k}{3})}|_{\zeta=0},\ rv_e^{(1+\f{k}{3})}|_{r=R_1}=-R_1\hat{v}_p^{(1+\f{k}{3})}|_{\eta=0},\\
&v_e^{(1+\f{k}{3})}(\th,r)=v_e^{(1+\f{k}{3})}(\th+2\pi,r).
\eali
\rt.
\ee
Noting in fact that $rv_e^{(1+\f{k}{3})}$ is harmonic, i.e.
\bes
r\p_rf^{(1+\f{k}{3})}_e-\p_\th g^{(1+\f{k}{3})}_e=0.
\ees
Actually, by the procedure of the Euler expansions, we construct
\bes
r\lt(\Dl-\f{1}{r^2}\rt) u_e^{(1+\f{k}{3}-2)}+\f{2}{r}\p_\th v_e^{(1+\f{k}{3}-2)}=\Dl (rv_e^{(1+\f{k+1}{3}-2)})=0,
\ees
Then a direct computation indicates that
\begin{align}
r\p_rf^{(1+\f{k}{3})}_e-\p_\th g^{(1+\f{k}{3})}_e =&\underbrace{-\sum^{k}_{i=0} r\p_r u_e^{(1+\f{i}{3})}\partial_\th u_e^{(\f{k-i}{3})}}_{I_1}\underbrace{-\sum^{k}_{i=0} u_e^{(1+\f{i}{3})}r\p_r\partial_\th u_e^{(\f{k-i}{3})}}_{I_2}\nn\\
   &\quad\underbrace{-\sum^{k}_{i=0} r\p_rv_e^{(1+\f{i}{3})}\partial_r(r u_e^{(\f{k-i}{3})})}_{I_3}\underbrace{-\sum^{k}_{i=0} v_e^{(1+\f{i}{3})}r\p_r(\p_r(r u_e^{(\f{k-i}{3})}))}_{I_4}\nn\\
   &\quad\underbrace{+\sum^{k}_{i=0} \p_\th u_e^{(1+\f{i}{3})}\partial_\th v_e^{(\f{k-i}{3})}}_{I_5}+\underbrace{\sum^{k}_{i=0} u_e^{(1+\f{i}{3})}\partial^2_\th v_e^{(\f{k-i}{3})}}_{I_6}\nn\\
   &\quad+\underbrace{\sum^{k}_{i=0} r\p_\th v_e^{(1+\f{i}{3})}\partial_r v_e^{(\f{k-i}{3})}}_{I_7}+\underbrace{\sum^{k}_{i=0} rv_e^{(1+\f{i}{3})}\p_\th \partial_r v_e^{(\f{k-i}{3})}}_{I_8}.\nn\\
   &\quad-\underbrace{\sum^{k}_{i=0} \p_\th u_e^{(1+\f{i}{3})} u_e^{(\f{k-i}{3})}}_{I_{9}}-\underbrace{\sum^{k}_{i=0} u_e^{(1+\f{i}{3})}\p_\th  u_e^{(\f{k-i}{3})}}_{I_{10}}.\nn
\end{align}
By using the incompressibility and harmonic property of $r v_e^{(1+\f{i}{3})}$, we can obtain that
\bes
I_2+I_6+I_{10}=0.
\ees
By using the incompressibility and the property of
\bes
r\lt(\Dl-\f{1}{r^2}\rt) u_e^{(1+\f{i}{3})}+\f{2}{r}\p_\th v_e^{(1+\f{i}{3})}=0,
\ees
we can obtain that
\bes
I_1+I_{3}+I_{4}+I_{5}+I_{7}+I_{8}+I_{9}=0.
\ees
This indicates that $  \Dl (rv_e^{(1+\f{1+k}{3})})=0$. Then the system \eqref{middleeulerk1} is solved by solving the following linear boundary value problem
\be\label{middleeulerk2}
\lt\{
\bali
&- \Dl (rv_e^{(1+\f{k}{3})}) =0,\\
&rv_e^{(1+\f{k}{3})}|_{r=r_0+1}=-(r_0+1){v}_p^{(1+\f{k}{3})}|_{\zeta=0},\ rv_e^{(1+\f{k}{3})}|_{r=r_0}=-r_0\hat{v}_p^{(1+\f{k}{3})}|_{\eta=0},\\
&v_e^{(1+\f{k}{3})}(\th,r)=v_e^{(1+\f{k}{3})}(\th+2\pi,r).
\eali
\rt.
\ee
The above equations can be solved by using the method of variables separation. By noting that
\bes
\int^{2\pi}_0 rv_e^{(1+\f{k}{3})} d\th\big|_{r=r_0}= \int^{2\pi}_0 rv_e^{(1+\f{k}{3})} d\th\big|_{r=r_0+1}=0,
\ees
and integrating \eqref{middleeulerk2}$_1$ with respect to $\th$ variable, we get that
\bes
\lt(\int^{2\pi}_0 rv_e^{(1+\f{k}{3})}(\th,r) d\th\rt)^{''}+\f{1}{r}\lt(\int^{2\pi}_0 rv_e^{(1+\f{k}{3})}(\th,r) d\th\rt)^{'}=0 \q \forall\ r\in [r_0,r_0+1],
\ees
which indicates that
\be\label{eulerkmeanzero}
\int^{2\pi}_0 rv_e^{(1+\f{k}{3})}(\th,r) d\th=0 \q \forall\ r\in [r_0,r_0+1].
\ee
This implies, after integrating \eqref{middleeulerk}$_1$ in $\th\in [0,2\pi]$, the compatibility condition: for any $r\in[r_0,r_0+1]$,
\be\label{middleeulerk3}
\int^{2\pi}_0 f^{(1+\f{k}{3})}_e(\th,r)d\th=0.
\ee

Actually by integration by parts, we have
\begin{align}
 &\sum^{k}_{i=0}\int^{2\pi}_0 u_e^{(1+\f{i}{3})}\partial_\th u_e^{(\f{k-i}{3})}d\th =-\sum^{k}_{i=0}\int^{2\pi}_0 \p_\th u_e^{(1+\f{i}{3})} u_e^{(\f{k-i}{3})}d\th\nn \text{ after relabeling the index} \label{middleeulerk4}\\
 & =-\sum^{k}_{i=0}\int^{2\pi}_0u_e^{(1+\f{i}{3})}\partial_\th u_e^{(\f{k-i}{3})}d\th,\nn
\end{align}
which implies that
\begin{align}
 &\sum^{k}_{i=0}\int^{2\pi}_0 u_e^{(1+\f{i}{3})}\partial_\th u_e^{(\f{k+1-i}{3})}d\th=0.
\end{align}

Using the incompressibility and harmonic property of $rv_e^{(1+\f{k}{3})}$, we have
\begin{align}
 &\sum^{k}_{i=0}\int^{2\pi}_0 v_e^{(1+\f{i}{3})}\partial_r(r u_e^{(\f{k-i}{3})})d\th =\sum^{k}_{i=0}\int^{2\pi}_0 v_e^{(1+\f{i}{3})}\int^\th_0\p_r(r\partial_\th u_e^{(\f{k-i}{3})})(\bar{\th},r)d\bar{\th}d\th\nn\\
& =-\sum^{k}_{i=0}\int^{2\pi}_0 v_e^{(1+\f{i}{3})}\int^\th_0\p_r(r\partial_r(r v_e^{(\f{k-i}{3})}))(\bar{\th},r)d\bar{\th}d\th =\sum^{k}_{i=0}\int^{2\pi}_0 v_e^{(1+\f{i}{3})}\int^\th_0\partial^2_\th v_e^{(\f{k-i}{3})}(\bar{\th},r)d\bar{\th}d\th \label{middleeulerk5}\\
 & =\sum^{k}_{i=0}\int^{2\pi}_0 v_e^{(1+\f{i}{3})}\lt(\partial_\th v_e^{(\f{k-i}{3})}(\th,r)-\partial_\th v_e^{(\f{k-i}{3})}(0,r)\rt)d\th =0.\nn
\end{align}
In the last but second line of \eqref{middleeulerk5},  integration by parts (the same as \eqref{middleeulerk4}) implies that the first term is zero, while the fact that the second term is zero is due to
\bes
\int^{2\pi}_0 v_e^{(1+\f{i}{3})}d\th=0, \text{ for } \forall\ r\in[r_0,r_0+1].
\ees
Combining \eqref{middleeulerk4} and \eqref{middleeulerk5}, we obtain \eqref{middleeulerk3}.

From \eqref{middleeulerk2}$_1$ and \eqref{eulerkmeanzero}, we can assume that
\be\label{eulerk2}
rv_e^{(1+\f{k}{3})}:=\sum^\i_{n=1} \lt(a^k_n\lt(\f{r}{r_0}\rt)^n-a^k_{-n}\lt(\f{r_0}{r}\rt)^{n}\rt)\cos n\th+\lt(b^k_n\lt(\f{r}{r_0}\rt)^n-b^k_{-n}\lt(\f{r_0}{r}\rt)^{n}\rt)\sin n\th,
\ee
where $(a^k_{n},a^k_{-n},b^k_{n},b^k_{-n})$ are constants.  From the boundary condition on the outer boundary, we have
\begin{align}
a^k_n\lt(\f{r_0+1}{r_0}\rt)^n-a^k_{-n}\lt(\f{r_0}{r_0+1}\rt)^{n}=-(r_0+1)\f{1}{\pi}\int^{2\pi}_0{v}_p^{(1+\f{k}{3})}\big|_{\zeta=0}\cos n\th d\th, \label{eulerk3}\\
b^k_n\lt(\f{r_0+1}{r_0}\rt)^n-b^k_{-n}\lt(\f{r_0}{r_0+1}\rt)^{n}=-(r_0+1)\f{1}{\pi}\int^{2\pi}_0{v}_p^{(1+\f{k}{3})}\big|_{\zeta=0}\sin n\th d\th, \nn\\
a^k_n-a^k_{-n}=-r_0\f{1}{\pi}\int^{2\pi}_0{v}_p^{(1+\f{k}{3})}\big|_{\eta=0}\cos n\th d\th, \label{eulerk5}\\
b^k_n-b^k_{-n}=-r_0\f{1}{\pi}\int^{2\pi}_0{v}_p^{(1+\f{k}{3})}\big|_{\eta=0}\sin n\th d\th. \nn
\end{align}

From \eqref{eulerk3} and \eqref{eulerk5}, integration by parts for the righthand terms, we see that
\begin{align}
a^k_n\lt(\lt(\f{r_0+1}{r_0}\rt)^n-\lt(\f{r_0}{r_0+1}\rt)^{n}\rt)\leq C_i\f{\dl}{n^i} (r_0+1), \q \forall\ i\in\bN,\nn\\
a^k_{-n}\lt(1-\lt(\f{r_0}{r_0+1}\rt)^{2n}\rt)\leq C_i\f{\dl}{n^i} (r_0+1), \q \forall\ i\in\bN \nn
\end{align}
By square-difference formula, it is easy to see that
\bes
r_0^{-1}\leq \lt(\f{r_0+1}{r_0}\rt)^n-\lt(\f{r_0}{r_0+1}\rt)^{n}\leq n 2^nr_0^{-1}.
\ees
Then we have
\begin{align}
|a^k_n|\leq  C_i\f{\dl}{n^i} r_0^2,\q |a^k_{-n}|\leq C_i\f{\dl}{n^i} r_0^2 \q \forall\ i\in\bN. \label{eulerk8}
\end{align}
The same estimate is true for $b^k_n$ and $b^k_{-n}$. By taking $\p^j_\th\p^{k}_r$ to $\f{1}{r}\eqref{eulerk2}$, and using \eqref{eulerk8}, we can see that
\be\label{eulerk9}
\|\p^i_\th\p^{j}_rv_e^{(1+\f{k}{3})}\|_{L^\i}\leq C_{i,j}\dl r_0^{-j} .
\ee
After $v^{(1+\f{k}{3})}_e$ is given, we define
\bes
u^{(1+\f{k}{3})}_e(\th,r)=-\int^\th_0 \p_r(rv^{(1+\f{k}{3})}_e(\bar{\th},r))d\bar{\th},
\ees
which satisfies
\be
\left \{
\begin {array}{ll}
\partial_\th u_e^{(1+\f{k}{3})}+\partial_r(rv_e^{(1+\f{k}{3})})=0,\\[5pt]
u_e^{(1+\f{k}{3})}(\th,r)=u_e^{(1+\f{k}{3})}(\th+2\pi,r).\nonumber
\end{array}
\right.
\ee
Using \eqref{eulerk9}, we see that
\bes
\|\p^i_\th\p^{j}_r u_e^{(1+\f{k}{3})}\|_{L^\i}\leq C_{j,k}\dl r_0^{1-j}.
\ees

After obtaining $(u_e^{(1+\f{k}{3})}, v_e^{(1+\f{k}{3})})$, we construct $p_e^{(1+\f{k}{3})}$ as following
\begin{align*}
p_e^{(1+\f{k}{3})}(\th,r):=\phi(r)- u^e(r)u_e^{(1+\f{k}{3})}(\th,r)-\p_r(ru^e(r)) \int^\th_0v_e^{(1+\f{k}{3})}(\bar{\th},r)d\bar{\th}+\int^\th_0f_e^{(1+\f{k}{3})}(\bar{\th},r)d\bar{\th},
\end{align*}
which satisfies
\bes
\p_\th p_e^{(1+\f{k}{3})}+u^e(r)\p_{\th} u_e^{(1+\f{k}{3})}+\p_r(ru^e(r))v_e^{(1+\f{k}{3})}=f_e^{(1+\f{k}{3})}.
\ees
Let $\phi(r)$ be a function satisfying
\begin{align*}
r\phi'(r)+u^e(r)\p_\th v^{(1+\f{k}{3})}_e(0,r)-2u^e(r)u_e^{(1+\f{k}{3})}(0,r)+g_e^{(1+\f{k}{3})}(0,r)=0.
\end{align*}
Combining the equations of $(u_e^{(1+\f{k}{3})}, v_e^{(1+\f{k}{3})})$, it's direct to obtain
\begin{align*}
u^e(r) \partial_\th v_e^{(1+\f{k}{3})}-2u^e(r) u^{(1+\f{k}{3})}_e+r\partial_rp_e^{(1+\f{k}{3})}=0.
\end{align*}
Hence, $(u_e^{(1+\f{k}{3})}, v_e^{(1+\f{k}{3})},p_e^{(1+\f{k}{3})})$ solves the equation (\ref{middleeulerk}) and satisfies (\ref{eulerkesti}). \ef

\subsubsection{Higher order boundary expansions near the outer boundary $\{r=r_0+1\}$}

\indent

By substituting the the outer boundary layer expansion \eqref{boundaryextensionu} into (\ref{nspolar}) and collecting the $\epsilon-1+\f{k}{3}$th $(k\geq 2)$ order terms, we obtain
 the following steady boundary equations for $(u_p^{(1+\f{k}{3})},v_p^{(2+\f{k}{3})},p_p^{(2+\f{k}{3})})$
 \be\label{boundarylayeruk}
 \lt\{
 \begin{aligned}
 &(\t{\o}(r_0+1)+u^{(0)}_p)\p_\th u^{(1+\f{k}{3})}_p+u^{(1+\f{k}{3})}_p\p_\th u^{(0)}_p+(v^{(1)}_e(\th,r_0+1)+v^{(1)}_p)(r_0+1)\p_\zeta u^{(1+\f{k}{3})}_p\\
 &-(r_0+1)\p^2_\zeta u^{(1+\f{k}{3})}_p+(v^{(2+\f{k}{3})}_p-v^{(2+\f{k}{3})}_p(\th,0))(r_0+1)\p_\zeta u^{(0)}_p=f_{p,k}(\th,\zeta),\\
 &-(r_0+1)\p_\zeta p^{(2+\f{k}{3})}_p=g_{p,k}(\th,\zeta),\\
 &\p_\th u^{(1+\f{k}{3})}_p+ \p_\zeta v^{(2+\f{k}{3})}_p+\p_\zeta(\zeta v_p^{(1+\f{k}{3})})=0,
 \end{aligned}
 \rt.
 \ee
 where the force term $f_{p,k}(\th,\zeta)$ and $g_{p,k}(\th,\zeta)$ are lower order terms, which decay fast at $\zeta$ infinity. The exact representations of $f_{p,k}$ and $g_{p,k}$ are not important here, we choose to omit to write it out.

For the system \eqref{boundarylayeruk}, we have the following proposition.

\begin{proposition}\label{propdcuk}
There exists $\dl_0>0$ such that for any $\dl\in(0,\dl_0)$, the equations (\ref{boundarylayeruk}) has a unique solution $(u_p^{(1+\f{k}{3})},v_p^{(2+\f{k}{3})})$ which satisfies
\begin{align*}
\begin{aligned}
&\lt\|\partial_\th^i\partial_\zeta^j \big({u}_p^{(1+\f{k}{3})}-A_{1+\f{k}{3}}\big)\big<\zeta\big>^{\ell}\rt\|_{L^\i}\leq C_{i,j,\ell}\dl r_0,\\
&\lt\|\partial_\th^i\partial_\zeta^j v_p^{(2+\f{k}{3})}\big<\zeta\big>^{\ell}\rt\|_{L^\i}\leq C_{i,j,\ell}\dl ,\q \int_0^{2\pi}v_p^{(2+\f{k}{3})}(\th,\zeta)d\th =0, \ \forall\ \zeta\leq 0,
\end{aligned}
\end{align*}
where
$A_{1+\f{k}{3}}:=\lim\limits_{\zeta\rightarrow -\infty} u_p^{(1+\f{k}{3})}(\th,\zeta)$ is a constant which satisfies $|A_{1+\f{k}{3}}|\leq C\dl r_0.$
\end{proposition}
\pf

Let $\kappa\in C_c^\infty ((-\infty,0])$ satisfy
\begin{align*}
\kappa(0)=1,\ \int^0_{-\infty}\kappa(\zeta)d\zeta=0.
\end{align*}
For simplicity, we set
\begin{align*}
\bar{u}:&=\t{\o}+(r_0+1)^{-1}u_p^{(0)}, \ \bar{v}:=v_p^{(1)}+v^{(1)}_p(\th,r_0+1),\\
u:&=(r_0+1)^{-1}u_p^{(1+\f{k}{3})}+(r_0+1)^{-1}u_e^{(1+\f{k}{3})}(\th,r_0+1)\kappa(\zeta),\\
v:&=v_p^{(2+\f{k}{3})}-v_p^{(2+\f{k}{3})}(\th,0)+\zeta v_p^{(1+\f{k}{3})} -(r_0+1)^{-1}\p_\th u_e^{(1+\f{k}{3})}(\th,r_0+1)\int_0^\zeta \kappa(\bar{\zeta})d\bar{\zeta}.
\end{align*}
Then, the equations (\ref{boundarylayeruk}) reduce to
\begin{eqnarray}\label{bluk1}
\left \{
\begin {array}{ll}
\bar{u}\partial_\th u+\bar{v}\partial_\zeta u+u\partial_\th \bar{u}+v\partial_\zeta\bar{u}-\partial^2_{\zeta}u=\bar{f}_{p,k},\\[7pt]
\partial_\th u+\partial_\zeta v=0,\\[5pt]
u(\th,\zeta)=u(\th+2\pi,\zeta),\ v(\th,\zeta)=v(\th+2\pi,\zeta)\\[5pt]
u|_{\zeta=0}=v|_{\zeta=0}=0,\  \lim\limits_{\zeta\rightarrow -\infty}\partial_\zeta u=0,
\end{array}
\right.
\end{eqnarray}
where $\bar{f}_{p,k}$ is $2\pi$-periodic function and decays fast as $\zeta\rightarrow -\infty$ and have the following estimate
\bes
\lt\|\partial_\th^i\partial_\zeta^j \bar{f}_{p,k}\big<\zeta\big>^\ell\rt\|_{L^\i}\leq C_{i,j,\ell}\dl,\ \forall\ \zeta\leq 0.
\ees

This system \eqref{bluk1} is solved  as \eqref{boundarylayeru1} and is exactly (2.47) in \cite{FeiGLT:2023CMP}. One can refer to there for further detailed proof of Proposition \ref{propdcuk}. Here we omit the details.

Next, we construct the pressure $p_p^{(2+\f{k}{3})}(\th,\zeta)$. Thanks to \eqref{boundarylayeruk}$_2$, we take
\bes
p_p^{(2+\f{k}{3})}(\th,\zeta)=\int^{-\i}_{\zeta} (r_0+1)^{-1}g_{p,k}(\th,\bar{\zeta})d\bar{\zeta}.
\ees
which satisfies
\begin{align*}
-(r_0+1)\p_\zeta p_p^{(2+\f{k}{3})}(\th,\zeta)=g_{p,k}(\th,\zeta), \quad \lim_{\zeta\rightarrow -\infty}p_p^{(2+\f{k}{3})}(\th,\zeta)=0.
\end{align*}

\ef

\subsubsection{Higher order boundary expansions near the inner boundary $\{r=r_0\}$}

\indent

By substituting the the inner boundary layer expansion \eqref{boundaryextensionl} into (\ref{nspolar}) and collecting the $\epsilon-(1+\f{2+k}{3})$th order terms, we obtain
 the following steady boundary layer equations for $(\hat{u}_p^{(1+\f{k}{3})},\hat{v}_p^{(1+\f{2+k}{3})},$ $\hat{p}_p^{(1+\f{2+k}{3})})$

 \be\label{boundarylayerlk}
 \lt\{
 \begin{aligned}
 &2A\eta\p_\th\hat{u}^{(1+\f{k}{3})}_p+2A r_0\lt(\hat{v}^{(1+\f{2+k}{3})}_p-\hat{v}^{(1+\f{2+k}{3})}_p(\th,0)\rt)+\p_\th\hat{p}^{(1+\f{2+k}{3})}_p-r_0\p^2_{\eta}\hat{u}^{(1+\f{k}{3})}_p=\hat{f}_{p,k}(\th,\eta),\\
 &r_0\p_\eta p^{(1+\f{2+k}{3})}_p=\hat{g}_{p,k}(\th,\eta),\\
 &\p_\th \hat{u}^{(1+\f{k}{3})}_p+ r_0\p_\eta \hat{v}^{(1+\f{2+k}{3})}_p+\p_\eta(\eta \hat{v}^{(1+\f{k}{3})}_p)=0,\\
& \lim\limits_{\eta\rightarrow +\infty} \hat{v}^{(1+\f{2+k}{3})}_p=0,
 \end{aligned}
 \rt.
 \ee
  where the force term $\hat{f}_{p,k}(\th,\eta)$ and $\hat{g}_{p,k}(\th,\eta)$ are lower order terms, which decay fast at $\eta$ infinity. The exact representations of $\hat{f}_{p,k}$ and $\hat{g}_{p,k}$ are not important here, we choose to omit to write it out.

For the system \eqref{boundarylayerlk}, we have the following proposition.
\begin{proposition}\label{propdclk}
There exists $\dl_0>0$ and a constant $\hat{A}_{1+\f{k}{3}}$ such that for any $\dl\in(0,\dl_0)$, the system (\ref{boundarylayerlk}) has a unique solution $(\hat{u}_p^{(1+\f{k}{3})},\hat{v}_p^{(1+\f{2+k}{3})},\hat{p}^{(1+\f{2+k}{3})}_p)$ which satisfies
\begin{align*}
\begin{aligned}
&\lt\|\eta^{\ell}\partial_\th^i\p^j_{\eta}\big(\hat{u}_p^{(1+\f{k}{3})}-\hat{A}_{1+\f{k}{3}}\big) \rt\|_{L^\i}\leq C_{i,j,\ell}\dl r_0,\\
&\lt\|\eta^{\ell}\partial_\th^i\p^j_{\eta}\hat{v}_p^{(1+\f{2+k}{3})} \rt\|_{L^\i}\leq C_{i,j,\ell}\dl ,\q \int_0^{2\pi}\hat{v}_p^{(1+\f{2+k}{3})}(\th,\eta)d\th=0, \ \forall\ \eta\geq 0,
\end{aligned}
\end{align*}
where $\hat{A}_{1+\f{k}{3}}=\lim\limits_{\eta\rightarrow+\i}\hat{u}_p^{(1+\f{k}{3})}$ and $|\hat{A}_{1+\f{k}{3}}|\leq C\dl r_0$.
\end{proposition}

\pf

We set
\bes
 \hat{p}_p^{(1+\f{3+k}{3})}=2Ar_0\int^\th_0\hat{v}_p^{(1+\f{2+k}{3})}(\bar{\th},0)d\bar{\th}-\int^{+\i}_\eta r_0^{-1}\hat{g}_{p,k}(\th,\bar{\eta})d\bar{\eta},
\ees

Let $\kappa\in C_c^\infty ([0,+\i))$ satisfy
\begin{align*}
\kappa(0)=1,\ \int_0^{\infty}\kappa(\eta)d\eta=0,\q  \kappa^{(k)}(\eta)\leq  C_k.
\end{align*}
Define
\begin{align*}
&\hat{u}(\th,\eta)=r_0^{-1}\hat{u}^{(1+\f{k}{3})}_p(\th,\eta)+\kappa(\eta)r_0^{-1}u^{(1+\f{k}{3})}_e(\th,r_0),\\
&\hat{v}(\th,\eta)=\hat{v}^{(1+\f{2+k}{3})}_p+r_0^{-1}\eta\hat{v}^{(1+\f{k}{3})}_p-r_0^{-1}\int^{\eta}_0\kappa(\bar{\eta})d\bar{\eta}\p_\th u^{(1+\f{k}{3})}_e(\th,r_0).
\end{align*}
Then system \eqref{boundarylayerlk} can be reformulated into
 \be\label{boundarylayerlkf}
 \lt\{
 \begin{aligned}
 &\lt(2Ar_0^{-1}\rt)\eta \p_\th \hat{u}+\lt(2Ar_0^{-1}\rt)\hat{v}-\p^2_{\eta}\hat{u}=\hat{f}_{p},\\
 &\p_\th \hat{u}+\p_{\eta}\hat{v}=0,\\
 &\hat{u}\big|_{\eta=0}=0\q \p_\eta\hat{u}\big|_{\eta=+\i}=\hat{v}\big|_{\eta=+\i}=0,
 \end{aligned}
 \rt.
 \ee
where $\hat{f}_{p}(\th,\eta)$ is defined as

\begin{align*}
\hat{f}_{p}:=&r_0^{-2}\hat{f}_{p,k}+r_0^{-3}\int^{+\i}_\eta \p_\th \hat{g}_{p,k}(\th,\bar{\eta})d\bar{\eta}\\
   &+\lt(2Ar_0^{-1}\rt) r_0^{-1}\p_\th u^{(1+\f{k}{3})}_e(\th,r_0)\lt(\eta_\ast\kappa(\eta)-\int^{\eta}_0\kappa(\bar{\eta})d\bar{\eta}\rt)\nn\\
&-\kappa^{''}(\eta)r_0^{-1}u^{(1)}_e(\th,r_0)+\lt(2Ar_0^{-1}\rt)r_0^{-1}\eta \hat{v}_p^{(1+\f{k}{3})},
\end{align*}

which decays fast enough at $\eta$ infinity. From the definition of $A$ in \eqref{defAB} and estimate for $\hat{f}_{p,k}$ and $\hat{g}_{p,k}$, we can obtain that
\bes
2Ar_0^{-1}\approx \t{\o},\q \| \eta^\ell\p^i_\th\p^j_{\eta} \hat{f}_{p}\|_{L^\i}\leq C_{i,j,\ell}\dl.
\ees

The system \eqref{boundarylayerlkf} is the same with (4.14) in \cite{FeiPZ:2023ARXIV} and has been solved there by Proposition 4.6. Actually we can obtain that there exists a Constant $\hat{B}_{1+\f{k}{3}}=\lim\limits_{\eta\rightarrow+\i}\hat{u}$ such that

\begin{align*}
&\lt\|\eta^{\ell} \partial_\th^i\p^j_{\eta}\big(\hat{u}-\hat{B}_{1+\f{k}{3}}\big)\rt\|_{L^\i}\leq C_{i,j,\ell}\dl, \\
&\lt\|\eta^{\ell} \partial_\th^i\p^j_{\eta}\hat{v}\rt\|_{L^\i}\leq C_{i,j,\ell}\dl,\q \int_0^{2\pi}\hat{v}_p^{(1+\f{2+k}{3})}(\th,\eta)d\th=0, \ \forall\ \ell\in \mathbb{N},\ \eta\geq 0.
\end{align*}
and $\hat{B}_{1+\f{k}{3}}=\lim\limits_{\eta\rightarrow+\i}\hat{u}$ satisfying $\hat{B}_{1+\f{k}{3}}\ls C\dl$. From the above estimates for $(\hat{u},\hat{v})$ and going back to $(\hat{u}_p^{(1+\f{k}{3})}, \hat{v}_p^{(1+\f{2+k}{3})})$, we have
\begin{align*}
\hat{u}^{(1)}_p(\th,\eta)=&r_0 \hat{u}(\th,\eta)-\kappa(\eta)r_0^{-1}u^{(1)}_e(\th,r_0),\\
\hat{v}^{(1+\f{2}{3})}_p(\th,\eta)=&\hat{v}(\th,\eta)+r_0^{-1}\int^{\eta}_0\kappa(\bar{\eta})d\bar{\eta}\p_\th u^{(1)}_e(\th,r_0)-r_0^{-1}\eta \hat{v}^{(1+\f{k}{3})}_p.
\end{align*}
Then define $\hat{A}_{1+\f{k}{3}}:=\hat{B}_{1+\f{k}{3}}r_0$
\begin{align*}
&\hat{u}^{(1+\f{k}{3})}_p(\th,\eta)-\hat{A}_{1+\f{k}{3}}=r_0 \lt[\hat{u}(\th,\eta)-\hat{B}_{1+\f{k}{3}}\rt]-\kappa(\eta)u^{(1+\f{k}{3})}_e(\th,r_0),\\
&\hat{v}^{(1+\f{2+k}{3})}_p(\th,\eta)=\hat{v}(\th,\eta)+r_0^{-1}\int^{\eta}_0\kappa(\bar{\eta})d\bar{\eta}\p_\th u^{(1+\f{k}{3})}_e(\th,r_0)-r_0^{-1}\eta \hat{v}^{(1+\f{k}{3})}_p,
\end{align*}

which indicates that
\begin{align*}
&\lt|\eta^\ell\partial_\th^j\p^k_{\eta}\lt[\hat{u}^{(1+\f{k}{3})}_p(\th,\eta)-\hat{A}_{1+\f{k}{3}}\rt]\rt|= r_0\lt|\eta^\ell\partial_\th^j\p^k_{\eta} \lt[\hat{u}(\th,\eta_\ast)-\hat{B}_{1+\f{k}{3}}\rt]\rt|-\lt||\eta^\ell\partial_\th^j\p^k_{\eta}\kappa(\eta)u^{(1+\f{k}{3})}_e(\th,r_0)\rt|\ls C_{j,k,\ell}\dl r_0,
\end{align*}
and
\begin{align*}
&\lt|\eta^\ell\partial_\th^j\p^k_{\eta}\hat{v}^{(1+\f{2}{3})}_p(\th,\eta)\rt|\leq \lt|\eta^\ell\partial_\th^j\p^k_{\eta} \lt[\hat{v}(\th,\eta)+r_0^{-1}\int^{\eta}_0\kappa(\bar{\eta})d\bar{\eta}\p_\th u^{(1)}_e(\th,r_0)\rt]\rt|\\
 &+\lt|r_0^{-1}\eta^\ell\partial_\th^j\p^k_{\eta}\lt(\eta \hat{v}^{(1+\f{k}{3})}_p\rt)\rt|\leq C_{j,k,\ell}\dl.
\end{align*}

Then we can obtain the result in Proposition \ref{propdclk}. \ef

\subsubsection{Correction of the boundary layer solutions and the Euler equations of higher order}
\indent

Since we want the solutions of the boundary layer equations decaying fast at infinity, we define
\bes
\t{u}^{(1+\f{k}{3})}_p:=u^{(1+\f{k}{3})}_p-A_{1+\f{k}{3}},\q \t{\hat{u}}^{(1)}_p:=\hat{u}^{(1+\f{k}{3})}_p-\hat{A}_{1+\f{k}{3}}.
\ees
and use $\t{u}^{(1+\f{k}{3})}_p$ and $\t{\hat{u}}^{(1+\f{k}{3})}_p$ to replace the original ${u}^{(1+\f{k}{3})}_p$ and ${\hat{u}}^{(1+\f{k}{3})}_p$ respectively. Then we need to correct $u^{(1+\f{k}{3})}_e$ to $\t{u}^{(1+\f{k}{3})}_e,$ such that they satisfy
\begin{align}
&\t{u}^{(1+\f{k}{3})}_e\big|_{r=r_0+1}=-u^{(1+\f{k}{3})}_p\big|_{\zeta=0}+A_{1+\f{k}{3}}={u}^{(1+\f{k}{3})}_e\big|_{r=r_0+1}+A_{1+\f{k}{3}},\nn\\
&\t{u}^{(1+\f{k}{3})}_e\big|_{r=r_0}=-\hat{u}^{(1+\f{k}{3})}_p\big|_{\eta=0}+\hat{A}_{1+\f{k}{3}}={u}^{(1+\f{k}{3})}_e\big|_{r=r_0}+\hat{A}_{1+\f{k}{3}}.\nn
\end{align}
and $(\t{u}^{(1+\f{k}{3})}_e,v^{(1+\f{k}{3})}_e)$ is still a solution of the system \eqref{middleeulerk}.

Also, using divergence-free condition and the harmonic property of  $rv_e^{(1+\f{k}{3})}$, a direct calculation implies that
\be
\p_\th \lt( r(\Dl-\f{1}{r^2})u_e^{(1+\f{k}{3})} +\f{2}{r}\p_\th v_e^{(1+\f{k}{3})}\rt)=0.
\ee
Let
\bes
\t{u}^{(1+\f{k}{3})}_e={u}^{(1+\f{k}{3})}_e+h_{1+\f{k}{3}}(r),
\ees
where $h_{1+\f{k}{3}}(r)$ satisfies the following ODE
\begin{align}
&r(\Dl-\f{1}{r^2})h_{1+\f{k}{3}}(r)=-r(\Dl-\f{1}{r^2})u_e^{(1+\f{k}{3})} +\f{2}{r}\p_\th v_e^{(1+\f{k}{3})},\nn\\
&h_{1+\f{k}{3}}(r_0)=\hat{A}_{1+\f{k}{3}},\q h_{1+\f{k}{3}}(r_0+1)=A_{1+\f{k}{3}}.\nn
\end{align}
Obviously, $(\t{u}^{(1+\f{k}{3})}_e, v^{(1+\f{k}{3})}_e)$ is still a solution of the system \eqref{middleeulerk} and satisfies \eqref{eulerkesti}. Also we have that
  \bes
 r(\Dl-\f{1}{r^2})\t{u}_e^{(1+\f{k}{3})} +\f{2}{r}\p_\th v_e^{(1+\f{k}{3})}\equiv 0.
  \ees
  By letting $\t{u}_p^{(1+\f{k}{3})}$ and $\t{u}_e^{(1+\f{k}{3})}$ to replace  ${u}_p^{(1+\f{k}{3})}$ and $u_e^{(1+\f{k}{3})}$, the equations for the $1+\f{k}{3}$th order outer boundary layer \eqref{boundarylayeruk} is invariant, which are also true for the equations of the $1+\f{2+k}{3}$th order inner boundary layer \eqref{boundarylayerlk} when replacing $\hat{u}_p^{(1+\f{k}{3})}$ and $u_e^{(1+\f{k}{3})}$ with $\t{\hat{u}}_p^{(1+\f{k}{3})}$ and $\t{u}_e^{(1+\f{k}{3})}$. \ef

\subsection{Approximate solutions}\label{Approximate solutions}

\indent

In this subsection, we construct an approximate solution of the Navier-Stokes equations (\ref{nspolar}). First let $\chi(y)\in C^\i_c((-\i,r_0+1])$ be a function satisfying
\bes
\chi(r)=\lt\{
\bali
&0,\q r\in \lt[-\i,r_0+0.25\rt],\\
&1,\q r\geq r_0+0.75,
\eali
\rt.
\ees
and $\hat{\chi}(r)=1-\chi(r)$.  Set
\begin{align*}
{u}_{p,\text{comb}}^a(\th,r):=&\chi(r)\lt[u_p^{(0)}+\sum^{45}_{k=0}\e^{1+\f{k}{3}}u^{(1+\f{k}{3})}_p\rt]+\hat{\chi}\lt[\sum^{45}_{k=0}\e^{1+\f{k}{3}}\hat{u}^{(1+\f{k}{3})}_p \rt]:=\chi(r)u_p^a+\hat{\chi}(r) \hat{u}_p^a,\\[5pt]
{v}_{p,\text{comb}}^a(\th,r):=&\chi(r)\lt[\sum^{48}_{k=0}\e^{1+\f{k}{3}}v^{(1+\f{k}{3})}_p\rt]+\hat{\chi}(r)\lt[\sum_{k=0}^{47}\e^{1+\f{k}{3}}\hat{v}^{(1+\f{k}{3})}_p\rt]:=\chi(r)v_p^a+\hat{\chi}(r)\hat{v}_p^a,\\[5pt]
{p}_{p,\text{comb}}^a(x,y):=&\chi^2(r)\lt[\sum^{48}_{k=0}\e^{1+\f{k}{3}}p^{(1+\f{k}{3})}_p\rt]+\hat{\chi}^2(r)\lt[\sum_{k=0}^{47}\e^{1+\f{k}{3}}\hat{p}^{(1+\f{k}{3})}_p\rt]:=\chi^2p_p^a+(1-\chi(r))^2\hat{p}_p^a,
\end{align*}
and
\begin{align*}
u_e^a:=u^e(r)+\sum^{45}_{k=0}\e^{1+\f{k}{3}}u^{(1+\f{k}{3})}_e, \q v_e^a:=\sum^{45}_{k=0}\e^{1+\f{k}{3}}v^{(1+\f{k}{3})}_e,\q p_e^a:=\sum^{45}_{k=0}\e^{1+\f{k}{3}}p^{(1+\f{k}{3})}_e.
\end{align*}

We construct an approximate solution $(u^a,v^a,p^a)$ by
\begin{align*}
u^a(x,y):&=u_e^a+{u}_{p,\text{comb}}^a+\epsilon^{16}h(\th,r),\q v^a(x,y):=v_e^a+{v}_{p,\text{comb}}^a,\q p^a(x,y):=p_e^a+{p}_p^a,
\end{align*}
where the corrector $h(\th,r)$ will be given in Appendix, which satisfies
\begin{align*}
h(\th,r_0)=h(\th,r_0+1)=0, \ \|\partial_\th^i\partial_r^j h\|_2\leq C_{i,j}r_0,
\end{align*}
and makes $(u^a, v^a)$ be divergence-free
\begin{align*}
\p_\th u^a+\p_r(rv^a)=0.
\end{align*}
Moreover, $(u^a, v^a)$  satisfies the following boundary conditions
\begin{align*}
&( u^a,v^a)(\th,r)=( u^a,v^a)(\th+2\pi,r),\\[5pt]
&u^a(\th, r_0+1)=(\o+\dl f(\th))(r_0+1), \ v^a(\th,r_0+1)=0,\\[5pt]
&u^a(\th,r_0)=0,\ v^a(\th,r_0)=0.
\end{align*}
By collecting the estimates in Proposition \ref{propeulerorder1} and Proposition \ref{propeulerorderk},  we deduce that
\begin{align}
&\|\p^j_\th\p^{k}_r(u^a_e-u^e(r))\|_\infty\leq C_{j,k}\epsilon(\dl+\epsilon)r_0^{1-k},\nn \\
& \|\p^j_\th\p^{k}_r v^a_e\|_\infty\leq C_{j,k}\epsilon(\dl+\epsilon)r_0^{-k}.\nn
\end{align}
By collecting the estimates in Proposition \ref{propdcu0}, Proposition \ref{propdcu1} and Proposition \ref{propdcuk}, one has
\begin{align}
\|\zeta^\ell\p^{j}_\th\partial_\zeta^k u_p^a\|_\infty\leq C_{j,k,\ell} (\dl+\epsilon)r_0, \ \|\zeta^\ell\p^{j}_\th\partial_\zeta^k v_p^a\|_\infty\leq C_{j,k,\ell} \epsilon(\dl+\epsilon).\nn
\end{align}
Similarly, there holds
\begin{align*}
&\|\eta^\ell\p^{j}_\th\partial_\eta^k \hat{u}_p^a\|_\infty\leq C_{j,k,\ell} \e(\dl+\epsilon^{2/3})r_0, \q \|\eta^\ell\p^{j}_\th\partial_\eta^k \hat{u}_p^a\|_\infty\leq C_{j,k,\ell} \epsilon^{1+2/3}(\dl+\epsilon^{2/3}).
\end{align*}

Finally, set
 \begin{align*}
 R_u^a:&=u^{a} u_\theta^{a}+r v^{a} u_r^{a}+u^{a} v^{a}+p_\theta^{a}-\epsilon^2\left(\frac{u_{\theta \theta}^{a}}{r}+r u_{r r}^{a}+u_r^{a}+\frac{2}{r} v_\theta^{a}-\frac{u^{a}}{r}\right), \\
R_v^a:&=u^{a} v_\theta^{a}+r v^{a} v_r^{a}-\left(u^{a}\right)^2+r p_r^{a}-\epsilon^2\left(\frac{v_{\theta \theta}^{a}}{r}+r v_{r r}^{a}+v_r^{a}-\frac{2}{r} u_\theta^{a}-\frac{v^{a}}{r}\right),
 \end{align*}
 then there holds
 \begin{align*}
\|\partial^i_\th \p^j_r (R_u^a,R_v^a)\|_2\leq C\e^{16-j}r_0^2,
 \end{align*}
and $(u^a, v^a, p^a)$ satisfies
\begin{eqnarray*}
\left\{
\begin{array}{lll}
u^{a} u_\theta^{a}+r v^{a} u_r^{a}+u^{a} v^{a}+p_\theta^{a}-\epsilon^2\left(\frac{u_{\theta \theta}^{a}}{r}+r u_{r r}^{a}+u_r^{a}+\frac{2}{r} v_\theta^{a}-\frac{u^{a}}{r}\right)=R_u^a,\\[5pt]
u^{a} v_\theta^{a}+r v^{a} v_r^{a}-\left(u^{a}\right)^2+r p_r^{a}-\epsilon^2\left(\frac{v_{\theta \theta}^{a}}{r}+r v_{r r}^{a}+v_r^{a}-\frac{2}{r} u_\theta^{a}-\frac{v^{a}}{r}\right)=R_v^a,\\[5pt]
 \p_\th u^a+\p_r(rv^a)=0, \\[5pt]
( u^a,v^a)(\th,r)=( u^a,v^a)(\th+2\pi,r),\\[5pt]
u^a(\th, r_0+1)=(\o+\dl f(\th))(r_0+1), \ v^a(\th,r_0+1)=0,\\[5pt]
u^a(\th,r_0)=0,\ v^a(\th,r_0)=0.
\end{array}
\right.
\end{eqnarray*}
This is exactly the constructed approximate solution that we stated in the beginning of Section \ref{sec2}. \qed

\section{Appendix}

In this section, we give a construction of corrector $h(\th,r)$ defined in section \ref{Approximate solutions}. Firstly, we give a simple lemma which is similar to Lemma 6.1 in Appendix B in \cite{FeiGLT:2023CMP}, hence we omit the proof.

\begin{lemma}\label{corector equation}
Assume that $K(\th,r)$ is a $2\pi$-periodic smooth function which satisfies
\bes
\int_0^{2\pi}K(\th,r)d\th=0, \ \forall\ r\in [r_0,r_0+1]; \quad K(\th, r_0)=K(\th, r_0+1)=0,
\ees
then there exists a $2\pi$-periodic function $h(\th,r)$ such that
\begin{align}
&\partial_\th h(\th,r)=K(\th,r); \ h(\th,r_0)=h(\th, r_0+1)=0;\nonumber\\
&\int_0^{2\pi}h(\th,r)d\th=0, \ \|\partial_\th^j\partial_r^kh\|_2\leq C\|\partial_\th^j\partial_r^kK\|_2.\nn
\end{align}
\end{lemma}

Next, we construct the corrector $h(\th,r)$ by the above Lemma.
Direct computation gives
\begin{align*}
&\p_\th (u^a_e+ u^a_{p,\text{comb}})+\p_r(rv^a_e+rv^a_{p,\text{comb}})=\p_\th  u^a_{p,\text{comb}}+\p_r(rv^a_{p,\text{comb}})\\
=&\chi(r)(\p_\th {u}_p^a+\p_r(r{v}_p^a))+\hat{\chi}(\p_\th \hat{u}_p^a+\p_r(r\hat{v}_p^a))+r(\chi^{\prime}(r) {v}^a_p+\hat{\chi}'\hat{v}^a_p),\nn\\
=&\chi\sum^{48}_{k=46} \e^{1+\f{k}{3}}\p_\zeta(\zeta v_p^{(1+\f{k}{3})})+\hat{\chi}\sum^{47}_{k=46} \e^{1+\f{k}{3}}\p_\eta(\eta \hat{v}_p^{(1+\f{k}{3})})+r(\chi'(r){v}^a_p+\hat{\chi}'\hat{v}^a_p),\nn\\
=&\chi\sum^{48}_{k=46} \e^{1+\f{k}{3}}\p_\zeta(\zeta v_p^{(1+\f{k}{3})})+\hat{\chi}\sum^{47}_{k=46} \e^{1+\f{k}{3}}\p_\eta(\eta \hat{v}_p^{(1+\f{k}{3})})\nn\\
 &+r\e^{16}(\chi^{\prime}(r-r_0-1)^{-16}\zeta^{16}{v}^a_p+\hat{\chi}^{\prime}(r-r_0)^{-16}\eta^{16}\hat{v}^a_p)\nn\\
:=&-\e^{10} K(\th,r).
\end{align*}

Notice that $\chi'(r)$ and $\hat{\chi}'(r)$ are supported in $ r\in [r_0+0.25,r_0+0.75]$ and the properties of ${v}^a_p,  \hat{v}^a_p$, we know that $K(\th,r)$ satisfies the assumption in Lemma \ref{corector equation}, then there exists  $h(\th,r)$ such that
\beas
&\p_\th (u^a_e+ u^a_{p,\text{comb}})+\p_r(rv^a_e+rv^a_{p,\text{comb}})=-\e^{10}\p_\th h(\th,r),
\eeas
which indicates that
\bes
\p_\th u^a+\p_r(rv^a)=0.
\ees

\section*{Data availability statement}

Data sharing is not applicable to this article as no datasets were generated or analysed during the current study.

\section*{Conflict of interest statement}

The authors declare that they have no conflict of interest.

\section*{Acknowledgments}

 X. Pan is supported by NSF of China  under Grant  No.12471222 and No.12031006. We thank Dr. Chen Gao for helpful discussions on this topic.

\end{document}